\documentclass[10pt]{article}
\usepackage[all,2cell]{xy} \UseAllTwocells \SilentMatrices
\usepackage{latexsym,amsfonts,amssymb}
\usepackage{amsmath,amsthm,amscd}
\usepackage{hyperref,psfrag}
\usepackage{diagcat}
\usepackage{color}
\usepackage{etoolbox}
\usepackage[dvips]{epsfig}
\usepackage{psfrag}
\usepackage{youngtab}
\usepackage{graphicx}
\usepackage{a4wide}
\usepackage{geometry}

\usepackage{epigraph,wrapfig}

\normalfont\upshape

\usepackage{fancyhdr}
\theoremstyle{definition}
\newtheorem{thm}{Theorem}[section]
\newtheorem{cor}[thm]{Corollary}

\newtheorem{lem}[thm]{Lemma}
\newtheorem{rem}[thm]{Remark}
\newtheorem{prop}[thm]{Proposition}
\newtheorem{defn}[thm]{Definition}
\newtheorem{example}[thm]{Example}

\newtheorem*{thm*}{Theorem}
\numberwithin{equation}{section}

\usepackage{bbm}

\def\R{{\mathbbm R}}
\def\Z{{\mathbbm Z}}

\def\F{{\mathbbm F}}

\def\1{{\mathbbm{1}}}

\newcommand{\HOM}{{\rm HOM}}
\renewcommand{\to}{\rightarrow}

\def\dif{\partial}

\def\lra{{\longrightarrow}}
   
\def\dmod{{\mathrm{\mbox{-}mod}}}   %% finitely-generated modules
\def\Id{\mathrm{Id}}
\def\mc{\mathcal}
\def\mf{\mathfrak}
\def\shuffle{\,\raise 1pt\hbox{$\scriptscriptstyle\cup{\mskip
               -4mu}\cup$}\,}
\def\udmod{{\mathrm{-\underline{mod}}}}   %% finitely-generated modules

\newcommand{\refequal}[1]{\xy {\ar@{=}^{#1}
(-1,0)*{};(1,0)*{}};
\endxy}

\newcommand{\mH}{\mathrm{H}} %cohomology
\newcommand{\mHH}{\mathrm{HH}}
\newcommand{\mHHH}{\mathrm{HHH}}
\newcommand{\mtHHH}{\widehat{\mathrm{HHH}}}
\newcommand{\pH}{p\mathrm{H}}
\newcommand{\pHH}{p\mathrm{HH}}
\newcommand{\pHHH}{p\mathrm{HHH}}
\newcommand{\ptHHH}{p\widehat{\mathrm{HHH}}}
\newcommand{\pC}{pC}
\newcommand{\pT}{pT}

\newstrandstyle{Thick}{type=ribbon,style=solid,ribboncolor=black,color=black}

\title{On some $p$-differential graded link homologies}
\author{You Qi and Joshua Sussan}

\date{\today}

\begin{document}
%
% ==============================================================================

\maketitle

\begin{abstract}
We show that the triply graded Khovanov-Rozansky homology of knots and links over a field of positive odd characteristic $p$ descends to an invariant in the homotopy category finite-dimensional $p$-complexes.

A $p$-extended differential on the triply graded homology discovered by Cautis is compatible with the $p$-DG structure.  As a consequence we get a categorification of the Jones polynomial evaluated at a $2p$th root of unity.
\end{abstract}

\setcounter{tocdepth}{2} \tableofcontents

\section{Introduction}
\subsection{Background}
The Jones polynomial is a quantum invariant of oriented links which may be defined using the natural two-dimensional representation of quantum $\mathfrak{sl}_2$.  Coloring the components of a link with other representations of this quantum group leads to a definition of the colored Jones polynomial.  Witten constructed an invariant of $3$-dimensional manifolds in a physical setting coming from Chern-Simons theory \cite{Witten} with a fixed level.  If the $3$-manifold is defined as surgery on a link $L$, Reshetikhin and Turaev \cite{RT} reconstructed Witten's invariant by summing over colored Jones polynomials of $L$.  In order for this summation to be finite, it is important that the colored Jones polynomials are evaluated at a root of unity, the order of which is determined by the level of Witten's theory.  The Witten-Reshetikhin-Turaev (WRT)  $3$-manifold invariant fits into the framework of a $(2+1)$-dimensional topological quantum field theories (TQFT).

Crane and Frenkel \cite{CF} initiated the categorification program with the aim of lifting the 
$(2+1)$-d WRT-TQFT to a $(3+1)$-d TQFT.  The first major success in this program was Khovanov's 
categorification of the Jones polynomial \cite{KhJones}.  Khovanov homology is a bigraded homology 
theory of links whose graded Euler characteristic is the Jones polynomial.  Since this discovery, 
there have been many other categorifications of the Jones polynomial as well as their quantum
$\mathfrak{sl}_n$ generalizations.  One such construction was Khovanov and Rozansky's 
categorification of the $\mathfrak{sl}_n$ and HOMFLYPT polynomials using matrix factorizations \cite{KR1, KR2}.  The HOMFLYPT homology 
theory is triply graded and the graded Euler characteristic recovers the two-variable HOMFLYPT 
polynomial.  Khovanov later recast this construction in the language of Soergel bimodules \cite{KR3} 
building upon earlier work of Rouquier \cite{RouBraid} who gave a categorical construction of the 
braid group.  
This was later reproved by Rouquier in \cite{Roulink}.
In these constructions, one represents a link as the closure of a braid.  To the braid, one associates
a complex of Soergel bimodules.  Taking Hochschild homology of each term yields a complex of
bigraded vector spaces.  Taking homology of this complex results in a triply graded theory.

The categorification of quantum groups and their associated link invariants at generic values of the quantum parameter has been the focus of a lot of research since Crane and Frenkel's work.
The first approach towards categorically specializing the quantum parameter to a root of unity was due to Khovanov \cite{Hopforoots} and later expanded upon in \cite{QYHopf}.
In this setup, one should consider algebraic structures over a field of characteristic $p$ and search for a derivation $\partial$ such that $\partial^p=0$.  This enhances the algebraic structure to a module category over the tensor category of graded modules over a particular Hopf algebra $H=\Bbbk[\dif]/(\dif^p)$.
Taking an appropriate homotopy or derived category gives rise to an action of the stable category
$H\udmod$.  Khovanov \cite{Hopforoots} showed that the Grothendieck group of this stable category is isomorphic to the cyclotomic ring for the prime $p$, thus categorifying a structure at a prime root of unity.

The first successful implementation of this idea was the categorification of the upper half of the small quantum group for $\mathfrak{sl}_2$ in \cite{KQ} by endowing the nilHecke algebra with a $p$-DG structure.  It is an interesting open question how to import $p$-DG theory into the construction of Khovanov homology.  A clearer path towards categorifying link invariants at roots unity was described in \cite{KQ} where a $p$-differential was defined on Webster's algebras \cite{Webcombined}.  One step in this direction was a categorification of the Burau representation of the braid group at a prime root of unity which used a very special $p$-DG Webster algebra \cite{QiSussan}. We also refer the reader to \cite{QiSussan2} for a survey in this direction.

\subsection{Methodology}
In this paper we propose a construction of a $p$-DG version of HOMFLYPT homology and its categorical specializations.  We closely mimic the work of Khovanov and Rozansky in \cite{KRWitt} where an action of the Witt algebra on HOMFLYPT homology is constructed, and adapt their framework in the $p$-DG setting. In particular, the action of one of their Witt algebra generators (denoted $L_1$ in  \cite{KRWitt}) corresponds to the $p$-differential $\dif$ considered in this work.

Let $R=\Bbbk[x_1,\ldots,x_n]$ be the polynomial algebra generated by elements of degree two.
The category of regular Soergel bimodules for $\mathfrak{gl}_n$ is the idempotent completion of the subcategory of $(R,R)$-bimodules generated by the so-called Bott-Samelson bimodules.
The Hopf algebra $H$ acts on the polynomial algebra $R$ determined by $\dif(x_i)=x_i^2$.  By the Leibniz rule, $H$ also acts on tensor products of Bott-Samelson bimodules.  We may then form the category $(R,R) \# H$ of such $p$-DG bimodules.  We show that in an appropriate homotopy category there is a categorical braid group action, extending the result of Rouquier \cite{RouBraid}.  For the proof, we follow the exposition of the braid group action in \cite{KRWitt} very closely.  Using a certain $p$-extension functor, we then obtain a braid group action on a relative $p$-homotopy category.  It follows that, to any braid group element $\beta$, there is a $p$-chain complex of $H$-equivariant Soergel bimodules $pT_\beta$ associated to $\beta$ that is well defined up to $p$-homotopy equivalences.

We  next turn our attention to extracting link invariants by taking various versions of Hochschild homology. A $p$-analogue of the usual Hochschild homology $\pHH_\bullet$, which goes back to the work of Mayer \cite{Mayer1, Mayer2}, is utilized. In the $p$-extended setting, we need to collapse the Hochschild and topological gradings into a single grading because of the Markov II invariance constraint. Thus the construction yields just a doubly graded categorification of the HOMFLYPT polynomial where the $a$ variable is now specialized to a prime root of unity.

Let $\zeta_C:=\sum_{i=1}^n x_i^2\frac{\dif}{\dif x_i}\in \mHH^1(R)$ be a Hochschild cohomology element of degree two, regarded as a derivation on $R$.
The cap product of $\zeta_C$ with an element in Hochschild homology yields a differential
$d_C \colon \mHH_i(M) \rightarrow \mHH_{i-1}(M)$
of $q$-degree $2$ and $a$-degree $-1$.
This differential gives rise to, via $p$-extension, a $p$-differential $\dif_C$ action on the $p$-Hochschild homology groups of any $H$-equivariant Soergel bimodule.  For a braid $\beta$, one may form a total $p$-differential $\dif_T:=\dif_t+\dif_C+\dif_q$ combining the topological differential $\dif_t$ coming from the Rouquier complex with the derivation actions arising from $\dif_C$ and $H$. The total differential acts on
$\pHH_\bullet(pT_\beta)$ and gives rise to an invariant upon taking homology.

\begin{thm*}
Let $L$ be a link presented as the closure of a braid $\beta$.
The slash homology of  $\pHH_\bullet(pT_\beta)$ with respect to $\dif_T$ is a finite-dimensional framed link invariant whose Euler characteristic is the Jones polynomial evaluated at a prime root of unity.
\end{thm*}

The link invariant using the action of the usual differential $d_C$ on Hochschild homology (ignoring the action of $H$)
was first constructed by Cautis \cite{Cautisremarks}, and further considered in other contexts by Robert-Wagner \cite{RW} and Queffelec-Rose-Sartori \cite{QRS}. The latter authors showed that it categorifies the Jones polynomial for a generic value of the quantum parameter and
is distinct from Khovanov homology.  These works actually utilized a degree $2N$ differential and categorified the link invariant arising from quantum $\mathfrak{sl}_N$.  We restrict to the case $N=2$ due to the fact that $\dif_q$ and $\dif_C$ do not commute for arbitrary values of $N$.
One may view this work as a combination of the results of \cite{KRWitt} with \cite{Cautisremarks, QRS, RW}.  

It is a natural problem to extend our result to categorify the colored Jones and $\mf{sl}_n$
polynomials evaluated at a prime root of unity.  The first technical obstacle to overcome in that
setting, is the construction of Koszul resolutions of the algebra of symmetric functions in the
presence of a $p$-differential. We plan to explore these questions in follow-up works.

\subsection{Outline}
We now summarize the contents of each section.

In Section \ref{sechopf} we review some constructions known in $p$-DG theory and develop some new
ones such as the $p$-extension functor, the totalization functor, the relative $p$-homotopy category,
and (relative) $p$-Hochschild homology. 

A review of Soergel bimodules is given in Section \ref{secbraidrelations}, where a $p$-categorical
braid group action is constructed.  Many of the techniques in this section parallel methods used in
\cite{KRWitt}.

Section \ref{sechomflypt} contains the construction of the categorification of the HOMFLYPT
polynomial at a root of unity.  The main technical result in this section is invariance under the
second Markov move.  The proof builds upon the techniques in \cite{RW} which in turn used ideas from \cite{Roulink} adapted to the $H$-equivariant and hopfological setting.

A categorification of the Jones polynomial at a prime root of unity is developed in Section
\ref{sl2section}.  We revisit the proof of the second Markov move given in the previous section but
now accounting for the extra differential $\dif_C$.  Here, again, we build upon ideas from \cite{Cautisremarks, RW, QRS}.

We conclude in Section \ref{sl2section} with the calculation of the homology theories developed in this work for $(2,n)$ torus links.  In particular, we exhibit nontrivial $p$-complexes as $p$-homologies of these links.

\subsection{Acknowledgements.}
The authors would like to thank Sabin Cautis, Mikhail Khovanov, and Louis-Hadrien Robert for helpful conversations.

Y.Q. is partially supported by the NSF grant DMS-1947532. J.S. is partially supported by the NSF grant DMS-1807161 and PSC CUNY Award 63047-00 51.

\section{Hopfological constructions} \label{sechopf}
 In this section, we recall some basic hopfological algebraic facts introduced in \cite{Hopforoots, QYHopf}. We also develop the necessary constructions of $p$-analogues of  classical Hochschild homology in the hopfological setting.

There will be several (super) differentials utilized in this section. We reserved the normal $d$ for the super differential ($d^2=0$), and the symbol $\dif$ to denote a $p$-differential ($\dif^p=0$) over a field of finite characteristic $p>0$. Various differentials will also be labeled with different subscripts to indicate their different meanings.

\subsection{Some  exact functors}
Let $A$ be an algebra over the ground field $\Bbbk$ of characteristic $p>0$. We equip $A$ with the trivial ($p$-)differential graded structure by declaring that $d_0\equiv 0$, $\dif_0\equiv 0$ and $A$ sits in degree zero. In this subsection, we study a functor relating the usual homotopy category $\mc{C}(A,d_0)$ of $A$ with its $p$-DG homotopy category $\mc{C}(A,\dif_0)$.

To do this, recall that a chain complex of $A$-modules consists of a collection of $A$-modules and homomorphisms $d_M:M_i\lra M_{i-1}$ called boundary maps
\[
\xymatrix{
\cdots \ar[r]^-{d_{M}}  & M_{i+1} \ar[r]^-{d_{M}} & M_{i} \ar[r]^-{d_{M}} & M_{i-1} \ar[r]^-{d_{M}} & M_{i-2} \ar[r]^-{d_{M}} & \cdots
} ,
\]
satisfying $d_{M}^2=0$ for all $i\in \Z$. A \emph{null-homotopic map} is a sequence of $A$-module maps $f_i:M_i\lra N_i$, $i\in \Z$, of $A$-modules, as depicted in the diagram below,
\begin{equation*}
 \xymatrix{
 \cdots\ar[r]^{d_{M}} & M_{i+1} \ar[dl]|-{h_{i+1}} \ar[r]^{d_{M}} \ar[d]|-{f_{i+1}} & M_{i} \ar[dl]|-{h_{i}} \ar[r]^{d_{M}} \ar[d]|-{f_{i}} & M_{i-1} \ar[dl]|-{h_{i-1}} \ar[r]^{d_{M}} \ar[d]|-{f_{i-1}} & M_{i-2}\ar[r]^{d_{M}} \ar[dl]|-{h_{i-2}} \ar[d]|-{f_{i-2}} & \cdots \ar[dl]|-{h_{i-3}}\\
 \cdots \ar[r]_{d_{N}} & N_{i+1} \ar[r]_{d_{N}} & N_{i} \ar[r]_{d_{N}} & N_{i-1} \ar[r]_{d_{N}} & N_{i-2}\ar[r]_{d_{N}} & \cdots
 }   
\end{equation*}
which satisfy $f_{i}=d_{N}\circ h_i+h_{i-1}\circ d_M$ for all $i\in \Z$. The homotopy category $\mc{C}(A,d_0)$, by construction, is the quotient of the category of chain complexes over $A$ by the ideal of null-homotopic morphisms. 

For ease of notation, we will use bullet points to stand for a general ($p$)-chain complex index in what follows.
Similarly, a $p$-chain complex of $A$-modules consists of a collection of $A$-modules and homomorphisms $\dif_M :M_i \lra M_{i-1}$ called $p$-boundary maps
\[
\xymatrix{
\cdots \ar[r]^-{\dif_{M}}  & M_{i+1} \ar[r]^-{\dif_{M}} & M_{i} \ar[r]^-{\dif_{M}} & M_{i-1} \ar[r]^-{\dif_{M}} & M_{i-2} \ar[r]^-{\dif_{M}} & \cdots
} ,
\]
satisfying $\dif_M^p \equiv 0$. A $p$-chain complex can also be regarded as a graded module over the tensor product algebra $A\otimes H_0$, where $H_0=\Bbbk[\dif_0]/(\dif_0^p)$ is a graded Hopf algebra where $\mathrm{deg}(\dif_0)=-1$ and $\mathrm{deg}(A)=0$.

We introduce some special notation for some specific indecomposable $p$-chain complexes over $\Bbbk$, by setting 
\begin{equation}
    U_i:=H_0/(\dif_0^{i+1}) ,\quad \quad 0\leq i \leq p-1.
\end{equation}
In particular, $U_i$ has dimension $i+1$. We will also use these modules with degree shifted up by an integer $n\in \Z$, which we denote by $U_i\{a\}$. Then $U_i\{a\}$ is concentrated in degrees $a, a-1,\dots, a-i$: 
\[
\xymatrix{
\overset{a}{\Bbbk} \ar@{=}[r] & \overset{a-1}{\Bbbk} \ar@{=}[r] & \cdots \ar@{=}[r] & \overset{a-i}{\Bbbk} 
}
\]

A map of $p$-complexes $f \colon M_\bullet \lra N_\bullet$ of $A$-modules is said to be \emph{null-homotopic} if there exists 
$$h: M_\bullet\lra N_{\bullet+p-1}$$
such that
\begin{equation}\label{eqn-p-null-homotopy}
    f = \sum_{i=0}^{p-1} \dif^{i}_N \circ h \circ \dif^{p-1-i}_M.
\end{equation}

The $p$-homotopy category, $\mc{C}(A,\dif_0)$, is then the quotient of $p$-chain complexes of $A$-modules by the ideal of null-homotopic morphisms. It is a triangulated category, whose homological shift functor $[1]_\dif$ is defined by
\begin{equation}\label{eqn-p-shift}
    M[1]_\dif:=M\otimes U_{p-2}\{p-1\}
\end{equation}
for any $p$-complex of $A$-modules. The inverse functor $[-1]_\dif$ is given by
\begin{equation}\label{eqn-p-shift-inverse}
    M[-1]_\dif:=M\otimes U_{p-2}\{-1\},
\end{equation}
which is a consequence of the fact that $U_{p-2}\otimes U_{p-2}$ decomposes into a direct sum of $\Bbbk\{2-p\}$ and copies of free $H_0$-modules.

\paragraph{Slash homology.} As an analogue of the usual homology functor, we have the notion of \emph{slash homology groups} \cite{KQ} of a $p$-complex. To recall its definition, let us set $A=\Bbbk$. For each $0\le k \le p-2$ form
the graded vector space
$$ \mH^{/k}(M) = \dfrac{\mathrm{Ker}(\dif^{k+1}_M)}{\mathrm{Im} (\dif^{p-k-1}_M)+\mathrm{Ker} (\dif^{k}_M)}.$$
The original $\Z$-grading on $M$ gives a decomposition
$$ \mH_{\bullet}^{/k}(M) = \bigoplus_{i\in \Z} \mH_{ i }^{/k}(M) .$$
The differential $\dif_M$ induces a map,
also denoted $\dif_M$, which takes  $\mH_{i}^{/ k}(U)$ to $\mH_{
i-1}^{/k-1}(U)$. Define the \emph{slash homology} of $M$ as
\begin{equation}
 \mH_{\bullet }^{/}(M) =  \bigoplus_{k=0}^{p-2} \mH_{\bullet }^{/k}(M).
\end{equation}
Also let
$$\mH_{i}^/(M) :=
\bigoplus_{k=0}^{p-2}\mH^{/k}_{i}(M).$$
We have the decompositions
\begin{equation}
\mH_{\bullet}^/(M) = \bigoplus_{i\in \Z} \mH_{i}^/(M)  = \bigoplus_{k=0}^{p-2}
\mH_{i}^{/k}(M)=\bigoplus_{i\in \Z}\bigoplus_{k=0}^{p-2} \mH^{/k}_{ i}(M).
\end{equation}
$\mH^/_{\bullet}(M)$ is a bigraded $\Bbbk$-vector space, equipped with an operator
$\dif_M$ of bidegree $(-1,-1)$, $\dif_M:\mH^{/k}_{i}\lra \mH^{/k-1}_{i-1}$.

Forgetting the $k$-grading gives us a graded vector space
 $\mH_{\bullet}^/(M)$ with differential $\dif_M$, which we can view as a graded $H_0$-module.
$\mH^/_{\bullet}(M)$ is isomorphic to $M$ in the homotopy category of $p$-complexes $\mc{C}(\Bbbk,\dif_0)$, and
we can decompose
\begin{equation}\label{eqn-decomp-of-M-into-cohomology-and-free}
M\cong \mH^/_{\bullet}(M)\oplus P(M)
\end{equation}
in the abelian category of $H$-modules,
where $P(M)$ is a maximal projective direct summand of $M$. In particular, we have
\begin{equation}
    \mH^/_\bullet(U_i) = 
    \begin{cases}
    U_i , & i=0,\dots, p-2, \\
    0 & i=p-1
    \end{cases}
\end{equation}
The slash homology group $\mH^/_{\bullet}(M)$, viewed
as an $H_0$-module, does not contain any direct summand isomorphic to a free $H_0$-module. 

The assignment $M\mapsto \mH^{/}_\bullet(M)$ is functorial in $M$ and can be viewed as a functor
$H_0\dmod \lra \mc{C}(\Bbbk,\dif_0)$ or as a functor $\mc{C}(\Bbbk,\dif_0) \lra \mc{C}(\Bbbk,\dif_0)$. The latter functor is then isomorphic to the identity functor.

As in the usual homological algebra case, we say a morphism $f:M\lra N$ of $p$-complexes of $A$-modules is a \emph{quasi-isomorphism} if, upon taking slash homology, $f$ induces an isomorphism $f^/: \mH^/_\bullet(M) \cong \mH^/_\bullet(N)$. The class of quasi-isomorphisms constitutes a localizing class in $\mc{C}(A,\dif_0)$ (\cite[Proposition 4]{Hopforoots}).

\begin{defn}\label{def-p-derived-category-homology-version}
The \emph{$p$-derived category} $\mc{D}(A,\dif_0)$ is the localization of $\mc{C}(A,\dif_0)$ at the class of quasi-isomorphisms.
\end{defn}

Alternatively, $\mc{D}(A,\dif_0)$ is the \emph{Verdier quotient} of $\mc{C}(A,\dif_0)$ by the class of \emph{acyclic $p$-complexes}, i.e., those $p$-complexes of $A$-modules annihilated by the slash-homology functor.

\paragraph{$p$-Extension.} We now define the \emph{$p$-extension functor}
\begin{equation}\label{eqn-p-extension}
   \mc{P}: \mc{C}(A,d_0)\lra \mc{C}(A ,\dif_0) 
\end{equation}
as follows. Given a chain complex of $A$-modules, we repeat every term sitting in odd homological degrees $(p-1)$ times. More explicitly, for a given complex
\[
\xymatrix{
\cdots \ar[r]^-{d_{2k+2}}  & M_{2k+1} \ar[r]^-{d_{2k+1}} & M_{2k} \ar[r]^-{d_{2k}} & M_{2k-1} \ar[r]^-{d_{2k-1}} & M_{2k-2} \ar[r]^-{d_{2k-2}} & \cdots
} ,
\]
the $p$-extended complex looks like
\[
\xymatrix{\cdots \ar[r]^-{d_{2k+2}} & M_{2k+1} \ar@{=}[r]& \cdots \ar@{=}[r] & M_{2k+1}\ar[r]^-{d_{2k+1}} \ar[r] &
M_{2k}\ar `r[rd] `_l `[llld] _-{d_{2k}} `[d] [lld]
& \\
& & M_{2k-1}\ar@{=}[r]&\cdots \ar@{=}[r]& M_{2k-1} \ar[r]^-{d_{2k-1}}& M_{2k-2} \ar[r]^{d_{2k-2}}& \cdots}
\ .
\]
Likewise, for a chain map
\begin{equation*}
 \xymatrix{
 \cdots\ar[r]^-{d_{2k+3}} & M_{2k+2}  \ar[r]^{d_{2k+2}} \ar[d]|-{f_{2k+2}} & M_{2k+1}  \ar[r]^{d_{2k+1}} \ar[d]|-{f_{2k+1}} & M_{2k}  \ar[r]^{d_{2k}} \ar[d]|-{f_{2k}} &  \cdots\\
 \cdots \ar[r]_-{d_{2k+3}} & N_{2k+2} \ar[r]_{d_{2k+2}} & N_{2k+1} \ar[r]_{d_{2k+1}} & N_{2k} \ar[r]_{d_{2k}}  & \cdots
 }   
\end{equation*}
the obtained morphism of $p$-complexes of $A$-modules is given by
\[
\xymatrix{
 \cdots\ar[r]^-{d_{2k+3}} & M_{2k+2}  \ar[r]^{d_{2k+2}} \ar[d]|-{f_{2k+2}} & M_{2k+1}  \ar[d]|-{f_{2k+1}} \ar@{=}[r] & \cdots \ar@{=}[r] &  M_{2k+1}   \ar[r]^{d_{2k+1}} \ar[d]|-{f_{2k+1}} & M_{2k}  \ar[r]^{d_{2k}} \ar[d]|-{f_{2k}} &  \cdots\\
 \cdots \ar[r]_-{d_{2k+3}} & N_{2k+2} \ar[r]_{d_{2k+2}} & N_{2k+1} \ar@{=}[r] & \cdots \ar@{=}[r] & N_{2k+1}  \ar[r]_{d_{2k+1}} & N_{2k} \ar[r]_{d_{2k}}  & \cdots
 }   
\]
This is clearly a functor from the abelian category of chain complexes over $A$ into the category of $p$-DG modules over $(A,\dif_0)$, ($p$-complexes of $A$-modules). Denote this functor by $\widehat{\mc{P}}$.

\begin{lem}
The functor $\widehat{\mc{P}}$ preserves the ideal of null-homotopic morphisms.
\end{lem}
\begin{proof}
It suffices to show that $\widehat{\mc{P}}$ sends null-homotopic morphisms in $\mc{C}(A,d_0)$ to null-homotopic morphisms in $\mc{C}(A,\dif_0)$. Suppose $f=dh+hd$ is a null-homotopic morphism in $\mc{C}(A,d_0)$. We first extend $h: M_\bullet\lra N_{\bullet+1}$ to a map 
$$\hat{\mc{P}}(h): \hat{\mc{P}}(M)_\bullet\lra \hat{\mc{P}}(N)_{\bullet+p-1}.$$
On unrepeated terms, $\hat{\mc{P}}(h)$ sends $M_{2k}$ to the copy of $N_{2k+1}$ sitting as the leftmost term in the repeated $N_{2k+1}$'s, while on the repeated terms, it only sends the rightmost $M_{2k+1}$ to the unrepeated $N_{2k+2}$ and acts by zero on the other repeated $M_{2k+1}$'s. Schematically, this has the effect as in the diagram below:
\[
\xymatrix{
 \cdots\ar@{=}[r] & M_{2k+3} \ar[r]^-{d_{2k+3}} & M_{2k+2}  \ar[r]^{d_{2k+2}} & M_{2k+1}  \ar[dlll]|-{0}  \ar@{=}[r] & \cdots \ar@{=}[r] &  M_{2k+1}   \ar[r]^{d_{2k+1}} \ar[dlll]|-{h_{2k+1}} & M_{2k} \ar[r]^{d_{2k}} \ar[dlll]|-{h_{2k}} &  M_{2k-1} \ar[dlll]|-{0} \ar@{=}[r] &\cdots\\
\cdots \ar@{=}[r]&  N_{2k+3} \ar[r]_-{d_{2k+3}} & N_{2k+2} \ar[r]_{d_{2k+2}} & N_{2k+1} \ar@{=}[r] & \cdots \ar@{=}[r] & N_{2k+1}  \ar[r]_{d_{2k+1}} & N_{2k} \ar[r]_{d_{2k}}  & N_{2k-1} \ar@{=}[r]&\cdots
 }   
 \ .
\]

Now it is an easy exercise to check that 
\begin{equation}\label{eqn-extended-homotopy}
\widehat{\mc{P}}(f)=\sum_{i=0}^{p-1} \dif^{p-1-i}_M \circ \widehat{\mc{P}}(h) \circ \dif^i_N,
\end{equation}
where $\dif_{M}$ denotes the extended $p$-differential on $\widehat{\mc{P}}(M)$, and similarly for $\dif_N$. For instance, between the rightmost repeated $M_{2k+1}$ and $N_{2k+1}$, the left-hand side of equation \eqref{eqn-extended-homotopy} equals $f_{2k+1}$, while there are only two nonzero terms contributing to the right-hand side of \eqref{eqn-extended-homotopy}, which are equal to, respectively, 
\[
\dif_N^{p-1}\circ \widehat{\mc{P}}(h)=d_{2k+2}\circ h_{2k+1}, \quad \quad \dif_N^{p-2}\circ \widehat{\mc{P}}(h)\circ \dif_M=h_{2k}\circ d_{2k+1}.
\]
The sum of these two nonzero terms is precisely $f_{2k+1}$ by the null-homotopy assumption on $h$. One similarly checks for the other repeated and unrepeated terms, and the lemma follows.
\end{proof}

This lemma implies that $\widehat{\mc{P}}$ descends to a functor
\begin{equation}\label{equation-functor-P}
    \mc{P}: \mc{C}(A,d_0)\lra \mc{C}(A,\dif_0).
\end{equation}
which we call the \emph{$p$-extension functor}.

\begin{prop}
The $p$-extension functor $\mc{P}$ is exact.
\end{prop}
\begin{proof}
It suffices to show that $\mc{P}$ commutes with homological shifts in both categories and sends distinguished triangles to distinguished triangles.

On $\mc{C}(A,d_0)$, the homological shift $[1]_d$ moves every term of a complex one step to the left, while the homological shift $[1]_\dif$ on $\mc{C}(A,\dif_0)$ is given by tensoring a $p$-complex of $A$-modules with the $(p-1)$-dimensional complex
\[
U_{p-2}\{p-1\}=
\left(
\xymatrix{
\underline{\Bbbk} \ar@{=}[r] & \cdots \ar@{=}[r] & { \Bbbk }
}
\right)
\]
where the underlined $\Bbbk$ sits in degree $p-1$. Note that the collection of repeated terms can be identified with
\[
\left(
\xymatrix{
\underline{M_{2k-1}}\ar@{=}[r] & \cdots \ar@{=}[r] & M_{2k-1},
}
\right) \cong M_{2k-1}\otimes U_{p-2}\{kp-1\},
\]
where the underlined term sits in degree $kp-1$.
Using the fact that 
$$U_{p-2}\{p-1\}\otimes U_{p-2} \{kp-1\}\cong U_0\{kp\} \oplus F $$ 
where $F$ is a direct sum of graded free $H_0$-modules, we see that
\[
U_{p-2}\{p-1\}\otimes (M_{2k-1}\otimes U_{p-2}\{kp-1\}) \cong M_{2k-1}\{kp\}
\]
in the homotopy category $\mc{C}(A,\dif_0)$. From this it follows that $U_{p-2}\{p-1\}\otimes \mc{P}(M)$ is homotopy equivalent to the $p$-complex $\mc{P}(M[1]_d)$. Thus $\mc{P}$ commutes with homological shifts.

To show that $\mc{P}$ sends distinguished triangles in $\mc{C}(A,d_0)$ to those in $\mc{C}(A,\dif_0)$, we use the characterization of distinguished triangles in $\mc{C}(A,d_0)$.
Recall that a distinguished triangle $P \rightarrow Q \rightarrow R$ in $\mc{C}(A,d_0)$ is a short exact sequence $0 \rightarrow P \rightarrow Q \rightarrow R \rightarrow 0$ of complexes of $A$-modules which split when ignoring the differentials.  After applying the $p$-extension functor $\mc{P}$ to $0 \rightarrow P \rightarrow Q \rightarrow R \rightarrow 0$,
one gets a short exact sequences of $p$-complexes which splits when ignoring the $p$-differentials.  This is precisely the condition that 
$\mc{P}(P) \rightarrow \mc{P}(Q) \rightarrow \mc{P}(R)$ is a distinguished triangle in $\mc{C}(A,\dif_0)$ (see \cite[Lemma 4.3]{QYHopf}).
\end{proof}

\paragraph{Totalization.}
Another useful functor is the \emph{totalization functor} $\mc{T}$, which we introduce next. To do so, we will need the following result.

\begin{lem}\label{lemma-acylicity-commutator-relation}
Let $(K_\bullet,\dif_K)$ be a $p$-complex of modules over $A$. Then $K_\bullet$ is null-homotopic if and only if there exists an $A$-module map $\sigma: K_{\bullet}\lra K_{\bullet +1}$ such that $\dif_K\sigma-\sigma\dif_K=\Id_K$.
\end{lem}
\begin{proof}
By definition (see equation \eqref{eqn-p-null-homotopy}), a $p$-complex is null-homotopic if and only if there is an $A$-linear map $h:K_{\bullet}\lra K_{\bullet+p-1}$ such that
$$\Id_K=\sum_{i=0}^{p-1} \dif_K^{p-1-i} \circ h \circ \dif_K^i.$$
For any linear map $\phi$ on $K_\bullet$, let $\mathrm{ad}_{\dif}(\phi):=[\dif_K,\phi]$. Now if $\sigma$ is as given satisfies $[\dif_K,\sigma]=\Id_K$, then the map
$h:=-\sigma^{p-1}$ satisfies
$[\dif_K,h]=-(p-1)\sigma^{p-2}$, and, inductively,
$$\mathrm{ad}_{\dif}^r(h)=-(p-1)\dots (p-r)\sigma^{p-r-1}.$$ 
In particular, we have 
$$
\sum_{i=0}^{p-1} (-1)^i \binom{p-1}{i} \dif_K^{p-1-i} \circ h \circ \dif_K^i=\mathrm{ad}_{\dif}^{p-1}(h)=-(p-1)!\Id_K=\Id_K.
$$
Since $\binom{p-1}{i}=(-1)^i$ in characteristic $p$, the null-homotopy formula follows.

Conversely, if $h$ is an $A$-linear null-homotopy map, then
$
\sigma:=\mathrm{ad}_{\dif}^{p-2}(h)
$
satisfies $[\dif_K, \sigma]=\Id_K$. As the iterated commutator of $A$-linear maps, $\sigma$ is also $A$-linear. The lemma follows.
\end{proof}

 Let $(A,\dif_A)$ be a $p$-DG algebra where $\dif_A$ has degree two\footnote{This degree convention is to match the usual
representation theoretical convention. See, for instance, \cite{KQ}. }. We regard $A$ as an algebra object in the graded
module category of the graded Hopf algebra $H_q=\Bbbk[\dif_q]/(\dif_q^p)$ in which $\mathrm{deg}(\dif_q)=2$. The
\emph{smash product algebra} $A\# H_q$ is then the graded algebra $A\otimes H_q$ containing the subalgebras $A\otimes 1$
and $1\otimes H_q$ and subject to the commutation relations
\[
(1\otimes \dif_q)(a\otimes 1)=a\otimes \dif_q + \dif_A(a)\otimes 1
\]
for any $a\in A$. Graded modules over $A\# H_q$ are also called \emph{$p$-DG modules over $A$}, the collection of which will be denoted $(A,\dif_A)\dmod$.

In analogy with the Hopf algebra $H_0$, we introduce the indecomposable \emph{balanced} $H_q$-complexes
\begin{equation}\label{eqn-Vi}
    V_i:=
    \left(
    \xymatrix{
\overset{-i}{\Bbbk} \ar@{=}[r] & \overset{-i+2}{\Bbbk} \ar@{=}[r] & \cdots \ar@{=}[r] & \overset{i-2}{\Bbbk} \ar@{=}[r] & \overset{i}{\Bbbk}
}
    \right)
\end{equation}
for $i=0, \dots, p-1$. 

As a matter of notation, we will denote the $q$-grading shifted copy of $V_i$ by $q^a V_i$, where the lowest degree term sits in degree $a-i$. Furthermore, if $M$ is any $p$-DG module over $A$, we will denote by $q^a M$ the $p$-DG module whose underlying module is the same as $M$, but grading shifted up by $a\in \Z$.

The \emph{$p$-DG homotopy category} $\mc{C}(A,\dif_A)$ can be defined, similarly as for $\mc{C}(A,\dif_0)$ before, by taking the quotient of the abelian category of $p$-DG modules by the ideal of \emph{null-homotopic morphisms}, which consists of homogeneous $A\# H_q$-module maps $f:(M,\dif_M)\lra (N,\dif_N)$ of the form
\begin{equation}
    f = \sum_{i=0}^{p-1} \dif_N^i \circ h \circ \dif_M^{p-1-i},
\end{equation}
where $h$ is an $A$-linear homomorphism from $(M,\dif_M)$ to $(N,\dif_N)$ of degree $2-2p$.

A $p$-DG homomorphism $f:(M,\dif_M)\lra (N,\dif_N)$ is called a \emph{quasi-isomorphism} if, again, $f$ induces an isomorphism of slash homology with respect to the $p$-differentials on $M$ and $N$. Inverting quasi-isomorphisms in $\mc{C}(A,\dif_A)$ results in the \emph{$p$-DG derived category} $\mc{D}(A,\dif_A)$.

Equip $A\# H_q$ with the zero $p$-differential $\dif_0$, and $\dif_0$ carries an additional $\Z$-grading that is independent of the original grading on $A$ and $H_q$. A $p$-complex of graded $A\# H_q$-modules thus has a $\Z\times \Z$-grading, where $\dif_q$ has degree $(2,0)$ and $\dif_0$ has degree $(0,-1)$. Consider the functor
\begin{equation}
    \widehat{\mc{T}}: (A\# H_q,\dif_0)\dmod \lra (A,\dif_A)\dmod
\end{equation}
defined as follows. To a $p$-complex $M$ of $A\# H_q$-modules
\[
\xymatrix{
\cdots \ar[r]^-{\dif_{M}}  & M_{i+1} \ar[r]^-{\dif_{M}} & M_{i} \ar[r]^-{\dif_{M}} & M_{i-1} \ar[r]^-{\dif_{M}} & M_{i-2} \ar[r]^-{\dif_{M}} & \cdots
} ,
\]
where each term $M_i$ is a graded $A$-module together with an internal $p$-differential $\dif_i$, compatible with $\dif_A$,
we assign to it the singly graded $A$-module $\oplus_{i\in \Z}q^{-2i} M_i$ whose new $p$-DG structure is given by
\[
\dif_T(m):=\dif_M(m)+\dif_i(m)\in q^{-2i+2}M_{i-1}\oplus q^{-2i}M_i
\]
if $m\in q^{-2i}M_i$.

\begin{lem}\label{lem-totalization-functor}
The functor $\widehat{\mc{T}}$ descends to a triangulated functor on the $p$-DG homotopy categories:
\begin{equation*}
    \mc{T}: \mc{C}(A\# H_q, \dif_0) \lra \mc{C}(A,\dif_A)
    \ .
\end{equation*}
\end{lem}
\begin{proof}
 If $Q_\bullet$ is a null-homotopic $p$-complex of $A\# H_q$-modules, then, by Lemma \ref{lemma-acylicity-commutator-relation}, there exist an $A\# H_q$-linear $\sigma: Q_\bullet \lra Q_{\bullet + 1}$ such that
\[
[\dif_Q, \sigma]= \Id_{Q}.
\]
Since $h$ commutes with the $H_q$-actions, we have
\[
[\dif_Q+\dif_q, \sigma]=\Id_{Q}.
\]
It follows that $\mc{T}(Q_\bullet)$ is null-homotopic, and the functor $\mc{T}$ is well defined on the $p$-homotopy categories. It is then an easy exercise to verify that $\mc{T}$ preserves the triangulated structures on both sides.
\end{proof}

\subsection{Grothendieck rings}
We will be considering the Grothendieck rings of the homotopy categories $\mc{C}(\Bbbk,\dif_0)$ and $\mc{C}(\Bbbk,\dif_q)$.

\begin{lem}
The Grothendieck rings of the tensor triangulated categories $\mc{C}(\Bbbk,\dif_0)$ and $\mc{C}(\Bbbk,\dif_q)$ are respectively isomorphic to
\begin{align*}
    K_0(\mc{C}(\Bbbk,\dif_0) \cong \Z[a,a^{-1}]/(1+a+\cdots+a^{p-1}).\\
    K_0(\mc{C}(\Bbbk,\dif_q) \cong \Z[q,q^{-1}]/(1+q^2+\cdots+q^{2(p-1)}).
\end{align*}
\end{lem}
\begin{proof}
 See \cite{Hopforoots,KQ} for the proof and motivation of introducing these rings, especially 
 the second one\footnote{Again, we emphasize that setting $\dif_q$ to be of degree two is to respect the usual
   convention in previous literature on Soergel bimodules, where polynomial generators are
   evenly graded. Alternatively, one may adapt the polynomial generators for Soergel bimodules
   in this paper to be of degree one, and both Grothendieck rings above are equal to the usual
   cyclotomic ring at a primitive $p$th root of unity.}.
\end{proof}

We will often abbreviate the Grothendieck rings by
\begin{subequations}
\begin{align}
    \mathcal{O}_p:=  K_0(\mc{C}(\Bbbk,\dif_0)) \cong \Z[a,a^{-1}]/(1+a+\cdots+a^{p-1}),\\
    \mathbb{O}_p:=K_0(\mc{C}(\Bbbk,\dif_q)) \cong \Z[q,q^{-1}]/(1+q^2+\cdots+q^{2(p-1)}).
\end{align}
\end{subequations}

\begin{rem}[Grading shift functors]
In what follows, we will freely use the notation $a^i(\mbox{-})$ and $q^{i}(\mbox{-})$, $i\in \Z$, to indicate the grading shift functors on $\mc{C}(\Bbbk,\dif_0)$ and $\mc{C}(\Bbbk,\dif_q)$. The functors then descend to multiplication by the corresponding monomials in the Grothendieck rings.
\end{rem}

In this paper, we will be working with (finite-dimensional) $a$ and $q$ bigraded complexes over $\Bbbk$ equipped with commuting differentials 
$\dif_0$ and $\dif_q$. On this category, one may consider the composition of slash-homology functors, first in the
$a$-direction and then in the $q$-direction:
\begin{equation}\label{eqn-composition-slash-homology}
    H_q\otimes H_0\dmod \xrightarrow{\mH^{/}_\bullet~\textrm{in $a$-direction}} \mc{C}(H_q,\dif_0) \xrightarrow{\mH^{/}_\bullet~\textrm{in $q$-direction}} \mc{C}(\mc{C}(\Bbbk,\dif_q),\dif_0).
\end{equation}
Here the last category stands for the homotopy category with object in $\mc{C}(\Bbbk,\dif_q)$.
As usual with taking homology of the usual bicomplexes, these functors do 
not commute, and their order matters in the construction.

\begin{cor}\label{cor-K0-bicomplexes}
The categories $\mc{C}(H_q,\dif_0)$ and $\mc{C}(\mc{C}(\Bbbk,\dif_q),\dif_0)$ have Grothendieck rings isomorphic to
\[
K_0(\mc{C}(H_q,\dif_0)) \cong \mathcal{O}_p[q,q^{-1}],
\quad \quad
K_0(\mc{C}(\mc{C}(\Bbbk,\dif_q),\dif_0)) \cong \mathcal{O}_p\otimes_\Z \mathbb{O}_p. 
\]
\end{cor}
\begin{proof}
  The bigraded abelian category $H_q\otimes H_0\dmod$ has its Grothendieck ring isomorphic to $\Z[a^{\pm 1},q^{\pm 1}]$. An object lying in the first slash homology functor has Euler characteristic in the ideal 
  $$(1+a+\dots+ a^{p-1})\subset \Z[a^{\pm 1},q^{\pm 1}].$$ 
  Similarly, a module lying inside the kernel of the composition functor has Euler characteristic in the ideal
  $$(1+a+\dots+ a^{p-1}, 1+q^2+\cdots q^{2(p-1)})\subset \Z[a^{\pm 1},q^{\pm 1}].$$ 
  The result follows.
\end{proof}

\subsection{\texorpdfstring{$p$}{p}-Hochschild homology and cohomology}
Now we come to the construction of the $p$-DG simplicial bar complex of Mayer \cite{Mayer1, Mayer2} (see also \cite{KWa}). The usual simplicial bar complex of a unital, associative algebra $A$ is the complex:
\begin{subequations}\label{eqn-bar-complex}
\begin{equation}
\cdots \xrightarrow{d_{n+1}} A^{\otimes(n+2)} \stackrel{d_n}{\lra} A^{\otimes (n+1)}\xrightarrow{d_{n-1}} \cdots \stackrel{d_2}{\lra} A^{\otimes 3}\stackrel{d_1}{\lra} A^{\otimes 2} \lra 0, 
\end{equation}
where
\begin{equation}
     d_i(a_0\otimes a_1\otimes\cdots \otimes a_{i+1})=\sum_{k=0}^{i+1} (-1)^i a_0\otimes \cdots a_{k-1}\otimes a_ka_{k+1}\otimes a_{k+2} \otimes \cdots \otimes a_{i+1}
 \end{equation}
 \end{subequations}
The bar complex is a free bimodule resolution of $A$, as the augmented complex
\begin{equation}\label{eqn-bar-complex-aug}
\cdots \xrightarrow{d_{n+1}} A^{\otimes(n+2)} \stackrel{d_n}{\lra} A^{\otimes (n+1)}\xrightarrow{d_{n-1}} \cdots \stackrel{d_2}{\lra} A^{\otimes 3}\stackrel{d_1}{\lra} A^{\otimes 2} \lra A \lra 0, 
\end{equation}
is acyclic. This can be seen by constructing a left $A$-module map
\begin{equation}\label{eqn-null-homotopy-map}
    \sigma: A^{\otimes n}\lra A^{\otimes (n+1)},\quad x\mapsto x \otimes 1 
\end{equation}
as the null-homotopy.

Let $A$ be a $\Bbbk$-algebra. In analogy with the usual simplicial bar complex, Mayer introduced on the usual augmented bar complex \eqref{eqn-bar-complex} the linear map
\[
\dif_H(a_0\otimes a_1\otimes\cdots \otimes a_{i+1}):=\sum_{k=0}^{i+1}  a_0\otimes \cdots a_{k-1}\otimes a_ka_{k+1}\otimes a_{k+2} \otimes \cdots \otimes a_{i+1}.
\]
Then it is an easy exercise to show that $\dif^p\equiv 0$. Furthermore, the null-homotopy map $\sigma$ in equation \eqref{eqn-null-homotopy-map} clearly satisfies
\[
\dif_H \sigma -\sigma \dif_H = \Id.
\]
It follows that the augmented $p$-complex
\begin{equation}\label{eqn-p-bar-complex}
(\mathbf{p}_\bullet^\prime (A),\dif_H):=\left(\cdots \xrightarrow{\dif_{H}} A^{\otimes(n+2)} \stackrel{\dif_H}{\lra} A^{\otimes (n+1)}\xrightarrow{\dif_{H}} \cdots \stackrel{\dif_H}{\lra} A^{\otimes 3}\stackrel{\dif_H}{\lra} A^{\otimes 2} \lra A \lra 0\right)
\end{equation}
is acyclic.

Assume next that $(A,\dif_A)$ is a $p$-DG algebra. Extend the $p$-differential on $A$ to any $A^{\otimes (n+1)}$ by the Leibniz rule so that, 
\[
\dif_A(a_0\otimes a_1\otimes\cdots \otimes a_{n}):=\sum_{i=0}^{n} a_0\otimes \cdots \otimes \dif_A(a_i) \otimes \cdots \otimes a_{n}.
\]
As the multiplication map
$m:A\otimes A\lra A$ commutes with $\dif_A$, it follows that the boundary maps in $(\mathbf{p}_\bullet^\prime (A), \dif_M)$ commute with the internal differentials $\dif_A$ on each $A^{\otimes n}$. We may thus consider the total complex $(\mathbf{p}_{\bullet}^\prime (A),\dif_H+\dif_A)$. This construction is equivalent to the totalization $\mc{T}(\mathbf{p}_{\bullet}^\prime (A),\dif_H)$. To make a distinction, we will denote the total differential on $\mathbf{p}_\bullet^\prime (A)$ by $\dif_T:=\dif_H+\dif_A$ in order to avoid potential confusion with the other differentials $\dif_H$ and $\dif_A$.

There is a natural inclusion map 
\[
\iota_A: A\lra \mathbf{p}_\bullet^\prime (A)
\]
of $p$-DG bimodules over $A$, whose cokernel is the $p$-DG bimodule
\[
(\widetilde{\mathbf{p}}_\bullet (A), \dif_T):= \left(\cdots \xrightarrow{\dif_{H}} A^{\otimes(n+2)} \stackrel{\dif_H}{\lra} A^{\otimes (n+1)}\xrightarrow{\dif_{H}} \cdots \stackrel{\dif_H}{\lra} A^{\otimes 3}\stackrel{\dif_H}{\lra} A^{\otimes 2}  \lra 0\right).
\]
Recall here that each $A^{\otimes n}$ also carries its internal differential $\dif_A$. 

\begin{prop}\label{prop-P-prime-acylic}
The total $p$-complex $(\mathbf{p}_{\bullet}^\prime (A),\dif_T)$ is acyclic. Furthermore, if $Q$ is any $p$-DG module over $A$, $\mathbf{p}_{\bullet}^\prime (A)\otimes_A Q$ is also acyclic.
\end{prop}
\begin{proof}
In order to show that $\mathbf{p}^\prime_\bullet(A)$ is acyclic, it suffices to check, by Lemma \ref{lemma-acylicity-commutator-relation}, that
\[
(\dif_H+\dif_A)\sigma-\sigma(\dif_H+\dif_A)=\Id_{\mathbf{p}^\prime},
\]
where $\mathrm{Id}_{\mathbf{p}^\prime}$ is the identity map of $\mathbf{p}^\prime_\bullet(A)$
This is clear since we have the easily verified commutator relations
\[
[\dif_H,\sigma]=\Id_{\mathbf{p}^\prime},\quad [\dif_A,\sigma]=0, \quad [\dif_H,\dif_A]=0.
\]
The last statement is similar, as one just needs to replace the last copy of $A$ in $A^{\otimes n}$ by $Q$.
\end{proof}

\begin{defn}
Suppose $(M,\dif_M)$ is a left $p$-DG module over $A$. Set $M[-1]$ to be the tensor product of $M$ with the $(p-1)$-complex  $q^{p}V_{p-2}$ (see equation \eqref{eqn-Vi}). The \emph{simplicial bar resolution} for $M$ is the $p$-DG module 
\[
\mathbf{p}_\bullet(M):=\widetilde{\mathbf{p}}_\bullet(A)\otimes_A M[-1] 
\]
It inherits the $p$-differential from that of $\dif_T$ and $\dif_M$ via the Leibniz rule. 
\end{defn}
Likewise, one defines the simplicial bar resolution for right $p$-DG modules.

\begin{prop}
For any left $p$-DG module $M$ over $A$, $\mathbf{p}_\bullet(M)$ is a cofibrant replacement of $M$. 
\end{prop}
\begin{proof}
To see the cofibrance of $\mathbf{p}_\bullet(M)$, first note that 
\[
\mathbf{p}_\bullet(M)=\left(\cdots \lra A^{\otimes 3}\otimes M[-1] \lra A^{\otimes 2}\otimes M[-1] \lra A\otimes M[-1] \lra 0\right),
\]
which carries $p$-differentials on the arrows and internal differentials within each term. It
has a natural filtration by $p$-DG submodules, whose subquotients have the form
\[
A^{\otimes n}\otimes M[-1], \quad (n \geq 1)
\ .
\]
As left $p$-DG modules over $A$, such modules are clearly direct sums of free $p$-DG $A$-modules. Therefore $\mathbf{p}_\bullet(M)$ satisfies the ``Property P'' criterion of \cite[Definition 6.3]{QYHopf} and is thus cofibrant.

By construction, there is a short exact sequence of $p$-DG modules over $A$
\[
0 \lra M[-1] \lra \mathbf{p}^\prime_\bullet (A)\otimes_A M[-1] \lra \mathbf{p}_\bullet (M)\lra 0
\ .
\]
Since $\mathbf{p}_\bullet(M)$ is projective as a left $A$-module, the sequence splits when forgetting about $p$-differentials. By \cite[Lemma 4.3]{QYHopf},
the short exact sequence above gives rise to a distinguished triangle in the homotopy category
\[
 M[-1] \lra \mathbf{p}^\prime_\bullet (A)\otimes_A M[-1] \lra \mathbf{p}_\bullet (M)\stackrel{[1]}{\lra} M
 \ .
\]
Therefore there is a morphism $f: \mathbf{p}_\bullet(M)\lra M$ representing the $[1]$ map on the last arrow. By Proposition \ref{prop-P-prime-acylic}, $f$ is a quasi-isomorphism.
\end{proof}

Using the bar resolution, we recall the derived tensor product functor construction in the $p$-DG setting.

\begin{defn} \label{defdertens}
Let $M$ be a left $p$-DG module and $N$ be a right $p$-DG module over $A$. The \emph{$p$-DG derived tensor product} of $N$ and $M$ is the object in $H_q\udmod$ 
\[
N\otimes^{\mathbf{L}}_A M:=  N\otimes_A \mathbf{p}_\bullet(M).
\]
\end{defn}

As in \cite[Corollary 8.9]{QYHopf}, the functor is well defined. Furthermore, it is readily seen that it is independent of cofibrant replacements one chooses for $M$ (or $N$).

\begin{cor}\label{cor-KW-bar-resolution}For any $p$-DG module $M$ over
$(A,\partial_A)$, there is a cofibrant $p$-DG
replacement $\mathbf{p}_{\bullet}(M)\cong M$ in $\mc{D}(A)$. $\hfill\square$
\end{cor}

We define the analogue of Hochschild homology in the $p$-DG setting. As a shorthand notation, for an algebra $A$, we will denote by $A^{\rm en }:=A\otimes A^{\mathrm{op}}$ the \emph{enveloping algebra} of $A$. Thus an $(A,A)$-bimodule is synonymous with a left module over $A^{\rm en}$.

\begin{defn}\label{defn-derived-tensor}
Let $A$ be a $p$-DG algebra, and $M$ be a bimodule over $A$. Then the \emph{$p$-DG Hochschild (co)homology} is the $p$-complex 
\[
\pHH_\bullet(M):=\mH^/_\bullet(\mathbf{p}_\bullet(A)\otimes_{A^{\rm en}} M)
\quad \quad 
(\textrm{resp.}~
\pHH^\bullet(M):=\mH^/_\bullet(\HOM_{A^{\rm en}}(\mathbf{p}_\bullet(A),M))
\ .
\]
\end{defn}

The $p$-DG Hochschild homology is a functorial ``categorical trace'' on the category of $p$-DG bimodules over a $p$-DG algebra, similar to the usual Hochschild homology functor. Here, the functoriality means that, a morphism of $p$-DG bimodules $f: M\lra N$ over $A$ induces a morphism 
of $p$-Hochschild homology groups,
which is defined by
\begin{equation}
    \pHH_\bullet(f):= \mH^/_\bullet ( \mathrm{Id}_{\mathbf{p}_\bullet(A)}\otimes f) : \mH^/_\bullet(\mathbf{p}_\bullet(A)\otimes_{A^{\rm en}} M) \lra \mH^/_\bullet(\mathbf{p}_\bullet(A)\otimes_{A^{\rm en}} N).
\end{equation}

\begin{thm} \label{HHcyclprop}
Given two $p$-DG bimodules $M$ and $N$ over $A$, there is an isomorphism of $p$-complexes
\[
\pHH_{\bullet}(M\otimes^{\mathbf{L}}_A N)\cong \pHH_{\bullet}(N\otimes^{\mathbf{L}}_A M).
\]
\end{thm}
\begin{proof}
This is more or less parallel to the classical Hochschild homology case. For this, we notice that by Definition 
\ref{defdertens},
%\ref{defn-derived-tensor},
\[
M\otimes^{\mathbf{L}}_A N = M\otimes_A \mathbf{p}_\bullet(A) \otimes_A N
\ .
\]
Then by Definition \ref{defn-derived-tensor}, one takes another tensor product over $A^{\rm en}$ with $\mathbf{p}_\bullet(A)$ with respect to the left $A$-action on $M$ and right $A$-action on $N$. This is best visualized as putting everything on a circle:
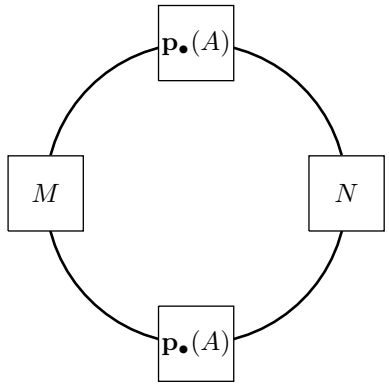
\[
\begin{DGCpicture}
\DGCbubble(2,0){2}
\DGCcoupon(-0.5,-0.5)(0.5,0.5){$M$}
\DGCcoupon(1.5,1.5)(2.5,2.5){${\mathbf{p}_\bullet(A)}$}
\DGCcoupon(3.5,-0.5)(4.5,0.5){$N$}
\DGCcoupon(1.5,-1.5)(2.5,-2.5){${\mathbf{p}_\bullet(A)}$}
\end{DGCpicture}
\ .
\]
Here the connecting lines joining the $p$-DG modules in the picture stand for the usual tensor product over $A$.

Rotating the picture by $180^\circ$, one obtains the $p$-complex computing the $p$-Hochschild homology for the bimodule $N\otimes^{\mathbf{L}}_A M$. The result follows.
\end{proof}

\subsection{Relative homotopy categories}
For any ungraded algebra $B$ over $\Bbbk$, denote by $d_0$ the zero ordinary differential and by $\dif_0$ the zero $p$-differential on $B$, while letting $B$ sit in homological degree zero. When $B$ is graded, the homological grading is indepednent of the internal grading of $B$.

Suppose $(A,\dif_A)$ is a $p$-DG algebra. There is an exact forgetful functor between the usual homotopy categories of chain complexes of graded $A\# H$-modules
\[
\mc{F}_d: \mc{C}(A\# H_q,d_0)\lra \mc{C}(A,d_0).
\]
An object $K_\bullet$ in $\mc{C}(A\# H_q,d_0)$ lies inside the kernel of the functor if and only if, when forgetting the $H_q$-module structure on each term of $K_\bullet$, the complex of graded $A$ modules $\mc{F}_d(K_\bullet)$ is null-homotopic. The null-homotopy map on $\mc{F}_d(K_\bullet)$, though, is not required to intertwine $H_q$-actions.
 
Likewise, there is an exact forgetful functor 
\[
\mc{F}_\dif: \mc{C}(A\# H_q,\dif_0)\lra \mc{C}(A,\dif_0).
\]
Similarly, an object $K_\bullet$ in $\mc{C}(A\# H_q,\dif_0)$ lies inside the kernel of the functor if and only if, when forgetting the $H_q$-module structure on each term of $K_\bullet$, the $p$-complex of $A$ modules $\mc{F}_\dif(K_\bullet)$ is null-homotopic. The null-homotopy map on $\mc{F}_\dif(K_\bullet)$, though, is not required to intertwine $H_q$-actions.

\begin{defn}\label{def-relative-homotopy-category}
Given a $p$-DG algebra $(A,\dif_A)$, the \emph{relative homotopy category} is the Verdier quotient 
$$\mc{C}^{\dif_q}(A,d_0):=\dfrac{\mc{C}(A\# H_q,d_0)}{\mathrm{Ker}(\mc{F}_d)}.$$
Likewise, the \emph{relative $p$-homotopy category} is the Verdier quotient 
$$\mc{C}^{\dif_q}(A,\dif_0):=\dfrac{\mc{C}(A\# H_q,\dif_0)}{\mathrm{Ker}(\mc{F}_\dif)}.$$
\end{defn}
The subscripts in the definitions are to remind the reader of the $H_q$-module structures on the objects.

The categories $\mc{C}^{\dif_q}(A,d_0)$ and  $\mc{C}^{\dif_q}(A,\dif_0)$ are triangulated. By construction, there is a factorization of the forgetful functor
\[
\begin{gathered}
\xymatrix{ \mc{C}(A\# H_q,d_0) \ar[rr]^{\mc{F}_d} \ar[dr] && \mc{C}(A,d_0)\\
& \mc{C}^{\dif_q}(A,d_0)\ar[ur]&
} 
\end{gathered}
\ ,
\quad
\begin{gathered}
\xymatrix{ \mc{C}(A\# H_q,\dif_0) \ar[rr]^{\mc{F}_\dif} \ar[dr] && \mc{C}(A,\dif_0)\\
& \mc{C}^{\dif_q}(A,\dif_0)\ar[ur]&
}
\end{gathered} \ .
\]

Let us briefly remark on the triangulated structures of the relative homotopy categories $\mc{C}^{\dif_q}(A,d_0)$ and $\mc{C}^{\dif_q}(A,\dif_0)$. By construction the shift functors $[\pm 1]$ are inherited from those of $\mc{C}(A\# H_q,d_0)$ and $\mc{C}(A\# H_q,\dif_0)$. In the first case, the functor $[\pm 1]_d$ just shifts complexes one step to the left or right. In the second case, the functor $[1]_\dif$ is given by tensoring with the $p$-complex $U_{p-2}\{p-1\}$ and $[-1]_\dif$ by tensoring with $U_{p-2}\{-1\}$ (see equations \eqref{eqn-p-shift} and \eqref{eqn-p-shift-inverse}).

For the usual homotopy category $\mc{C}(A,d_0)$ of an algebra $A$, standard distinguished triangles arise from short exact sequences
  \[
 0 \lra M_\bullet \stackrel{f}{\lra} N_\bullet  \stackrel{g}{\lra} L_\bullet  \lra 0
  \]
of complexes of $A$-modules that are termwise split exact. The class of distinguished triangles in $\mc{C}(A,d_0)$ are declared to be those that are isomorphic to standard ones. Similarly, termwise split short exact sequences of $p$-complexes of $A$-modules lead to standard distinguished triangles in $\mc{C}(A,\dif_0)$ (\cite[Lemma 4.3]{QYHopf}).
For distinguished triangles in the relative homotopy categories, we have the following construction.

\begin{prop}\label{prop-construction-of-triangle}
\begin{enumerate}
    \item[(i)] A short exact sequence of chain complexes of $A\#H_q$-modules
  \[
 0 \lra M_\bullet  \stackrel{f}{\lra} N_\bullet  \stackrel{g}{\lra} L_\bullet  \lra 0
  \] that is termwise $A$-split exact gives rise to a distinguished triangle in $\mc{C}^{\dif_q}(A,d_0)$. Conversely, any distinguished triangle in $\mc{C}^{\dif_q}(A,d_0)$ is isomorphic to one that arises in this form.
     \item[(ii)]   A short exact sequence of $p$-complexes of $A\#H_q$-modules
  \[
 0 \lra M_\bullet  \stackrel{f}{\lra} N_\bullet  \stackrel{g}{\lra} L_\bullet  \lra 0
  \]
 that is termwise $A$-split exact gives rise to a distinguished triangle in $\mc{C}^{\dif_q}(A,\dif_0)$. Conversely, any distinguished triangle in $\mc{C}^{\dif_q}(A,\dif_0)$ is isomorphic to one that arises in this form.
\end{enumerate}
\end{prop}
\begin{proof}
We will only show the first statement. The proof of the second one is entirely similar.

By construction, distinguished triangles are those in $\mc{C}^{\dif_q}(A,d_0)$ that are isomorphic to standard distinguished triangles arising from  short exact sequences of $A\# H_q$-modules that are termwise $A\#H_q$-split exact. Forgetting about the $H_q$-actions, such sequences are also termwise $A$-split exact.

Now let $f: M_\bullet  \lra N_\bullet $ be the injection as in the statement.  The cone of $f$ in $\mc{C}(A\# H_q, d_0)$ is given by
\[
C_\bullet (f) \cong 
\left(
M_\bullet [1]_d \oplus N_\bullet , d_{C}:=
\begin{pmatrix}
d_{M[1]_d} & f \\
0 & d_N
\end{pmatrix}
\right).
\]  
The cone fits into a short exact sequence of $A\#H_q$-modules that are termwise $A\# H_q $-split:
\[
0\lra N_\bullet  \lra C_\bullet (f)\lra M_\bullet [1]_d\lra 0.
\]
Associated with this sequence is the (rotated) standard distinguished triangle 
\[
N_\bullet  \lra C_\bullet (f) \lra M_\bullet [1]_d \lra N_\bullet [1]_d
\]
in $\mc{C}(A\# H_q, d_0)$, which descends to a standard distinguished triangle in $\mc{C}^{\dif_q}(A,d_0)$.

To prove the statement, it then suffices to show that, in the relative homotopy category, we have an isomorphism
$C_\bullet (f) \cong L_\bullet $. Consider the map
\[
g^\prime : C_\bullet (f) = (M_\bullet [1]_d \oplus N_\bullet , d_{C}) \lra L_\bullet , \quad (m,n)\mapsto g(n).
\]
It is easily checked that $g^\prime$ is a surjective map of chain complexes, and the kernel is isomorphic to $C_\bullet (\Id_M)$. Thus we have a short exact sequence of chain complexes of $A\# H_q$-modules
\[
0\lra C_\bullet (\Id_M) \lra C_\bullet (f) \stackrel{g^\prime}{\lra} L_\bullet  \lra 0 .
\]
Now, under $\mc{F}_d$, the sequence termwise splits over $A$: 
\[
\mc{F}_d (C_\bullet (f)) \cong 
\left(
\mc{F}_d(M_\bullet [1]_d\xrightarrow{\Id_M} M_\bullet ) \oplus \mc{F}_d(L_\bullet )
\right)
\cong \left(\mc{F}_d(C_\bullet (\Id_M))\oplus \mc{F}_d(L_\bullet )\right).
\]
It follows that we have a distinguished triangle in $\mc{C}(A,d_0)$
\[
0\cong \mc{F}_d (C_\bullet (\Id_M)) \lra \mc{F}_d (C_\bullet (f)) \xrightarrow{\mc{F}_d(g^\prime)} \mc{F}_d (L_\bullet ) \lra \mc{F}_d (C_\bullet (\Id_M)[1]_d)\cong 0,
\]
implying that $g^\prime$ is an isomorphism under $\mc{F}_d$. The result follows.
\end{proof}

We also record the following useful fact.

\begin{prop}
The $p$-extension functor $\mc{P}: \mc{C}(A\# H_q, d_0)\lra \mc{C}(A\# H_q, \dif_0)$ descends to a functor, still denoted by $\mc{P}$, between the relative homotopy categories:
$$\mc{P}: \mc{C}^{\dif_q} (A, d_0)\lra \mc{C}^{\dif_q}(A, \dif_0) \ .$$
\end{prop}
\begin{proof}
It suffices to show that, if $K_\bullet \in \mathrm{Ker}(\mc{F}_d)$, then $\mc{P}(K_\bullet)\in \mathrm{Ker}(\mc{F}_\dif)$. This is clear since, if $h$ provides a null-homotopy in $\mc{C}(A,d_0)$ for $\mc{F}_d(K_\bullet)$, then $\mc{P}(h)$ is the null-homotopy for $\mc{F}_\dif(\mc{P}(K_\bullet))$.
\end{proof}

\subsection{Relative \texorpdfstring{$p$}{p}-Hochschild homology}
In this paper, instead of the absolute version of $p$-Hochschild homology, we will need a relative version of $p$-Hochshild homology for a $p$-DG algebra, which we define now. An important reason for introducing the relative homotopy category is that the relative $p$-Hochschild homology functor descends to this category.

Let $(A,\dif_A)$ be a $p$-DG algebra. Equip the zero differential $d_0$ and $p$-differential $\dif_0$ on $A$, and denote the resulting trivial ($p$)-DG algebras by $(A_0,d_0)$ and $(A_0,\dif_0)$. Likewise, for a ($p$-)DG bimodule $M$ over $A$, we temporarily denote by $M_0$ the $A$-bimodule equipped with zero ($p$-) differentials.

The usual Hochschild homology of $M_0$ over $(A_0,d_0)$ in this case carries a natural $H_q$-action, since the $H_q$-action commutes with all differentials in the simplicial bar complex \eqref{eqn-bar-complex} for $A_0$. 

\begin{defn}
The \emph{relative Hochschild homology} of a $p$-DG bimodule $(M,\dif_M)$ over $(A,\dif_A)$ is the usual Hochschild homology of $M_0$ over $(A_0,d_0)$ equipped with the induced $H_q$-action from $\dif_M$ and $\dif_A$, and denoted
\[
\mHH^{\dif_q}_\bullet(M):=\mHH_\bullet(A_0,M_0)
\]
\end{defn}

Replacing the usual simplicial bar complex by Mayer's $p$-simplicial bar complex $\mathbf{p}_\bullet(A_0)$ for $(A_0,\dif_0)$, we make the following definition.

\begin{defn}
The \emph{relative $p$-Hochschild homology} of $M$ is the $p$-complex of 
\[
\pHH^{\dif_q}_\bullet (M):=\mH^/_{\bullet}(A_0 \otimes_{A_0^{\rm en}}^{\mathbf{L}} M_0)=\mH^/_{\bullet}(\mathbf{p}(A_0) \otimes_{A_0^{\rm en}} M_0)
\ .
\]
\end{defn}
Similar to $p$-Hochschild homology, the relative case is also covariant functor: If $f: M \lra N$ is a morphism of $p$-DG bimodules over $A$, it induces
\[
\pHH_\bullet^{\dif_q}( f ):=\mH^/_{\bullet}(\mathrm{Id}_{A_0}\otimes f): \mH^/_{\bullet}(A_0 \otimes_{A_0^{\rm en}}^{\mathbf{L}} M_0) 
\lra \mH^/_{\bullet}(A_0 \otimes_{A_0^{\rm en}}^{\mathbf{L}} N_0)
\ .
\]

\begin{prop}
The relative $p$-Hochschild homology descends to a functor defined on the relative homotopy category $\mc{C}^{\dif_q}(A,\dif_0)$ of $p$-DG bimodules over $A$.
\end{prop}
\begin{proof}
An object that lies in the kernel $\mc{F}$ for $p$-DG modules over $A\otimes A^{\mathrm{op}}$ consists of null-homotopic $p$-complexes of bimodules over $(A_0,\dif_0)$. Thus the relative $p$-Hochschild homology functor annihilates such objects, and descends to the quotient category.
\end{proof}

We also have the trace-like property for relative $p$-Hochschild homology.

\begin{prop}\label{HHrelativecyclprop}
Given two $p$-DG bimodules $M$ and $N$ over $A$, there is an isomorphism of $p$-complexes of $H_q$-modules
\[
\pHH^{\dif_q}_\bullet(M\otimes^{\mathbf{L}}_A N)\cong \pHH^{\dif_q}_\bullet(N\otimes^{\mathbf{L}}_A M).
\]
\end{prop}
\begin{proof}
This follows from Theorem \ref{HHcyclprop} by replacing $(A,\dif_A)$ with $(A_0,\dif_0)$.
\end{proof}

Our next goal is to show that we may relax the requirement that we utilize the simplicial bar resolution when computing the relative Hochschild homology. For the next theorem, we use the fact that in the simplicial bar complex $\mathbf{p}_\bullet(A_0)$, all the $p$-complex maps are $H_q$-equivariant since they are just sums of multiplications maps of $A$ tensored with identities maps on $A$.

\begin{thm}\label{thm-resolution-independence}
Let $M$ be a $p$-DG bimodule over $A$. Suppose $f:Q_\bullet \lra M$ is a $p$-complex resolution of $M$ over $(A_0,\dif_0)$ which is $H_q$-equivariant, and each term of $Q_\bullet$ is projective as an $A_0^{\rm en}$-module. Then $f$ induces an isomorphism of $H_q$-modules
\[
\mH^{/}_\bullet(A_0\otimes_{A_0^{\rm en}}Q_\bullet)\cong \pHH^{\dif_q}_\bullet(M).
\]
\end{thm}
\begin{proof}
By definition, there is a short exact sequence over $(A_0,\dif_0)$:
\[
0\lra M \lra C_\bullet(f)\lra Q_\bullet[1]\lra 0.
\]
The cone $C_\bullet(f)$, by construction, is equal to
\[
C_\bullet(f) \cong \dfrac{Q_\bullet \otimes \Bbbk[\dif_0]/(\dif_0^p)\oplus M}{ \left\{
(x\otimes \dif_0^{p-1}, f(x))|x\in Q_\bullet
\right\}} \ ,
\]
which is then an acyclic $p$-complex of bimodules over $(A,\dif_0)$. Since the $H_q$-actions on the modules $Q_\bullet$ and $M$ commute with the $\dif_0$ action, the above short exact sequence is an exact sequence of $H_q$-modules.

Taking the tensor product of $\mathbf{p}_\bullet(A_0)$ with the sequence, we get a short exact sequence since $\mathbf{p}_\bullet(A_0)$ is projective as a module over $A_0^{\rm en}$
\[
0\lra \mathbf{p}_\bullet(A_0)\otimes_{A_0^{\rm en}} M
\lra \mathbf{p}_\bullet(A_0)\otimes_{A_0^{\rm en}} C_\bullet(f)\lra 
\mathbf{p}_\bullet(A_0)\otimes_{A_0^{\rm en}} Q_\bullet[1]\lra 0
\ .
\]
The middle term $\mathbf{p}_\bullet(A_0)\otimes_{A_0^{\rm en}} C_\bullet(f)$ is a $p$-complex of $A_0$-bimodules equipped with a filtration, whose subquotients are isomorophic to grading shifts of $C_\bullet(f)$. Therefore it is acyclic. It follows that we have an isomorphism in $\mc{D}((A_0^{\rm en}\# H_q,\dif_0)$
\[
\mathbf{p}_\bullet(A_0)\otimes_{A_0^{\rm en}} Q_\bullet\cong \mathbf{p}_\bullet(A_0)\otimes_{A_0^{\rm en}} M
\ .
\]
Taking $p$-Hochschild homology, we obtain an isomorphism of $H_q$-modules
\[
\mH^/_\bullet(\mathbf{p}_\bullet(A_0)\otimes_{A_0^{\rm en}} Q_\bullet) \cong \mH^/_\bullet(\mathbf{p}_\bullet(A_0)\otimes_{A_0^{\rm en}} M).
\]
Similarly, since $Q_\bullet$ is projective as a bimodule over $A_0$, tensoring Mayer's short exact sequence of bimodules with $Q_\bullet$ over $A_0^{\rm en}$ remains exact:
\[
0\lra A_0 \otimes_{A_0^{\rm en}}  Q_\bullet  \lra  \mathbf{p}^{\prime}_\bullet (A_0)\otimes_{A_0^{\rm en}}  Q_\bullet\lra   \mathbf{p}_\bullet(A_0)[1]\otimes_{A_0^{\rm en}}  Q_\bullet \lra 0
\]
The middle term is, as above, acyclic since $\mathbf{p}_\bullet^{\prime}(A_0)$ is. It follows that
\[
 A_0 \otimes_{A_0^{\rm en}} Q_\bullet  \cong  \mathbf{p}_\bullet(A_0)\otimes_{A_0^{\rm en}} Q_\bullet 
\]
as objects in $\mc{D}(H_q,\dif_0)$. Taking slash homology on both sides gives the desired result.
\end{proof}

\section{\texorpdfstring{$p$}{p}-DG Soergel bimodules and braid relations} \label{secbraidrelations}

\subsection{\texorpdfstring{$p$}{p}-DG bimodules over the polynomial algebra}\label{subset-p-DG-pol}
Let $\Bbbk$ be a field of characteristic\footnote{This characteristic assumption is standard in the literature on Soergel bimodules. See, for instance \cite[Section 2]{EliasKh}.  Remark \ref{rem-char-2} also provides  an explanation for this assumption.} $p>2$. The graded polynomial algebra $R_n=\Bbbk[x_1,\ldots,x_n]$ has a natural module-algebra structure over the graded Hopf algebra $H_q=\Bbbk[\dif_q]/(\dif_q^p)$, where the generator $\partial_q \in H_q$ acts as a derivation determined by $\partial_q(x_i)=x_i^2$ for $i=1,\ldots,n$. Here the degree of each $x_i$ and $\dif_q$ are both two, and will be referred to as the \emph{$q$-degree}.

When $n$ is clear from the context, we will abbreviate $R_n$ by just $R$.
The differential is invariant under the permutation action of the symmetric group $S_n$ on the indices of the variables. Therefore let the subalgebra of polynomials symmetric in variables $x_i$ and $x_{i+1}$ with its inherited $H_q$-module structure be denoted by
\[
R^i_n=\Bbbk[x_1,\ldots,x_{i-1},x_i+x_{i+1},x_i x_{i+1},x_{i+2},\ldots,x_n].
\]
More generally, given any subgroup $G\subset S_n$, the invariant subalgebra $R_n^G$ inherits an $H_q$-algebra structure from $R_n$ (and is thus a $p$-DG algebra). In particular, we will also use the $H_q$-subalgebra 
$
R^{i,i+1}_n:= R^{S_3}_n
$,
where $S_3$ is the subgroup generated by permuting the indices $i$, $i+1$ and $i+2$.

When the number of variables $n$ is clear from the context or is irrelevant of the statements, we will abbreviate $R_n$ by just $R$ in what follows.

The $(R,R)$-bimodule 
$B_i=R \otimes_{R^i} R$ has the structure of an $H_q$-module (and is thus a $p$-DG bimodule) where the differential acts via the Leibniz rule: for any $h\otimes g\in R\otimes_{R^i} R$,
$$
\partial_q(h \otimes g)=\partial_q(h) \otimes g+ h \otimes \partial_q(g).
$$
The tensor category of $(R,R)$-bimodules generated by the $B_i$, has an $H_q$-module structure where the action comes from the comultiplication in $H_q$.  We denote this category by
$(R,R) \# H_q \dmod$.

Let $f=\sum_{i=1}^{n} a_ix_i \in \F_p[x_1,\dots, x_n]\subset R $ be a linear function.  We twist the $H_q$-action on the bimodule $B_i$ to obtain a bimodule $B_i^f$ defined as follows.
As an $(R_n,R_n)$-bimodule, it is the same as $B_i$ but the action of $H_q$ is twisted by defining 
\begin{subequations}
\begin{equation}\label{eqn-twistonBi-left}
  \partial_q(1 \otimes 1)=(1 \otimes 1)f.
\end{equation}
Similarly we define ${}^f B_i$ where now 
\begin{equation}\label{eqn-twistonBi-right}
    \partial_q(1\otimes1)=f(1 \otimes 1).
\end{equation}
\end{subequations}

For $R_n$ as a bimodule over itself, it is clear that $^fR_n \cong R_n^f$ as $p$-DG bimodules.
It follows that there are $p^n$ ways to put an $H_q$-module structure on a rank-one free module over $R_n$.
Each such $H_q$-module is quasi-isomorphic to a finite-dimensional $p$-complex. Choose numbers $b_i\in \{1,\dots ,p  \}$ such that $b_i\equiv a_i~(\mathrm{mod}~p)$, $i=1,\dots, n$, and define the $H_q$-ideal of $R$ 
\begin{equation}
I=(x_1^{p+1-b_1},\cdots, x_{n}^{p+1-b_n}).
\end{equation}
Then the natural quotient map
\begin{equation}
\pi:R_n^f \twoheadrightarrow R_n^f/(I\cdot R_n^f)
\end{equation}
is readily seen to be a quasi-isomorphism. 

\begin{lem}\label{lem-pol-mod}
For each $f=\sum_{i=1}^n a_ix_i$, the rank-one $p$-DG module $R_n^f$ has finite-dimensional slash homology: \
\[
\mH^/_{\bullet}(R_n^f)\cong \bigotimes_{i=1}^n V_{p-b_i} \{p-b_i\} .
\]
In particular, if any $a_i$ of $f=\sum_i a_ix_i$ is equal to one, then $\mH^/_{\bullet}(R_n^f)=0$.
\end{lem}
\begin{proof}
The first statement follows from the discussion before the lemma. For the second statement, note that $V_{p-1}$ is acyclic, and the tensor product of an acyclic $H_q$-module with any $H_q$-module is acyclic.
\end{proof}

\begin{cor}\label{cor-finite-slash-homology}
Let $M$ be a $p$-DG module over $R$ which is equipped with a finite filtration, whose subquotients are isomorphic to $R^f$ for various $f$. Then $M$ has finite-dimensional slash homology.
\end{cor}
\begin{proof}
Induct on the length of the filtration, and apply the previous lemma.
\end{proof}

By Definition \ref{def-relative-homotopy-category}, a morphism $f \colon A \longrightarrow B$ in the homotopy category
$\mc{C}((R,R) \# H_q,d_0)$ is a relative isomorphism if $\mathcal{F}_d(f)$
is an isomorphism in the homotopy category $\mc{C}((R,R),d_0)$.
Localizing $\mc{C}((R,R) \# H_q,d_0)$ at all relative isomorphisms produces the relative homotopy category of $(R,R)$-bimodules $\mc{C}^{\dif_q}(R,R,d_0)$.
Similarly we have the $p$-version of the relative homotopy category $\mc{C}^{\dif_q}(R,R,\dif_0)$.
%, which we will abbreviate $\mc{C}_\dif(R,R)$ for ease of notation.

\subsection{Elementary braiding complexes}
\begin{lem} 
There are $(R,R) \# H_q$-module homomorphisms
\begin{enumerate}
\item[(i)] $rb_i \colon R \longrightarrow q^{-2} B_i^{-(x_i+x_{i+1})}$, where $1 \mapsto (x_{i+1} \otimes 1 
- 1 \otimes x_{i}) $;
\item[(ii)] $br_i \colon B_i \longrightarrow R$, where $1 \otimes 1 \mapsto 1$.
\end{enumerate}
\end{lem}

\begin{proof}
The fact that these maps are $(R,R)$-bimodule homomorphisms is well known. See, for instance, \cite{EliasKh} for more details.

In order to check that the homomorphism is compatible with the $H_q$-module structure, we must check that $rb_i(\partial_q(1))=\partial_q(rb_i(1))$.
Clearly $rb_i(\partial_q(1))=0$.
On the other hand
\begin{align*}
\partial_q(rb_i(1))&=\partial_q(x_{i+1} \otimes 1 - 1 \otimes x_{i}) \\
&=x_{i+1}^2 \otimes 1 - 1 \otimes x_{i}^2
+(x_{i+1} \otimes 1 - 1 \otimes x_{i})
(-x_i - x_{i+1})  \\
&=0
\end{align*}
where the third term in the second equation above comes from the twist on $B_i$.

The second homomorphism clearly respects the $H_q$-structure since 
$\partial_q(1 \otimes 1)=0=\partial_q(1)$.
\end{proof}

We now have complexes of $(R,R) \# H_q$-modules 
\begin{equation}\label{eqn-elementary-braids}
T_i :=
\left(t B_i  \xrightarrow{br_i} R\right)
,
\quad \quad \quad
T_i' :=  \left(R \xrightarrow{rb_i} q^{-2} t^{-1} B_i^{-(x_i+x_{i+1})}\right)
.
\end{equation}
In the coming sections we will, for presentation reasons, often omit the various shifts built into the definitions of $T_i$ and $T_i'$.

We associate respectively to the left and right crossings  $\sigma_i$ and $\sigma_i^{\prime}$ between the $i$th and $(i+1)$st strands in \eqref{2crossings} the chain complexes of $(R,R)\#H_q$-bimodules $T_i$ and $T_i'$: 
\begin{equation} \label{2crossings}
\sigma_i:=
\begin{DGCpicture}
\DGCPLstrand(-1,0)(-1,1)
\DGCPLstrand(0,0)(1,1)
\DGCPLstrand(1,0)(.75,.25)
\DGCPLstrand(0,1)(.25,.75)
\DGCPLstrand(2,0)(2,1)
\DGCcoupon*(-1,0)(0,1){$\cdots$}
\DGCcoupon*(1,0)(2,1){$\cdots$}
\end{DGCpicture}
\hspace{1in}
\sigma_i^\prime:=
\begin{DGCpicture}
\DGCPLstrand(-1,0)(-1,1)
\DGCPLstrand(0,0)(.25,.25)
\DGCPLstrand(.75,.75)(1,1)
\DGCPLstrand(1,0)(0,1)
\DGCPLstrand(2,0)(2,1)
\DGCcoupon*(-1,0)(0,1){$\cdots$}
\DGCcoupon*(1,0)(2,1){$\cdots$}
\end{DGCpicture}
\end{equation}
More generally, if $\beta\in \mathrm{Br}_n$ is a braid group element written as a product $\sigma_{i_i}^{\epsilon_1}\cdots \sigma_{i_k}^{\epsilon_k}$ in the elementary generators where $\epsilon_i\in \{\emptyset, \prime\}$, we assign the chain complex of $(R,R)\# H_q$-bimodules
\begin{equation}
    T_\beta:=T_{i_1}^{\epsilon_1}\otimes_R\cdots \otimes_R T_{i_k}^{\epsilon_k}.
\end{equation}

\begin{thm}\label{thm-braid-invariant}
Given any braid group element $\beta\in \mathrm{Br}_n$, the chain complex of $T_\beta$ associated to it is a well-defined element of the relative homotopy category $\mc{C}^{\dif_q}(R,R,d_0)$.
\end{thm}

The proof of the theorem, which is like the analogous proof in \cite{KRWitt}, will take up the rest of this section. We will repeatedly apply Proposition \ref{prop-construction-of-triangle} to simplify complexes in the relative homotopy category.

\subsection{Reidemeister II}
The following lemma is crucial to proving the Reidemeister II braid relation and should be compared with \cite[Lemma 4.3]{KRWitt}.  

\begin{lem} \label{BB=B+B}
There exists an isomorphism of $(R,R) \#H_q$-modules
\[
B_i \otimes_R B_i \cong B_i \oplus q^2 B_i^{x_i+x_{i+1}}
\]
defined by
\[
\phi=\begin{pmatrix} \phi_1 \\ \phi_2 \end{pmatrix} \colon B_i \otimes_R B_i \longrightarrow B_i \oplus q^2 B_i^{x_i+x_{i+1}},
\quad 
\phi_1(1 \otimes 1 \otimes 1)=
\begin{pmatrix}
1 \otimes 1 \\ 0
\end{pmatrix},
\quad 
\phi_2(
1 \otimes (x_{i+1}-x_{i}) \otimes 1 )
=
\begin{pmatrix}
0 \\
1 \otimes 1
\end{pmatrix},
\]
\[
\psi=(\psi_1, \psi_2) \colon 
B_i \oplus q^2 B_i^{x_i+x_{i+1}} \longrightarrow B_i \otimes_R B_i ,
\quad
\psi_1(1 \otimes 1)=1 \otimes 1 \otimes 1, 
\quad
\psi_2(1 \otimes 1)=1 \otimes (x_{i+1}-x_i) \otimes 1.
\]
\end{lem}

\begin{proof}
As an $(R,R)$-bimodule, $B_i \otimes_R B_i$ is generated by
$1 \otimes 1 \otimes 1 $ and
$1 \otimes (x_{i+1}-x_i) \otimes 1$.
It is a classical fact that the maps $\phi$ and $\psi$ provide inverse isomorphisms of bimodules.

It is straightforward to check that $\phi$ and $\psi$ are compatible with the $H_q$-module structure.
Note that this condition forces the twist $x_i+x_{i+1}$ on the second $B_i$ term.
\end{proof}

\begin{rem}\label{rem-char-2}
This lemma is where the characteristic $p>2$ assumption is crucially needed. Indeed, in characteristic $2$, the element $1 \otimes (x_{i+1}-x_i) \otimes 1 =1 \otimes (x_{i+1}+x_i) \otimes 1= (x_{i+1}+x_i)\otimes 1  \otimes 1$ becomes a redundant element in the standard set of bimodule generators for $B_i\otimes_R B_i$.
\end{rem}

\begin{prop} \label{propR2}
There are isomorphisms in the homotopy category
$\mc{C}((R,R) \# H_q, d_0)$
\[ 
T_i \otimes_R T_i' \cong \Id \cong T_i' \otimes_R T_i.
\]
\end{prop}

\begin{proof}
In order to prove the isomorphism $T_i' \otimes_R T_i\cong \Id$, we tensor the complexes for $T_i'$ and $T_i$ to obtain 
%\[
%\xymatrix{
%R \otimes_R R \ar[rr]^-{rb_i \otimes \Id}&& B^{-(x_i+x_{i+1})}_i \otimes_R %R\\
%R \otimes_R B_i\ar[u]^{\Id \otimes br_i}\ar[rr]^-{rb_i \otimes \Id}&& %B^{-(x_i+x_{i+1})}_i \otimes_R B_i~.\ar[u]_{\Id \otimes br_i}
%}.
%\]
%We then collapse this bicomplex along the northwest-southeast diagonal into a complex
\begin{equation} \label{TiTi'1}
\begin{gathered}
\xymatrix{
& R \otimes_R R \ar[dr]^{rb_i \otimes \Id} & \\
R \ar[ur]^{\Id \otimes br_i} \ar[dr]_{-rb_i \otimes \Id} \otimes_R B_i & & q^{-2} B_i^{-(x_i+x_{i+1})} \otimes_R R \\
& q^{-2} B^{-(x_i+x_{i+1})}_i \otimes_R B_i \ar[ur]_{\Id \otimes br_i} &
}
\end{gathered}
\end{equation}
By Lemma \ref{BB=B+B}, the complex \eqref{TiTi'1} is isomorphic to 
\begin{equation} \label{TiTi'2}
\begin{gathered}
\xymatrix{
& R \otimes_R R \ar[dr]^{rb_i \otimes \Id} & \\
R \otimes_R B_i \ar[ur]^{\Id \otimes br_i} \ar[dr]_-{\begin{pmatrix}
\gamma_1 \\ \gamma_2 \end{pmatrix}}  & & q^{-2} B_i^{-(x_i+x_{i+1})} \otimes_R R \\
& q^{-2} B^{-(x_i+x_{i+1})}_i \oplus B_i \ar[ur]_{\begin{pmatrix} \delta_1 & \delta_2 \end{pmatrix}} &
}
\end{gathered}
\end{equation}
where 
\begin{equation*}
\gamma_1(1 \otimes 1)=  \frac{1}{2}(x_i-x_{i+1}) \otimes 1, \quad \quad
\gamma_2(1 \otimes 1)=  -\frac{1}{2}(1 \otimes 1), \quad \quad
\delta_1(1 \otimes 1)=1 \otimes 1, \quad \quad
\delta_2(1 \otimes 1)=1 \otimes (x_{i+1}-x_i).
\end{equation*}
Contracting out the terms $B_i$, one gets that equation \eqref{TiTi'2} is homotopic to
\begin{equation} \label{TiTi'3}
\begin{gathered}
\xymatrix{
 R  \ar[dr]^{rb_i} & \\
 & q^{-2} B^{-(x_i+x_{i+1})}_i  \\
 q^{-2} B^{-(x_i+x_{i+1})}_i  \ar[ur]_{\Id} &
}
\end{gathered}
\end{equation}
Finally, contracting out the terms $B^{-(x_i+x_{i+1})}$, one gets
$T'_i \otimes_R T_i \cong R$.
Note that each homotopy $\gamma$ used in the contractions has the property $\partial(\gamma)=0$.

%There are maps of complexes of $p$-DG bimodules 
%\[
%\xymatrix{
%R \ar[r]^{\Gamma_1} & T_i \otimes_R T_i' \ar[r]^{\Gamma_2} & R
%}
%\]
%given by
%\[
%\xymatrix{
%& R \ar@/_/[d]_{\Id} \ar@/^2pc/[ddd]^{f} & \\
%& R \otimes_R R \ar[dr]^{rb_i \otimes \Id} \ar@/_5pc/[ddd]_{\Id} & \\
%R \ar[ur]^{\Id \otimes br_i} \ar[dr]^{-rb_i \otimes \Id} \otimes_R B_i & %& B^{-(x_i+x_{i+1})} \otimes_R R \\
%& B^{-(x_i+x_{i+1})}_i \otimes_R B_i \ar[ur]^{\Id \otimes br_i} %\ar@/^2pc/[d]^{g} & \\
%& R &
%}
%\]
%where
%\[
%f= 
%-\psi_1 \circ rb_i \quad \quad
%g=  br_i \circ \phi_2.
%\]
%It is easy to see that the composition of these maps $\Gamma_2 \circ %\Gamma_1$ is  $\Id \colon R \rightarrow R$.

%Below is a map $H \colon T_i \otimes_R T_i' \rightarrow T_i \otimes_R %T_i' $
%\[
%\xymatrix{
%& R \otimes_R R \ar[dr]^{rb_i \otimes \Id} & \\
%R \ar[ur]^{\Id \otimes br_i} \ar[dr]^{-rb_i \otimes \Id} \otimes_R B_i & %& B^{-(x_i+x_{i+1})} \otimes_R R \ar@/_2pc/[ldddd]_{-\psi_1} \\
%& B^{-(x_i+x_{i+1})}_i \otimes_R B_i \ar[ur]^{\Id \otimes br_i} %\ar@/_2pc/[ldd]_{2\phi_2}& \\
%& R \otimes_R R \ar[dr]^{rb_i \otimes \Id} & \\
%R \ar[ur]^{\Id \otimes br_i} \ar[dr]^{-rb_i \otimes \Id} \otimes_R B_i & %& B^{-(x_i+x_{i+1})} \otimes_R R \\
%& B^{-(x_i+x_{i+1})}_i \otimes_R B_i \ar[ur]^{\Id \otimes br_i} &
%}
%\]
%It is easy to check that $H$ is a homotopy between $\Gamma_1 \circ %\Gamma_2 - \Id$ and the zero map.  It is also straightforward to check %that $H$ is compatible with the $p$-DG structure.

%The proof that $T_i' \otimes_R T_i \cong \Id$ is similar.
\end{proof}

\subsection{Reidemeister III} \label{reid3sect}
Let $B_{i,i+1,i}=R \otimes_{R^{i,i+1}}R$ where
\begin{equation*}
R^{i,i+1}=\Bbbk[x_1,\ldots,x_{i-1},e_1(x_i,x_{i+1},x_{i+2}),e_2(x_i,x_{i+1},x_{i+2}),e_3(x_i,x_{i+1},x_{i+2}),x_{i+3},\ldots,x_n].    
\end{equation*}

The following lemma is a rephrasing of \cite[Lemma 4.7]{KRWitt}. We will use the map $\phi_2$ from the statement of Lemma \ref{BB=B+B}.

\begin{lem} \label{Bii+1ilemma}
There exists a short exact sequence of $(R,R) \# H_q$-modules which splits upon restricting to the category of $(R,R)$-bimodules
\begin{equation*}
\xymatrix{
0 \ar[r] & B_{i,i+1,i} \ar[r]^-{f_1} & B_i \otimes B_{i+1} \otimes B_i \ar[r]^-{f_2}
& B_i^{x_i+x_{i+1}} \ar[r] & 0
}
\end{equation*}
where
\begin{equation*}
f_1(1 \otimes 1)=1 \otimes 1 \otimes 1, \quad \quad \quad
f_2=\phi_2 \circ (\Id \otimes br_{i+1} \otimes \Id).
\end{equation*}
\end{lem}

\begin{proof}
Ignoring the $H_q$-module structure, this is a classical statement.  See, for example \cite[Definition 3.9, Section 4.5]{EliasKh}.

Clearly $\partial_q(f_1)=0$.  Since $\partial_q(\phi_2)=0$ and
$ \partial_q(br_{i+1})=0$, it follows that $\partial_q(f_2)=0$.
\end{proof}

\begin{prop} \label{R3prop}
There exist isomorphisms in the relative homotopy category
$\mc{C}^{\dif_q}(R,R,d_0)$ 
\begin{enumerate}
    \item[(i)] $T_i T_{i+1} T_i \cong T_{i+1} T_i T_{i+1}$,
    \item[(ii)] $T_i' T_{i+1}' T_i' \cong T_{i+1}' T_i' T_{i+1}'$.
\end{enumerate}
\end{prop}

\begin{proof}
We will prove the first isomorphism.  The second follows from the first part and Proposition \ref{propR2}.

By definition
\begin{equation*}
T_i T_{i+1} T_i \cong
\big(\xymatrix{
B_i \ar[r]^{br_i} & R
}
\big)
\big(\xymatrix{
B_{i+1} \ar[r]^{br_{i+1}} & R
}
\big)
\big(\xymatrix{
B_i \ar[r]^{br_i} & R
}
\big).
\end{equation*}
We will use the notation $f_1$ and $f_2$ defined in Lemma \ref{Bii+1ilemma}
and $\phi_1, \phi_2, \psi_1,$ and $\psi_2$ defined in 
Lemma \ref{BB=B+B}. 
There is a short exact sequence of $(R,R) \# H_q$-modules which splits when forgetting the $H_q$-actions involved:
\begin{equation*}
0 \longrightarrow E_1 \longrightarrow T_i T_{i+1} T_i \longrightarrow E_2 \longrightarrow 0  \end{equation*} 
which we write vertically in diagram \eqref{R3eq1} below.
Most of the maps in the complexes below transform a $B_i$ or $B_{i+1}$ into an $R$ (via $br_i$ or $br_{i+1}$) and act as the identity on the other tensor factors.  We use the following shorthand notation. A ``$+$'' indicates that the coefficient of such a map is $1$, and a ``$-$'' indicates that the coefficient of such a map is $-1$. For example, the negative sign below stands for
\begin{equation*}
- := - br_i \otimes \Id \otimes \Id \colon
B_i \otimes R \otimes B_i \longrightarrow R \otimes R \otimes B_i.
\end{equation*}

\begin{equation}
\label{R3eq1}
\begin{gathered}
\xymatrix{
& & R \otimes B_{i+1} \otimes B_i \ar[r]^{+} \ar[dr]^>{ - }
& R \otimes R \otimes B_i \ar[dr]^{+}
& \\
E_1:= \ar[dddd] &B_{i,i+1,i} \ar[r]^{\phi_1 \circ + \circ f_1} \ar[ur]^{+ \circ f_1} \ar[dr]^{+ \circ f_1} \ar[dddd]_{f_1}
& B_i   \ar[ur]^<<<{- \circ \psi_1} \ar[dr]_<<<{+ \circ \psi_1}
& R \otimes B_{i+1} \otimes R \ar[r]^-{+}
& R \ar[dddd]^{\Id} \\
& & B_i \otimes B_{i+1} \otimes R \ar[r]^{-} \ar[ur]^<{+} \ar[dd]_{\begin{pmatrix}
\Id & & \\
& \psi_1 & \\
 & & \Id \end{pmatrix}}
& B_i \otimes R  \otimes R \ar[ur]^{+} \ar[dd]_{\begin{pmatrix}
\Id & & \\
 & \Id & \\
 & & \Id \end{pmatrix}}
& \\
& & & & \\
%%%%%
& & R \otimes B_{i+1} \otimes B_i \ar[r]^{+} \ar[dr]_<{-}
& R \otimes R \otimes B_i \ar[dr]^{+}
& \\
T_i T_{i+1} T_i= \ar[ddd] & B_i \otimes B_{i+1} \otimes B_i \ar[r]^{+} \ar[ur]^{+} \ar[dr]^{+} \ar[ddd]^{f_2}
& B_i \otimes R \otimes B_i  \ar[ur]^{-} \ar[dr]_<{+}
& R \otimes B_{i+1} \otimes R \ar[r]^-{+}
& R \\
& & B_i \otimes B_{i+1} \otimes R \ar[r]^{-} \ar[ur]_{+} \ar[dd]^{\begin{pmatrix} 0 & \phi_2 & 0 \end{pmatrix}}
& B_i \otimes R  \otimes R \ar[ur]^{+}
& \\
%%%%%%
& & & & \\
E_2:= & B_i^{x_i+x_{i+1}} \ar[r]^{\Id} & B_i^{x_i+x_{i+1}} &
}
\end{gathered}
\end{equation}
All of the chain maps are annhilated by $\partial_q$ and clearly $E_2$ is homotopically equivalent to $0$.  Thus $T_i T_{i+1} T_i$ is isomorphic to $E_1$ in the relative homotopy category.

There is an isomorphism of complexes $E_1 \cong E_3$, via a basis change respecting the $H_q$-structures, given by the diagram

\begin{equation*}
\xymatrix{
& & R \otimes B_{i+1} \otimes B_i \ar[r] \ar[dr]
& R \otimes R \otimes B_i \ar[dr]
& \\
E_1= \ar[dddd] &B_{i,i+1,i} \ar[r] \ar[ur] \ar[dr] \ar[dddd]_{\Id}
& B_i  \ar[ur] \ar[dr]
& R \otimes B_{i+1} \otimes R \ar[r]
& R \ar[dddd]^{\Id} \\
& & B_i \otimes B_{i+1} \otimes R \ar[r] \ar[ur] \ar[dd]_{\begin{pmatrix}
\Id & &  \\
& \Id & \\
 & & \Id \end{pmatrix}}
& B_i \otimes R  \otimes R \ar[ur] \ar[dd]_{\begin{pmatrix}
\Id & & \Id \\
 & \Id & \\
 & & \Id \end{pmatrix}}
& \\
& & & & \\
%%%%%
& & R \otimes B_{i+1} \otimes B_i \ar[r]^{+} \ar[dr]^{-}
& R \otimes R \otimes B_i \ar[dr]^{+}
& \\
E_3:= & B_{i,i+1,i} \ar[r]^-{\phi_1 \circ (1 \otimes br \otimes 1) \circ f_1} \ar[ur]^{(br \otimes 1 \otimes 1) \circ f_1} \ar[dr]^{(1 \otimes 1 \otimes br) \circ f_1} 
& B_i   \ar[dr]^>{(1 \otimes br) \circ \psi_1}
& R \otimes B_{i+1} \otimes R \ar[r]^-{+}
& R \\
& & B_i \otimes B_{i+1} \otimes R \ar[r]_{-} \ar[ur]_>{+} \ar[uur]^{-} 
& B_i \otimes R  \otimes R 
& \\
}
\end{equation*}
There is a short exact sequence of complexes of $(R,R)$-bimodules compatible with the $H_q$-structure which splits when forgetting the $H_q$-structure
\begin{equation*}
\xymatrix{
%%%%%
E_4:= \ar[ddd] & & B_i \ar[r]^{\Id} \ar[dd]_{\begin{pmatrix} 0 \\ \Id \\ 0 \end{pmatrix}} & B_i \ar[dd]^{\begin{pmatrix} 0 \\ 0 \\  \Id \end{pmatrix}} & \\
\\
%%%%%
& & R \otimes B_{i+1} \otimes B_i \ar[r] \ar[dr]
& R \otimes R \otimes B_i \ar[dr]
& \\
E_3= \ar[dddd] & B_{i,i+1,i} \ar[r] \ar[ur] \ar[dr] \ar[dddd]_{\Id}
& B_i   \ar[dr]
& R \otimes B_{i+1} \otimes R \ar[r]
& R \ar[dddd]^{\Id} \\
& & B_i \otimes B_{i+1} \otimes R \ar[r] \ar[ur] \ar[uur] \ar[dd]^{\begin{pmatrix} \Id & 0 & 0 \\ 0 & 0 & \Id \end{pmatrix}}
& B_i \otimes R  \otimes R \ar[dd]^{\begin{pmatrix} \Id & 0 & 0 \\ 0 & \Id & 0 \end{pmatrix}}
& \\
\\
%%%%%%%
& &  B_{i+1} \otimes B_i \ar[r]^{br_{i+1} \otimes \Id} \ar[ddr]^<<<<<<{-\Id \otimes br_i}
&  B_i \ar[dr]^{br_i}
& \\
E_5:= & B_{i,i+1,i}  \ar[ur]_{(br_i \otimes \Id \otimes \Id) \circ f_1} \ar[dr]^{(\Id \otimes \Id \otimes br_i) \circ f_1}
&   
&  
& R \\
& & B_i \otimes B_{i+1} \ar[r]_{br_i \otimes \Id}  \ar[uur]_<<<<<<{-\Id \otimes br_{i+1} } 
& B_{i+1} \ar[ur]^{br_{i+1}}
& 
}
\end{equation*}
Since $E_4$ is contractible, $T_i T_{i+1} T_i \cong E_5$.
In the same exact way, one deduces that $T_{i+1} T_i T_{i+1} \cong E_5$, which proves the proposition.
\end{proof}

\section{HOMFLYPT homology under a differential} \label{sechomflypt}
\subsection{Triply graded theory} \label{sectriple}
In this section we categorify the HOMFLYPT polynomial of any link using analogous arguments from  \cite{Cautisremarks}, \cite{RW} and \cite{Roulink} adapted to the $p$-DG setting. 

For the next definition, we will allow complexes of Soergel bimodules to sit in half-integer degrees in the Hochschild ($a$) and topological ($t$) degrees when considering the usual complexes of vector spaces.
We then modify the elementary braiding complexes of equation \eqref{eqn-elementary-braids} to be
\begin{equation}\label{eqn-elementary-braids-half-grading}
T_i :=
 (at)^{-\frac{1}{2}}q^{-2} \left(t B_i  \xrightarrow{br_i} R\right)
,
\quad \quad \quad
T_i' := (at)^{\frac{1}{2}} q^{2} \left(R \xrightarrow{rb_i} q^{-2} t^{-1} B_i^{-(x_i+x_{i+1})}\right)
.
\end{equation}

Let $\beta\in \mathrm{Br}_n$ be a braid group element on $n$ strands. By Theorem \ref{thm-braid-invariant}, there is a chain complex of $(R_n,R_n)\# H_q$-bimodules $T_\beta$, well defined up to homotopy, associated with $\beta$. Then we write
\begin{equation}\label{eqn-chain-complex-for-braid}
    T_\beta = \left(\dots\stackrel{d_0}{\lra} T_\beta^{i+1} \stackrel{d_0}{\lra} T_\beta^{i} \stackrel{d_0}{\lra} T_\beta^{i-1}\stackrel{d_0}{\lra}\dots\right).
\end{equation}

\begin{defn}\label{def-HHH}
The \emph{untwisted $H_q$-HOMFLYPT homology} of $\beta$ is the object 
\[
\mtHHH^{\dif_q}(\beta):=a^{-\frac{n}{2}}t^{\frac{n}{2}}\mH_\bullet \left(\dots \lra \mHH_\bullet^{\dif_q}(T_\beta^{i+1}) \xrightarrow{d_t} \mHH_\bullet^{\dif_q}(T_\beta^{i}) \xrightarrow{d_t}  \mHH_\bullet^{\dif_q}(T_\beta^{i-1})\lra \dots\right)
\]
in the category of triply graded $H_q$-modules, where $d_t:=\mHH_\bullet^{\dif_q}(d_0)$ is the induced map of $d_0$ on Hochschild homology.  
\end{defn}

By construction, the space $\mtHHH^{\dif_q}(\beta)$ is triply graded by the topological ($t$) degree, the Hochschild ($a$) degree as well as the quantum ($q$) degree. When necessary to emphasize each graded homogeneous piece of the space, we will write
$\mtHHH^{\dif_q}_{i,j,k}(\beta)$ to denote the homogeneous component concentrated in $t$-degree $i$, $a$-degree $j$ and $q$-degree $k$.

The following theorem is a particular case of the main result of \cite{KRWitt}, where we have only kept track of the degree two $p$-nilpotent differential in finite characteristic $p$. The detailed verification given below, however, uses the main ideas of \cite{Roulink} and differs from that of \cite{KRWitt}. This proof serves as the model for the other link homology theories in this paper. 

\begin{thm}\label{thm-untwisted-HOMFLY}
The untwisted $H_q$-HOMFLYPT homology of $\beta$ depends only on the braid closure of $\beta$ as a framed link in $\R^3$.
\end{thm}

As a convention for the framing number of braid closure, if a strand for a component of link is altered as in the left of \eqref{framing}, then we say that the framing of the component is increased by $1$ (with respect to the blackboard framing).
If a strand for a component of link is altered as in the right of \eqref{framing}, then we say that the framing of the component is decreased by $1$.

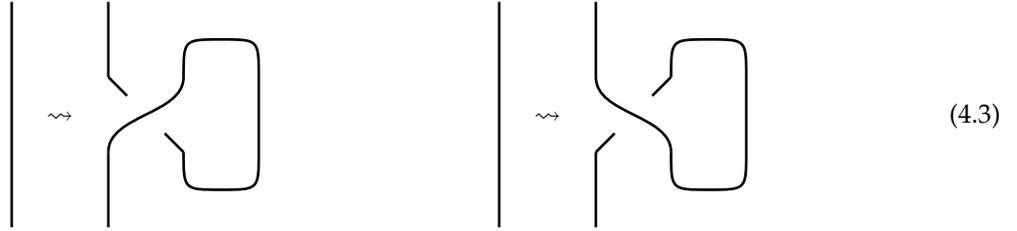
\begin{equation} \label{framing}
\begin{DGCpicture}
\DGCPLstrand(-2,-1)(-2,2)
\end{DGCpicture}
\quad
\rightsquigarrow
\quad
\begin{DGCpicture}
\DGCstrand(0,0)(1,1)
\DGCPLstrand(1,0)(.75,.25)
\DGCPLstrand(0,1)(.25,.75)
\DGCstrand(1,1)(1.5,1.5)(2,1)(2,0)(1.5,-.5)(1,0)
\DGCPLstrand(0,1)(0,2)
\DGCPLstrand(0,0)(0,-1)
\end{DGCpicture}
\quad \quad \quad \quad \quad \quad \quad \quad \quad
\begin{DGCpicture}
\DGCPLstrand(-2,-1)(-2,2)
\end{DGCpicture}
\quad
\rightsquigarrow
\quad
\begin{DGCpicture}
\DGCPLstrand(0,0)(.25,.25)
\DGCPLstrand(.75,.75)(1,1)
\DGCstrand(1,1)(1.5,1.5)(2,1)(2,0)(1.5,-.5)(1,0)
\DGCstrand(1,0)(0,1)
\DGCPLstrand(0,1)(0,2)
\DGCPLstrand(0,0)(0,-1)
\end{DGCpicture}
\end{equation}
Denote by $\mathtt{f}_i(L)$ the framing number of the $i$th strand of a link $L$. Then, under the two Reidemeister moves of \eqref{framing}, $\mathtt{f}_i(L)$ adds or subtracts $1$ respectively when changing from the corresponding left local picture to the right local picture.

Our main goal in this section is to establish this result. Due to Theorem \ref{thm-braid-invariant}, the proof reduces to showing the invariance under the two Markov moves.

\subsection{Doubly graded theory}
We next seek to define the analogue of $\mtHHH$ in the category of $p$-complexes. This construction will serve as a precursor to the finite-dimensional $\mf{sl}_2$-homology theory defined in the next section.

We first would like to define $\pT_\beta$ to be the $p$-complex of Soergel bimodules associated with $\beta$ by
\begin{equation}
    \pT_\beta :=\mc{P}(T_\beta).
\end{equation}
In other words, $\pT_\beta$ should be a $p$-complex of the form
\[
\pT_\beta = \left(
\cdots \lra T_\beta^{2k} \xrightarrow{d_0} T_\beta^{2k-1} =\cdots = T_\beta^{2k-1} \xrightarrow{d_0} T_\beta^{2k-2} \lra \cdots
\right),
\]
where every term in odd topological degree is repeated $p-1$ times.
We will denote the boundary maps in the $p$-extended complex $\pT_\beta$ by $\dif_0$, in contrast to the usual  topological differential $d_0$.

\begin{rem}[Half grading shifts for $p$-complexes]\label{rem-half-grading-shifts}
Here, we point out that, unlike in the ordinary homotopy category of complexes, we do not need to formally introduce half grading shifts in $\mc{C}(\Bbbk,\dif)$ in odd prime characteristic. For instance, one may set
\begin{equation}
a^{[\frac{1}{2}]}:=a^{\frac{p+1}{2}}[-1]_{\dif}^a  ,  \quad \quad
t^{[\frac{1}{2}]}:=t^{\frac{p+1}{2}}[-1]_\dif^t.
\end{equation}
Using $[-1]^a_\dif\circ [-1]^a_\dif=[-2]^{a}_\dif=a^{-p}$, one sees that, as functors, $a^{[\frac{1}{2}]}\circ a^{[\frac{1}{2}]} = a$. Likewise, $t^{[\frac{1}{2}]}\circ t^{[\frac{1}{2}]}=t$.

The same half grading shift functors can also be interpreted as
\begin{equation}
a^{[\frac{1}{2}]}:=a^{\frac{1-p}{2}}[1]_{\dif}^a  ,  \quad \quad
t^{[\frac{1}{2}]}:=t^{\frac{1-p}{2}}[1]_\dif^t.
\end{equation}
The two seemingly different definitions actually agree, as both are given by taking tensor product with the $p$-complex $U_{p-2}\{\frac{p-1}{2}\}$ in the $a$ or $t$ direction.

However, the $p$-extension functor $\mc{P}$ does not intertwine between the ordinary half graded complexes and $p$-complexes, since, if we were to set 
$\mc{P}(a^{\frac{1}{2}})=a^{[\frac{1}{2}]}$, then we would have
\[
\mc{P}([1]_d^a)=\mc{P}(a)=\mc{P}(a^{\frac{1}{2}} \circ a^{\frac{1}{2}})=a^{[\frac{1}{2}]} \circ a^{[\frac{1}{2}]}=a \neq [1]^a_\dif.
\]
Thus we do not naively $p$-extend the braiding complexes \eqref{eqn-elementary-braids-half-grading} via $\mc{P}$.
\end{rem}

We emphasize that the proof of invariance under both Markov II moves of our theory, forces us to collapse the $a$ and $t$ gradings in this construction.
%We emphasize here that, as the proof of Markov II invariance of the theory will show below, we will
%be forced to collapse the $t$ and $a$ gradings in this construction. 
The resulting homology theory will be 
doubly graded. 
More specifically,  let us first collapse the $a$ and $t$ gradings in $\mHHH$ into a single grading satisfying $a=q^2t$, then we will categorically specialize $t=[1]^t_d$ into $[1]^t_\dif$ by $p$-extension. We then lose the $a$-grading below, which is determined by the $t$-grading and $q$-grading. The $t$-degree remains independent of the $q$-degree on Soergel bimodules. 
%For any integer $n$, we set
%\begin{equation}
%    t^{[\frac{n}{2}]}:= ( t^{[\frac{1}{2}]} )^{\circ n}.
%\end{equation}

We then define 
\begin{equation}\label{eqn-elementary-braids-p-ext}
pT_i :=
 q^{-3} \left(B_i \xrightarrow{br_i} R[-1]^t_\dif\right)
,
\quad \quad \quad
pT_i' := q^{3} \left(R[1]^t_\dif \xrightarrow{rb_i} q^{-2}  B_i^{-(x_i+x_{i+1})}\right)
.
\end{equation}
Since $pT_i$ and $pT_i^\prime$ are, up to $t$-grading shifts, obtained by applying $\mc{P}$ to $T_i$ and $T_i^\prime$, it follows that they satisfy the same braid relations, and can be used to define $pT_\beta$ similarly as done in Theorem \ref{thm-braid-invariant}.

\begin{defn}\label{def-pTbeta}
Let $\beta\in \mathrm{Br}_n$ be a braid group element written as a product $\sigma_{i_i}^{\epsilon_1}\cdots \sigma_{i_k}^{\epsilon_k}$ in the elementary generators, where $\epsilon_i\in \{\emptyset, \prime\}$. We assign to $\beta$ the $p$-chain complex of $(R_n,R_n)\# H_q$-bimodules
\begin{equation}
    pT_\beta:=pT_{i_1}^{\epsilon_1}\otimes_R\cdots \otimes_R pT_{i_k}^{\epsilon_k}.
\end{equation}
\end{defn}

The boundary maps in the $t$-direction will be denoted $\dif_0$ for any $pT_\beta$.

\begin{defn}\label{def-pHHH}
The \emph{ untwisted $H_q$-HOMFLYPT $p$-homology } of $\beta$ is the object 
\[
 \ptHHH^{\dif_q}(\beta):= q^{-n} \mH^/_{\bullet} \left(\dots \lra \pHH_\bullet^{\dif_q}(\pT_\beta^{i+1}) \xrightarrow{\dif_t}  \pHH_\bullet^{\dif_q}(\pT_\beta^{i}) \xrightarrow{\dif_t}  \pHH_\bullet^{\dif_q}(\pT_\beta^{i-1})\lra \dots\right)
\]
in the homotopy category of bigraded $H_q$-modules. Here $\dif_t$ stands for the induced map of the topological differentials on $p$-Hochschild homology groups $\dif_t:=\pHH_\bullet^{\dif_q}(\dif_0)$.
\end{defn}

In the definition of the $H_q$-HOMFLYPT $p$-homology, we have applied the $p$-extensions in both the topological and the Hochschild directions so that they can be collapsed into a single degree. The reason will become clearer later when categorifying the Jones polynomial at roots of unity. Therefore, in contrast to $\mtHHH(\beta)$, $\ptHHH(\beta)$ is only doubly graded, and we will adopt the notation $\ptHHH_{i,j}(\beta)$ as above to stand for its homogeneous components in topological degree $i$ and $q$ degree $j$. Further, the overall grading shift in the definition will be utilized in the invariance under the Markov II moves below.

\begin{thm}\label{thm-untwisted-HOMFLY-p-ext}
The untwisted $H_q$-HOMFLYPT $p$-homology of $\beta$ depends only on the braid closure of $\beta$ as a framed link in $\R^3$.
\end{thm}

The proof of Theorems \ref{thm-untwisted-HOMFLY} and \ref{thm-untwisted-HOMFLY-p-ext} will occupy the next few subsections, after we introduce the $H_q$-equivariant ($p$-)Koszul resolutions.

\subsection{Koszul complexes}
We recall here the Koszul resolution $C_n$ of the polynomial algebra $R_n$ as well as its adaption in the $p$-DG setting.  In the one variable case, we have a short exact sequence of $\Bbbk[x]$-bimodules:
\begin{equation}
    0 \lra q^2\Bbbk[x]\otimes \Bbbk[x] \xrightarrow{x\otimes 1-1\otimes x}\Bbbk[x]\otimes \Bbbk[x] \stackrel{m}{\lra} \Bbbk[x]\lra 0.
\end{equation}
In order to make the maps $H_q$-equivariant, we twist the $H_q$-action on the leftmost bimodule:
\begin{equation}
    0 \lra q^2\Bbbk[x]^x\otimes \Bbbk[x]^x \xrightarrow{x\otimes 1-1\otimes x}\Bbbk[x]\otimes \Bbbk[x] \stackrel{m}{\lra} \Bbbk[x]\lra 0.
\end{equation}

In the usual homotopy category of $(\Bbbk[x], \Bbbk[x])\# H_q$-modules, we have then an $H_q$-equivariant relative replacement of $\Bbbk[x]$
\begin{equation}
    0\lra q^2a\Bbbk[x]^x\otimes \Bbbk[x]^x \xrightarrow{x\otimes 1-1\otimes x}\Bbbk[x]\otimes \Bbbk[x]\lra 0 
\end{equation} 
where we have inserted $a$ in the leftmost nonzero term to emphasize the homological degree that it sits in.
Denote by $(C_1,d)$ this $H_q$-equivariant Koszul resolution in the one-variable case.  

In the $p$-homotopy category, a relative replacement is then obtained by applying the $p$-extension functor $\mc{P}$ to $C_1$: 
\begin{equation}
     0\lra q^2a^{p-1}\Bbbk[x]^x\otimes \Bbbk[x]^x= \dots =  q^2a\Bbbk[x]^x\otimes \Bbbk[x]^x \xrightarrow{x\otimes 1-1\otimes x}\Bbbk[x]\otimes \Bbbk[x]\lra 0 
\end{equation}
We abbreviate this $p$-resolution by $\pC_1$.

For the ease of notation, it will be useful to denote both replacements by
\[
0\lra q^2\Bbbk[x]^x\otimes \Bbbk[x]^x [1]^a_{*}\xrightarrow{x\otimes 1-1\otimes x}\Bbbk[x]\otimes \Bbbk[x]\lra 0 
\]
where $*\in \{d, \dif\}$ will label $[1]$ as either the usual homological or  $p$-homological shift in $a$-degree.

For the polynomial ring $R_n=\Bbbk[x_1,\dots, x_n]$, one takes the $n$-fold tensor product over $\Bbbk$ of the bimodule resolution $C_1$ of $\Bbbk[x_1]$ to get $(C_n=C_1^{\otimes n},d)$. We also define
\begin{equation}
  (\pC_n,\dif):=\pC_1\otimes_\Bbbk \cdots \otimes_\Bbbk\pC_1 
\end{equation}
as the $n$-fold tensor product of the one-variable resolution. Note that $\pC_n$ is homotopic to, but bigger than, the $p$-complex of bimodules $\mc{P}(C_n)$.

\begin{example}
In characteristic $3$, the $3$-complex $3C_2$ is the total complex of the cube
\[
\xymatrix{
\Bbbk[x,y]^{x}\otimes \Bbbk[x,y]^{x} \ar[r]^{=} & \Bbbk[x,y]^{x}\otimes \Bbbk[x,y]^{x}\ar[rr]^-{x\otimes 1-1\otimes x} && \Bbbk[x,y]\otimes \Bbbk[x,y]\\
\Bbbk[x,y]^{x+y}\otimes \Bbbk[x,y]^{x+y} \ar[r]^{=}\ar[u]_-{y\otimes1-1\otimes y} & \Bbbk[x,y]^{x+y}\otimes \Bbbk[x,y]^{x+y}\ar[rr]^-{x\otimes 1-1\otimes x} \ar[u]_-{y\otimes1-1\otimes y} && \Bbbk[x,y]^y\otimes \Bbbk[x,y]^y\ar[u]^{y\otimes1-1\otimes y}\\
\Bbbk[x,y]^{x+y}\otimes \Bbbk[x,y]^{x+y} \ar[r]^{=}\ar[u]_{=} & \Bbbk[x,y]^{x+y}\otimes \Bbbk[x,y]^{x+y}\ar[rr]^-{x\otimes 1-1\otimes x} \ar[u]_{=} && \Bbbk[x,y]^y\otimes \Bbbk[x,y]^y\ar[u]^{=}
}
\]
Under the total $p$-differential, the copy of $\Bbbk[x,y]^{x+y}\otimes \Bbbk[x,y]^{x+y}$ sitting in the southwest corner generates an acyclic $3$-subcomplex, modulo which one obtains the total $3$-complex of $\mc{P}(C_2)$:
\[
\xymatrix{
& \Bbbk[x,y]^{x}\otimes \Bbbk[x,y]^x \ar[r]^= &\Bbbk[x,y]^{x}\otimes \Bbbk[x,y]^x \ar[dr]^{x\otimes 1-1\otimes x} & \\
\Bbbk[x,y]^{x+y}\otimes \Bbbk[x,y]^{x+y} \ar[ur]^{-y\otimes 1+1\otimes y} \ar[dr]_{x\otimes 1-1\otimes x} & & & \Bbbk[x,y]\otimes \Bbbk[x,y]\\
& \Bbbk[x,y]^{y}\otimes \Bbbk[x,y]^y \ar[r]^=& \Bbbk[x,y]^{y}\otimes \Bbbk[x,y]^y \ar[ur]_{y\otimes 1-1\otimes y}& 
}
\]
\end{example}

Using this Koszul resolution, one immediately obtains, via the following result, that the relative $p$-Hochschild homology is completely determined by the classical relative Hochschild homology. Therefore, it may appear that one does not gain any more information by introducing $\pHHH_\bullet^\dif$. However, this construction is essential for the collapsed homology theory to be defined in the next section.

\begin{prop}\label{prop-pHH-determined-by-HH}
Let $M$ be an $(R_n,R_n) \#H_q$-bimodule. Then $\pHH_\bullet^{\dif_q}(M)$ is determined by $\mHH_\bullet^{\dif_q}(M)$ by
\begin{gather}
    \pHH^{\dif_q}_i (M) = 
    \begin{cases}
     \mHH_{2k-1}^{\dif_q}(M) &  (k-1)p+1 \leq i \leq kp-1\\
     \mHH_{2k}^{\dif_q}(M) & i=kp
    \end{cases}
\end{gather}
\end{prop}
\begin{proof}
Using the Koszul resolution, we have\footnote{See the remark below for more explanation of the slash homology computation.} 
\[
\pHH^\dif_\bullet(M) = \mH^/_{\bullet}(M\otimes_{R_n^{\rm en}} \mc{P}(C_n))= \mH^/_{\bullet}(\mc{P}(M\otimes_{R_n^{\rm en}} C_n))=\mc{P}(\mH_\bullet (M\otimes_{R_n^{\rm en}}C_n)).
\]
The result follows.
\end{proof}

\begin{rem}\label{rem-ususal-homology-determines-p-homology}
The result is true in more generality. If $A$ is an (ungraded) algebra equipped with the zero $p$-differential, then its $p$-Hochschild homology is entirely determined by its usual Hochschild homology. This result is essentially due to Spanier \cite{Spanier}, but is also proved in more generality by \cite{KWa}.

Indeed, if $P_\bullet \lra A$ is any projective resolution of $A$ over $A\otimes A^{\mathrm{op}}$, then $\mc{P}(P_\bullet)$ provides a $p$-resolution of $A$, and
\begin{equation}
   \pHH_\bullet(A)=\mH^/_{\bullet} (\mc{P}(P_\bullet)\otimes_{A\otimes A^{\mathrm{op}}}A)=\mH^/_{\bullet}(\mc{P}(P_\bullet\otimes_{A\otimes A^{\mathrm{op}}}A)).
\end{equation}
When computing the last slash homology, one may safely forget about the module structures involved in $\mc{P}(P_\bullet\otimes_{A\otimes A^{\mathrm{op}}}A)$ and think of it as a direct sum, possibly infinite copies, of chain-complexes of the form 
\begin{equation}
    0\lra \underline{\Bbbk} \lra 0,\quad \quad  \quad 0\lra \underline{\Bbbk}\lra \Bbbk \lra 0,
\end{equation}
where the underlined term sits in some homological degree $i$.
Under the $p$-extension functor, the last complex extends to a contractible $p$-complex, while the first complex becomes
\[
0\lra \Bbbk \lra 0,
\]
if $i$ is even, or the $(p-1)$-dimensional
\[
0\lra \Bbbk=\cdots = \Bbbk\lra 0
\]
if $i$ is odd. The slash homology computation then follows.
\end{rem}

%A tensor product of $C_1$ gives a resolution $C_n$ of $R_n$.
%\JS{We already defined and used $C_n$.  Is the restatement here supposed to be connected the definition below?} \QY{These two paragraphs were left over from earlier editions. I think we can remove them now.}

%The $p$-HOMFLYPT homology of the closure of a braid $\beta$ is defined to be
%\begin{equation*}
%\pHH\mH(\beta) := (a^{-1} t^{-1})^{\frac{n}{2}} %\mH_*(\pHH(R_n,\beta)) .  
%\end{equation*}

%\JS{We did define the homology earlier.  Do we want to emphasize here that it could now be computed via $C_n$? Or is this supposed to be something distinct- where we don't consider the relative version?}

\subsection{Markov I} \label{secmarkov1}
The usual HOMFLYPT homologies of  two braid compositions $\beta_1\beta_2$ and $\beta_2\beta_1$ are isomorphic due to the trace-like property of the usual Hochschild homology functor. The relative Hochschild homology also remembers the $H_q$-action.

\begin{prop}
Let $\beta_1$ and $\beta_2$ be two braids on $n$ strands.
Then $\mtHHH^{\dif_q}(\beta_1 \beta_2) \cong \mtHHH^{\dif_q}(\beta_2 \beta_1)$. \qedhere
\end{prop}

The same property also holds for the HOMFLYPT $pH_q$-homology groups.

\begin{prop}
Let $\beta_1$ and $\beta_2$ be two braids on $n$ strands.
Then $\ptHHH^{\dif_q}(\beta_1 \beta_2) \cong \ptHHH^{\dif_q}(\beta_2 \beta_1)$.
\end{prop}

\begin{proof}
This follows from Proposition \ref{HHrelativecyclprop}, since we have the functorial isomorphism
\[
\pHH^{\dif_q}_\bullet(\pT_{\beta_1}^i\otimes_{R_n} \pT_{\beta_2}^i)\cong\pHH^{\dif_q}_\bullet(\pT_{\beta_2}^i\otimes_{R_n} \pT_{\beta_1}^i)
\]
for all $i\in \Z$. Alternatively, this follows from combining the previous result with Proposition \ref{prop-pHH-determined-by-HH}.
\end{proof}

\subsection{Markov II} \label{secmarkov2}
In order to prove the second Markov move, one needs to show that for a (complex of) Soergel bimodules $M$ over the polynomial $p$-DG algebra $R_n$, that HOMFLYPT ($p$)$H_q$-homologies of the bimodules \eqref{markov2pic} are isomorphic (up to shifts and twists).
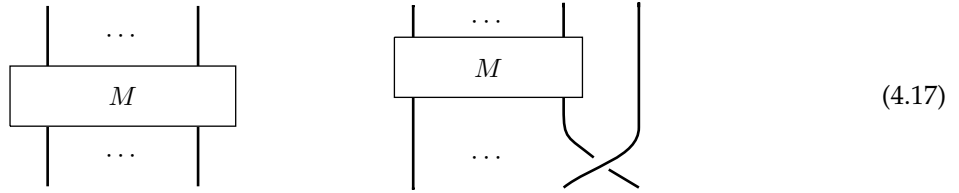
\begin{equation}
\label{markov2pic}
\begin{DGCpicture}[scale={1,0.8}]
\DGCstrand(0,0)(0,3)
\DGCstrand(2,0)(2,3)
\DGCcoupon(-0.5,1)(2.5,2){$M$}
\DGCcoupon*(0,0)(2,1){$\cdots$}
\DGCcoupon*(0,2)(2,3){$\cdots$}
%\DGCPLstrand(-1,0.5)(-1,-1)
%\DGCPLstrand(0,.5)(0,-1)
%\DGCcoupon(-1.25,-0.5)(.25,0){$M$}
%\DGCcoupon*(-1,-1)(0,-.5){$\cdots$}
%\DGCcoupon*(-1,0)(0,.5){$\cdots$}
%%%%BENEATH IS CLOSURE%%%%
%\DGCstrand(0,.5)/u/(1,.5)/d/
%\DGCstrand(-1,.5)/u/(2,.5)/d/
%\DGCstrand(0,-1)/d/(1,-1)/u/
%\DGCstrand(-1,-1)/d/(2,-1)/u/
%\DGCPLstrand(1,.5)(1,-1)
%\DGCPLstrand(2,.5)(2,-1)
\end{DGCpicture}
\quad \quad \quad \quad \quad \quad 
\begin{DGCpicture}[scale={1,0.8}]
\DGCstrand(0,0)(0,-3)
\DGCstrand(2,0)(2,-1.7)(2.4,-2.5)/dr/
\DGCstrand(3,0)(3,-2)(2,-3)/dl/
\DGCPLstrand(2.6,-2.7)(3,-3)
\DGCcoupon(-0.25,-0.5)(2.25,-1.5){$M$}
\DGCcoupon*(0,0)(2,-0.5){$\cdots$}
\DGCcoupon*(0,-2)(2,-3){$\cdots$}
\end{DGCpicture}
%\begin{DGCpicture}
%\DGCstrand(0,0)(0.75,0.75)
%\DGCPLstrand(-1,2)(-1,-2)
%\DGCPLstrand(.75,.25)(1,0)(1,-1)
%\DGCPLstrand(0.75,0)(.25,.75)
%\DGCPLstrand(0,1)(0,2)
%\DGCPLstrand(0,0)(0,-2)
%\DGCcoupon(-1.25,-0.5)(.25,0){$M$}
%\DGCcoupon*(-1,-1)(0,-.5){$\cdots$}
%\DGCcoupon*(-1,1)(0,2){$\cdots$}
%\DGCstrand(1,-1)/d/(2,-1)/u/
%\DGCstrand(2,-1)(2,1)
%\DGCstrand(1,1)/u/(2,1)/d/
%\DGCstrand(0,2)/u/(3,2)/d/
%\DGCstrand(0,-2)/d/(3,-2)/u/
%\DGCPLstrand(3,2)(3,-2)
%\DGCstrand(-1,2)/u/(4,2)/d/
%\DGCstrand(-1,-2)/d/(4,-2)/u/
%\DGCPLstrand(4,2)(4,-2)
%\end{DGCpicture}
\end{equation}
Let $\Lambda \langle x_{n+1} \rangle$ be the exterior algebra in the variable $x_{n+1}$.  Recall that $R_n=\Bbbk[x_1,\ldots,x_n]$ and let
$M \in (R_n,R_n) \#H_q \dmod$. Set $C_1^\prime=\Bbbk[x_{n+1}] \otimes \Lambda \langle x_{n+1} \rangle \otimes \Bbbk[x_{n+1}] \cong C_1$.
Letting $C_n$ denote the Koszul resolution of $R_n$, we have inductively that $C_{n+1}=C_n \otimes C_1^\prime$.

As in the proof of Theorem \ref{HHcyclprop}, the Hochschild homology of $M$ is depicted by the closure diagram
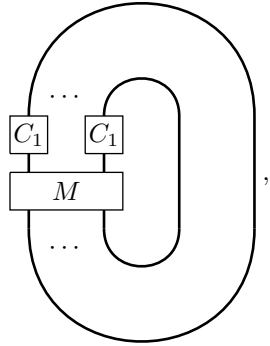
\[
\begin{DGCpicture}
\DGCPLstrand(-1,0.5)(-1,-1)
\DGCPLstrand(0,.5)(0,-1)
\DGCcoupon(-1.25,-0.25)(.25,-0.75){$M$}
\DGCcoupon(-1.25,0)(-0.75,0.5){$C_1$}
\DGCcoupon(-0.25,0)(0.25,0.5){$C_1$}
\DGCcoupon*(-1,-1.75)(0,-.75){$\cdots$}
\DGCcoupon*(-1,0.5)(0,1){$\cdots$}
%%%%BENEATH IS CLOSURE%%%%
\DGCstrand(0,.5)/u/(1,.5)/d/
\DGCstrand(-1,.5)/u/(2,.5)/d/
\DGCstrand(0,-1)/d/(1,-1)/u/
\DGCstrand(-1,-1)/d/(2,-1)/u/
\DGCPLstrand(1,.5)(1,-1)
\DGCPLstrand(2,.5)(2,-1)
\end{DGCpicture} \ ,
\]
where the single strands connecting the boxes indicate tensor products over the one-variable polynomial rings labelling those strands.

The proof of second Markov move essentially reduces to a computation of the partial Hochschild homology with respect to the last variable $x_{n+1}$.  This operation is diagramatically represented in \eqref{picpartialtrace}.
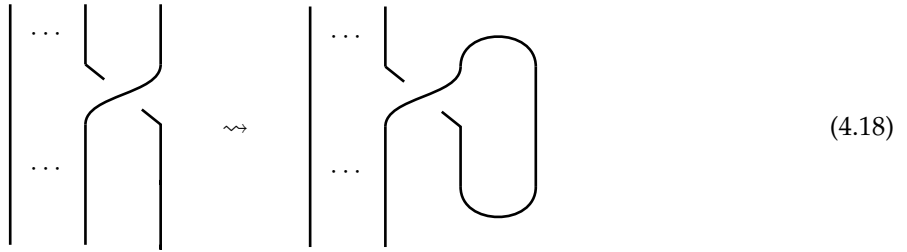
\begin{equation}
\label{picpartialtrace}
\begin{DGCpicture}[scale={1,0.8}]
\DGCstrand(0,0)(1,1)
\DGCPLstrand(-1,2)(-1,-2)
\DGCPLstrand(.75,.25)(1,0)(1,-1)
\DGCPLstrand(0,1)(.25,.75)
\DGCPLstrand(0,1)(0,2)
\DGCPLstrand(0,0)(0,-2)
\DGCcoupon*(-1,-1)(0,-.5){$\cdots$}
\DGCcoupon*(-1,1)(0,2){$\cdots$}
\DGCstrand(1,-1)(1,-2)
\DGCstrand(1,1)(1,2)
\end{DGCpicture}
\quad \quad
\rightsquigarrow
\quad \quad
\begin{DGCpicture}[scale={1,0.8}]
\DGCstrand(4,0)(5,1)
\DGCPLstrand(3,2)(3,-2)
\DGCPLstrand(4.75,.25)(5,0)(5,-1)
\DGCPLstrand(4,1)(4.25,.75)
\DGCPLstrand(4,1)(4,2)
\DGCPLstrand(4,0)(4,-2)
\DGCcoupon*(3,-1)(4,-.5){$\cdots$}
\DGCcoupon*(3,1)(4,2){$\cdots$}
\DGCstrand(5,-1)/d/(6,-1)/u/
\DGCstrand(6,-1)(6,1)
\DGCstrand(5,1)/u/(6,1)/d/
\end{DGCpicture}
\end{equation}
This leads to an analysis of 
$C_1^\prime \otimes_{\Bbbk[x_{n+1}]^{\rm en}} T_n$ in Proposition \ref{prop-Markov-II-for-HHH}.

The following technical result will be the heart of establishing the invariance under the Markov II moves. We start with the usual HOMFLYPT homology case under the Hopf algebra $H_q$-action.

%\begin{equation}
%\begin{DGCpicture}
%\DGCstrand(0,0)(1,1)
%\DGCPLstrand(-1,2)(-1,-2)
%\DGCPLstrand(.75,.25)(1,0)(1,-1)
%\DGCPLstrand(0,1)(.25,.75)
%\DGCPLstrand(0,1)(0,2)
%\DGCPLstrand(0,0)(0,-2)
%\DGCcoupon(-1.25,-0.5)(.25,0){$M$}
%\DGCcoupon*(-1,-1)(0,-.5){$\cdots$}
%\DGCcoupon*(-1,1)(0,2){$\cdots$}
%\DGCstrand(1,-1)/d/(2,-1)/u/
%\DGCstrand(2,-1)(2,1)
%\DGCstrand(1,1)/u/(2,1)/d/
%\end{DGCpicture}
%\end{equation}

\begin{prop}\label{prop-Markov-II-for-HHH}
Let $\beta$ be a braid with $n$ strands which is assigned a usual complex of Soergel bimodules $M$.  Then there is an $H_q$-equivariant isomorphism of the HOMFLYPT homology groups
\begin{enumerate}
    \item[(i)] $\mtHHH^{\dif_q}((M\otimes \Bbbk[x_{n+1}]) \otimes_{R_{n+1}} T_n) \cong \mtHHH^{\dif_q}(M)^{2x_n}$,
    \item[(ii)] $\mtHHH^{\dif_q}((M\otimes \Bbbk[x_{n+1}])\otimes_{R_{n+1}} T_n') \cong \mtHHH^{\dif_q}(M)^{-2x_n}$,
\end{enumerate}
where $\mtHHH^{\dif_q}(M)^{\pm 2x_n}$ denotes a twisting in the $H_q$-action. 
\end{prop}

\begin{proof}
Both identities are proved in a similar way.  For the first statement, we note that by definition
\begin{equation*}
\mHH^{\dif_q}_\bullet((M\otimes \Bbbk[x_{n+1}])\otimes_{R_{n+1}} T_n) =
\mH_\bullet^v(C_{n+1} \otimes_{R_{n+1}^{\rm en}} ((M \otimes \Bbbk[x_{n+1}])\otimes_{R_{n+1}} T_n)).
\end{equation*}
where the (vertical) homology $\mH_\bullet^v$ above is taken with respect to the differential coming from the  Koszul complex $C_{n+1}$. 

Note that
\begin{align}\label{eqn-bimod-iso-tensor-with-Tn}
C_{n+1} \otimes_{R_{n+1}^{\rm en}} ((M \otimes \Bbbk[x_{n+1}]) \otimes_{R_{n+1}} T_n) &=
(C_n \otimes C_1^\prime) \otimes_{R_{n+1}^{\rm en}} ((M \otimes \Bbbk[x_{n+1}]) \otimes_{R_{n+1}} T_n) \nonumber \\
& \cong 
C_n \otimes_{R_n^{\rm en}} (M \otimes_{R_n} (C_1^\prime \otimes_{\Bbbk[x_{n+1}]^{\rm en}} T_n)).
\end{align}
These isomorphisms, in terms of diagrammatics, can be interpreted as taking closures of the following diagrammatic equalities:
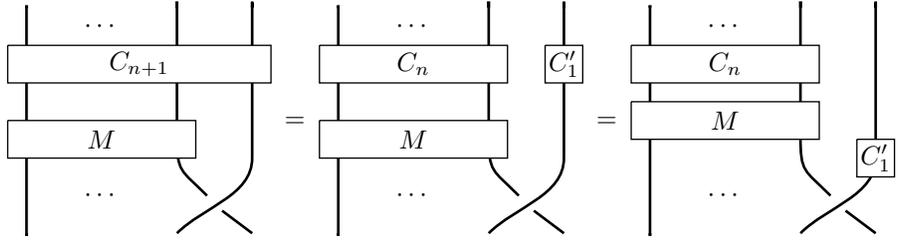
\[
\begin{DGCpicture}
\DGCstrand(0,0)(0,-3)
\DGCstrand(2,0)(2,-1.7)(2.4,-2.5)/dr/
\DGCstrand(3,0)(3,-2)(2,-3)/dl/
\DGCPLstrand(2.6,-2.7)(3,-3)
\DGCcoupon(-0.25,-1.5)(2.25,-2){$M$}
\DGCcoupon(-0.25,-0.5)(3.25,-1){$C_{n+1}$}
\DGCcoupon*(0,0)(2,-0.5){$\cdots$}
\DGCcoupon*(0,-2)(2,-3){$\cdots$}
\end{DGCpicture}
~=~
\begin{DGCpicture}
\DGCstrand(0,0)(0,-3)
\DGCstrand(2,0)(2,-1.7)(2.4,-2.5)/dr/
\DGCstrand(3,0)(3,-2)(2,-3)/dl/
\DGCPLstrand(2.6,-2.7)(3,-3)
\DGCcoupon(-0.25,-1.5)(2.25,-2){$M$}
\DGCcoupon(-0.25,-0.5)(2.25,-1){$C_{n}$}
\DGCcoupon(2.75,-0.5)(3.25,-1){$C_{1}^\prime$}
\DGCcoupon*(0,0)(2,-0.5){$\cdots$}
\DGCcoupon*(0,-2)(2,-3){$\cdots$}
\end{DGCpicture}
~=~
\begin{DGCpicture}
\DGCstrand(0,0)(0,-3)
\DGCstrand(2,0)(2,-1.7)(2.4,-2.5)/dr/
\DGCstrand(3,0)(3,-2)(2,-3)/dl/
\DGCPLstrand(2.6,-2.7)(3,-3)
\DGCcoupon(-0.25,-1.25)(2.25,-1.75){$M$}
\DGCcoupon(-0.25,-0.5)(2.25,-1){$C_{n}$}
\DGCcoupon(2.75,-1.75)(3.25,-2.25){$C_{1}^\prime$}
\DGCcoupon*(0,0)(2,-0.5){$\cdots$}
\DGCcoupon*(0,-2)(2,-3){$\cdots$}
\end{DGCpicture}
\ .
\]

Observe that $C_1^\prime \otimes_{\Bbbk[x_{n+1}]^{\rm en}} T_n$ is a bicomplex of $(R_{n+1},R_{n+1})$-bimodules
\begin{equation*}
\xymatrix{
a^{\frac{1}{2}}t^{\frac{1}{2}}\left({}^{x_{n+1}}B_n^{x_{n+1}}\right) \ar[rr]^-{br} \ar[d]_{x_{n+1} \otimes 1 - 1 \otimes x_{n+1}} && a^{\frac{1}{2}}t^{-\frac{1}{2}}R_{n+1}^{2x_{n+1}} \ar[d]^{0} \\
a^{-\frac{1}{2}}t^{\frac{1}{2}}q^{-2}B_n \ar[rr]^-{br} && a^{-\frac{1}{2}}t^{-\frac{1}{2}}q^{-2}R_{n+1}.
}
\end{equation*}
Here the grading shift conventions follow from equation \eqref{eqn-elementary-braids}. For ease of notation, we will mostly ignore them within this proof below.

It follows that there is a short exact sequence of bicomplexes of $(R_{n+1},R_{n+1})$-bimodules
\begin{equation*}
0 \longrightarrow Y_1 \longrightarrow C_1^\prime \otimes_{\Bbbk[x_{n+1}]^{\rm en}} T_n
\longrightarrow Y_2 \longrightarrow 0
\end{equation*}
where the terms of the sequence are defined by

\begin{equation} \label{sesbicomplexes}
\begin{gathered}
\xymatrix@R=1em@C=3em{
& & & & & 0 \ar[dd] \\
R_{n+1}^{x_n+3x_{n+1}} \ar@/_2pc/[dddd]^(0.75){(x_{n+1}-x_n) \otimes 1 + 1 \otimes (x_{n+1}-x_n)} \ar[rrr]^{2(x_{n+1}-x_n)} \ar[dd] & & & R_{n+1}^{2x_{n+1}} \ar[dd] \ar@/^2pc/[dddd]^{\Id} &  \\
& & & & := & Y_1 \ar[dddd]\\
0 \ar[rrr] & & & 0 \\ 
& & &  \\
{}^{x_{n+1}}B_n^{x_{n+1}} \ar[rrr]^{br} \ar[dd]^{x_{n+1} \otimes 1 - 1 \otimes x_{n+1}} \ar@/_2pc/[dddd]_{\tilde{br}} & & & R_{n+1}^{2x_{n+1}} \ar[dd]^{0} &  & \\
& & &  & = &
C_1^\prime \otimes_{\Bbbk[x_{n+1}]^{\rm en}} T_n \ .\ar[dddd] \\
B_n \ar[rrr]^{br} \ar@/_2pc/[dddd]_{2 \Id} & & & R_{n+1} \ar@/^2pc/[dddd]^{2\Id} \\
&  &  & \\
\widetilde{R}_{n+1}^{x_n+x_{n+1}} \ar[rrr] \ar[dd]^{(x_{n+1}-x_n) \otimes 1 - 1 \otimes (x_{n+1}-x_n)} & & & 0 \ar[dd] &  & \\
& & & & := & Y_2 \ar[dd] \\
B_n \ar[rrr]^{br} & & & R_{n+1} \\
& & & & & 0  
}
\end{gathered}
\end{equation}
Here $\widetilde{R}_{n+1}$ is equal to $R_{n+1}$ as a left $R_{n+1}$-module but the right action of $R_{n+1}$ is twisted by the permutation $\sigma_n \in S_{n+1}$ and $\tilde{br}(a \otimes b)=br(a \sigma_n(b))$.
It is a straightforward exercise to check that all maps above are equivariant with respect to the $H_q$-action.  We show it for the map
\begin{equation*}
\phi := (x_{n+1}-x_n) \otimes 1 + 1 \otimes (x_{n+1}-x_n) \colon     
R_{n+1}^{x_n+3x_{n+1}} \longrightarrow {}^{x_{n+1}}B_n^{x_{n+1}}.
\end{equation*}
One calculates
\begin{align*}
\phi(\partial_q(1))
 & = \phi(x_n+x_{n+1}+2x_{n+1})
 \\
 & =
(x_{n+1}^2-x_n^2) \otimes 1 + 1 \otimes (x_{n+1}^2-x_n^2)
+2x_{n+1}(x_{n+1} -x_n) \otimes 1
+2x_{n+1} \otimes (x_{n+1}-x_n),
\end{align*}
and
\begin{align*}
\partial_q(\phi(1)) = & ~ \dif((x_{n+1}-x_n)\otimes 1 + 1\otimes (x_{n+1}-x_n))  \\
 = & ~(x_{n+1}^2-x_n^2) \otimes 1 + 1 \otimes (x_{n+1}^2-x_n^2) 
+ x_{n+1}(x_{n+1}-x_n) \otimes 1 + \\
& ~(x_{n+1}-x_n) \otimes x_{n+1}
+ x_{n+1} \otimes (x_{n+1}-x_n)
+ 1 \otimes (x_{n+1}-x_n)x_{n+1}.
\end{align*}
Comparing the terms of $\phi(\partial_q(1))$ and $\partial_q(\phi(1))$ it suffices to check the identity inside the bimodule $B_n$
\begin{equation*}
x_{n+1}^2 \otimes 1 - x_{n+1} \otimes x_n = 1 \otimes x_{n+1}^2 - x_n \otimes x_{n+1}.    
\end{equation*}
Rearranging the terms of the above equation, we must show
\begin{equation} \label{needtocheck1}
x_{n+1}^2 \otimes 1 + x_{n} \otimes x_{n+1} = 1 \otimes x_{n+1}^2 + x_{n+1} \otimes x_{n}.    
\end{equation}
Adding $x_n x_{n+1} \otimes 1$ to both sides of \eqref{needtocheck1} and using the fact that symmetric functions in $x_n$ and $x_{n+1}$ may be brought through a tensor product finishes the proof that $\phi(\partial_q(1))=\partial_q(\phi(1))$.
%Clearly $\partial(\phi(1))=0$.  On the other hand,
%\begin{align*}
%\partial(\phi(1)) = & \partial((x_{n+1}-x_n) \otimes 1 -1 \otimes %(x_{n+1}-x_n)) \\
%=& (x_{n+1}^2-x_n^2) \otimes 1 - 1 \otimes (x_{n+1}^2-x_n^2)
%-x_{n+1}(x_{n+1}-x_n) \otimes 1 - (x_{n+1}-x_n) \otimes x_{n+1} \\
% & +x_{n+1} \otimes (x_{n+1}-x_n) +1 \otimes (x_{n+1}-x_n)x_{n+1} \\
% = &  -x_n^2 \otimes 1 +1 \otimes x_n^2 + x_n \otimes x_{n+1} - x_{n+1} %\otimes x_n \\
% = & -x_n^2 \otimes 1 - x_n x_{n+1} \otimes 1 + x_n x_{n+1} \otimes 1
% + 1 \otimes x_n^2 + x_n \otimes x_{n+1} - x_{n+1} \otimes x_n \\
%  = & -x_n \otimes x_n - x_n \otimes x_{n+1} + 1 \otimes x_n x_{n+1} 
% + 1 \otimes x_n^2 + x_n \otimes x_{n+1} - x_{n+1} \otimes x_n \\
%  = & -x_n \otimes x_n  + 1 \otimes x_n x_{n+1} 
% + 1 \otimes x_n^2  - x_{n+1} \otimes x_n \\
%   = & -x_n \otimes x_n  + x_n \otimes x_n 
% + x_{n+1} \otimes x_n  - x_{n+1} \otimes x_n \\
% = & 0.
%\end{align*}

There is a splitting of the short exact sequence \eqref{sesbicomplexes} regarded as a short exact sequence of $(R_n,R_n)$-bimodules, given by
\begin{equation} \label{sesbicomplexessplitting}
\begin{gathered}
\xymatrix{
{}^{x_{n+1}}B_n^{x_{n+1}} \ar[rrr]^{} \ar[d]^{}  & & & R_{n+1}^{2x_{n+1}} \ar[d]^{} \\
B_n \ar[rrr]^{}  & & & R_{n+1}  \\
&  & & \\
\tilde{R}_{n+1}^{x_n+x_{n+1}} \ar[rrr] \ar[d]^{} \ar@/^3pc/[uuu]^{\theta} & & & 0 \ar[d]  \\
B_n \ar[rrr]^{} \ar@/^3pc/[uuu]^{\frac{1}{2} \Id} & & & R_{n+1} \ar@/_2pc/[uuu]_{\frac{1}{2} \Id} \\
}
\end{gathered}
\end{equation}
where 
\begin{equation*}
\theta(f(x_1,\ldots,x_{n-1})x_n^i x_{n+1}^j)
=f(x_1,\ldots,x_{n-1}) x_n^i \otimes x_{n}^j.
\end{equation*}
We briefly explain why $\theta$ is a well-defined bimodule homomorphism. By definition 
$\theta(x_n^i x_{n+1}^j)=x_n^i \otimes x_{n}^j$.
Note that $\theta(x_n^i x_{n+1}^j)=\theta(x_n^i \cdot x_{n+1}^j)=x_n^i \theta(x_{n+1}^j)=x_n^i \otimes x_{n}^j $ where we viewed $x_n^i$ as acting on the left of $x_{n+1}^j$.
Similarly, 
$\theta(x_n^i x_{n+1}^j)=\theta(x_n^i \cdot x_{n}^j)=\theta(x_n^i)x_n^j=
x_n^i \otimes x_{n}^j $ where we viewed $x_n^j$ as acting on the right of $x_{n}^i$.

The short exact sequence \eqref{sesbicomplexes} plugged back into \eqref{eqn-bimod-iso-tensor-with-Tn} gives us a short exact sequence
\begin{equation}\label{eqn-split-exact-sequence}
   0\to  C_n \otimes_{R_n^{\rm en}} (M \otimes_{R_n} Y_1) \to C_n \otimes_{R_n^{\rm en}}(M \otimes \Bbbk[x_{n+1}])\otimes_{R_{n+1}} T_n \to
C_n \otimes_{R_n^{\rm en}} (M \otimes_{R_n} Y_2) \to 0,
\end{equation}
which is split as a sequence of $(R_n,R_n)$-bimodules. Taking homology with respect to the vertical differentials gives rise to a long exact sequence 
\begin{equation}
\begin{gathered}
\xymatrix@C=2em{
  & \cdots  \ar[r] & \mH_{i+1}^v(C_n \otimes_{R_n^{\rm en}} (M \otimes_{R_n} Y_2)) \ar `r[rd] `_l `[lld] `[d] [ld] & \\
& \mH_i^v(C_n \otimes_{R_n^{\rm en}} (M \otimes_{R_n} Y_1)) \ar[r] & \mHH_i^{\dif_q}((M\otimes \Bbbk[x_{n+1}]) \otimes_{R_{n+1}} T_n) \ar `r[rd] `_l `[lld] `[d] [ld] & \\
& 
\mH_i^v(C_n \otimes_{R_n^{\rm en}} (M \otimes_{R_n} Y_2))\ar[r]  & \cdots &
}
 \end{gathered}
\end{equation}
\begin{equation*}
\cdots
\rightarrow \mH_i^v(C_n \otimes_{R_n^{\rm en}} (M \otimes_{R_n} Y_1)) \rightarrow \mHH_i^{\dif_q}((M\otimes \Bbbk[x_{n+1}]) \otimes_{R_{n+1}} T_n) \rightarrow 
\mH_i^v(C_n \otimes_{R_n^{\rm en}} (M \otimes_{R_n} Y_2))
\rightarrow \cdots.
\end{equation*}
Due to the splitting exactness of \eqref{eqn-split-exact-sequence}, the long exact sequence breaks up into short exact sequences of the form
\begin{equation}\label{eqn-ses-in-homology}
0
\rightarrow \mH_i^v(C_n \otimes_{R_n^{\rm en}} (M \otimes_{R_n} Y_1)) \rightarrow \mHH_i^{\dif_q}((M\otimes \Bbbk[x_{n+1}]) \otimes_{R_{n+1}} T_n) \rightarrow 
\mH_i^v(C_n \otimes_{R_n^{\rm en}} (M \otimes_{R_n} Y_2))
\rightarrow 0,
\end{equation}
one for each $i\in \Z$. 

So far we have ignored the topological (horizontal) differential in diagram \eqref{sesbicomplexes}. Each term in  the short exact sequence \eqref{eqn-ses-in-homology} carries the topological differential $d_t$. Taking homology with respect to $d_t$ gives us another long exact sequence 
\begin{equation}
\begin{gathered}
\xymatrix{
\cdots \ar[r] & \mH^h_j \mH_i^v(C_n \otimes_{R_n^{\rm en}} (M \otimes_{R_n} Y_1)) \ar[r] & \mtHHH_{j,i}^{\dif_q}((M\otimes \Bbbk[x_{n+1}]) \otimes_{R_{n+1}} T_n) \ar `r[rd] `_l `[lld] `[d] [ld] & \\
& \mH^h_j\mH_i^v(C_n \otimes_{R_n^{\rm en}} (M \otimes_{R_n} Y_2))
\ar[r]  & \mH^h_{j-1} \mH_i^v(C_n \otimes_{R_n^{\rm en}} (M \otimes_{R_n} Y_1))  \ar[r] & \cdots 
}
\end{gathered}
\ .
\end{equation}

We claim that 
$\mH^h_j(\mH_i^v(C_n \otimes_{R_n^{\rm en}} (M \otimes_{R_n} Y_2)))=0$ for all $j$, (which we will show shortly), which implies that
\begin{equation*}
\mH^h_j\mH_i^v(C_n \otimes_{R_n^{\rm en}} (M \otimes_{R_n} Y_1)) \cong 
\mtHHH_{j,i}^{\dif_q}((M\otimes \Bbbk[x_{n+1}]) \otimes_{R_{n+1}} T_n).
\end{equation*}
The definition \eqref{sesbicomplexes} of $Y_1$ shows that
$ \mH^h_j\mH_i^v(C_n \otimes_{R_n^{\rm en}} (M \otimes_{R_n} Y_1)) $ 
is the homology of the two-term complex 
\begin{equation*}
\xymatrix{
0 \ar[r] &
\mH_i^v(C_n \otimes_{R_n^{\rm en}} M \otimes_{R_n} R_{n+1}^{x_n+3x_{n+1}})
\ar[rr]^-{2(x_{n+1}-x_n)} & &
\mH_i^v(C_n \otimes_{R_n^{\rm en}} M \otimes_{R_n} R_{n+1}^{2x_{n+1}})
\ar[r] & 0
}.
\end{equation*}
Taking grading shifts back into account, this complex has homology concentrated in the second nonzero term, which is isomorphic to
\begin{equation*}
\mH^h_j\mH_i^v(C_n \otimes_{R_n^{\rm en}} M \otimes_{R_n} R_{n}^{2x_n})
\cong
\mtHHH_{j,i}^{\dif_q}(M)^{2x_n}
\end{equation*}
where the latter space is twisted as an $H_q$-module by $2x_n$. This is part $(i)$ of the proposition.

We now prove the claim that 
$\mH^h_j\mH_i^v(C_n \otimes_{R_n^{\rm en}} (M \otimes_{R_n} Y_2))=0$.
Note that $Y_2$ fits into a short exact sequence of $(R_{n+1},R_{n+1})$-bimodules
\[
0\lra Y_2^{\prime \prime} \lra Y_2\stackrel{\psi}{\lra} Y_2^\prime \lra 0,
\]
where the surjective map $\psi$ is given by 
\begin{equation} \label{bicomplexmor}
\begin{gathered}
\xymatrix@R=1em{
& \widetilde{R}_{n+1}^{x_n+x_{n+1}} \ar[r] \ar[dd]^{}    & 0 \ar[dd]  & &  0 \ar[r] \ar[dd] & 0 \ar[dd] &  \\
\psi: Y_2= & & && & & :=Y_2' \\
& B_n \ar[r]^{} \ar@/_2pc/[rrr]_{br}   & R_{n+1} \ar@/_2pc/[rrr]_{\Id} & & R_{n+1} \ar[r]^{\Id} & R_{n+1} \\
}
\end{gathered}
\ .
\end{equation}
Since the kernel of the multiplication map
$br \colon B_n \rightarrow R_{n+1}$ is generated as an
$(R_{n+1},R_{n+1})$-bimodule by
$v=x_{n+1} \otimes 1 -1 \otimes x_{n+1}$, it is easy to check that
$x_n v =v x_{n+1}$ and $v$ generates a copy of bimodule isomorphic to $\widetilde{R}_{n+1}^{x_n+x_{n+1}}$.
Thus the kernel of $\psi$ is given by
\begin{equation}
\begin{gathered}
\xymatrix@R=1em{
 & \widetilde{R}_{n+1}^{x_n+x_{n+1}} \ar[dd]^{\mathrm{Id}} \\
 Y_2^{\prime \prime} := & \\
 & \widetilde{R}_{n+1}^{x_n+x_{n+1}}
}
\end{gathered}
\ .
\end{equation}
Clearly, $\mH^v_\bullet(Y_2^{\prime \prime})=0$, and it follows that 
 $\mH^v_\bullet(Y_2)\cong \mH^v_\bullet (Y_2^{\prime})$.
Taking homology with respect to the topological (horizontal) differential $d_t$ then yields $\mH^h_j\mH_i^v(Y_2')=0$ and $\mH^h_j\mH^v_i(Y_2)=0$.

The computation of $\mHH_\bullet^{\dif_q}((M\otimes \Bbbk[x_{n+1}]) \otimes_{R_{n+1}} T_n') $ is very similar. We only outline the necessary changes. Again, first note that $C_1^\prime \otimes_{\Bbbk[x_{n+1}]^{\rm en}} T_n^\prime$ is a bicomplex of $(R_{n+1},R_{n+1})$-bimodules
\begin{equation}
\begin{gathered}
\xymatrix{
a^{\frac{3}{2}}t^{\frac{1}{2}}q^2 R_{n+1}^{2x_{n+1}} \ar[rr]^-{rb} \ar[d]_{0} && a^{\frac{3}{2}} t^{-\frac{1}{2}} {}^{x_{n+1}}B_{n}^{-x_n} \ar[d]^{x_{n+1} \otimes 1 - 1 \otimes x_{n+1}} \\
a^{\frac{1}{2}}t^{\frac{1}{2}}q^2R_{n+1} \ar[rr]^-{rb} && a^{\frac{1}{2}}t^{-\frac{1}{2}}B_{n}^{-(x_n+x_{n+1})}.
}
\end{gathered}
\end{equation}
Ignore the grading shifts for now for ease of notation. There is a short exact sequence of bicomplexes of $(R_{n+1},R_{n+1})$-bimodules
\[
0\lra Z_1 \lra C_1^\prime \otimes_{\Bbbk[x_{n+1}]^{\rm en}} T_n^\prime \lra Z_2\lra 0,
\]
whose terms are defined by
\begin{equation} \label{sesbicomplexes2}
\begin{gathered}
\xymatrix@R=1em{
& & & & & 0 \ar[dd] \\
R_{n+1}^{2x_{n+1}} \ar@/_2pc/[dddd]_{\Id} \ar[rrr]^{\Id} \ar[dd] & & & R_{n+1}^{2x_{n+1}} \ar[dd] \ar@/^2pc/[dddd]^{rb} &  & \\
& & & & := & Z_1\ar[dddd] \\
0 \ar[rrr] & & & 0 \\ 
& & & & &  \\
R_{n+1}^{2x_{n+1}} \ar[rrr]^{rb} \ar[dd]^{0}  & & & {}^{x_{n+1}}B_{n}^{-x_n} \ar[dd]_{x_{n+1} \otimes 1 - 1 \otimes x_{n+1}} \ar@/^3pc/[dddd]^{\tilde{br}} & & \\
& & & & = & 
C_1^\prime \otimes_{\Bbbk[x_{n+1}]^{\rm en}} T_n' \ar[dddd] \\
R_{n+1} \ar[rrr]^{rb} \ar@/_3pc/[dddd]_{ \Id} & & & B_{n}^{-(x_n+x_{n+1})} \ar@/^2pc/[dddd]^{\Id} \\
& & & & & \\
0 \ar[rrr] \ar[dd]^{} & & & \tilde{R}_{n+1} \ar[dd]_{x_{n+1} \otimes 1 - 1 \otimes x_{n+1}} &  &   \\
 & & & & := & Z_2\ar[dd]\\
R_{n+1} \ar[rrr]^{rb} & & & B_{n}^{-(x_n+x_{n+1})} \\
& & & & & 0  
}
\end{gathered}
\end{equation}

As in the previous part, there is a splitting of bicomplexes of $(R_n,R_n)$-bimodules given by
\begin{equation} \label{bicomplexsplitting'}
\begin{gathered}
\xymatrix{
R_{n+1}^{2x_{n+1}} \ar[rrr]^{rb} \ar[d]^{0}  & & & {}^{x_{n+1}}B_{n}^{-x_n} \ar[d]_{x_{n+1} \otimes 1 - 1 \otimes x_{n+1}}  \\
R_{n+1} \ar[rrr]^{rb}  & & & B_{n}^{-(x_n+x_{n+1})}  \\
&  & & \\
0 \ar[rrr] \ar[d]^{} & & & \tilde{R}_{n+1} \ar[d]_{x_{n+1} \otimes 1 - 1 \otimes x_{n+1}} \ar@/_3pc/[uuu]_{ \phi}  \\
R_{n+1} \ar[rrr]^{rb} \ar@/^3pc/[uuu]^{\Id} & & & B_{n}^{-(x_n+x_{n+1})}
\ar@/_3pc/[uuu]_{ \Id} \\
}
\end{gathered}
\end{equation}
where $\phi$ was defined earlier.
Thus we get short exact sequences of the form
\begin{equation}\label{eqn-ses-in-homology-prime}
0
\rightarrow \mH_i^v(C_n \otimes_{R_n^{\rm en}} (M \otimes_{R_n} Z_1)) \rightarrow \mHH_i^{\dif_q}((M\otimes \Bbbk[x_{n+1}]) \otimes_{R_{n+1}} T_n^\prime) \rightarrow 
\mH_i^v(C_n \otimes_{R_n^{\rm en}} (M \otimes_{R_n} Z_2))
\rightarrow 0
\end{equation}
for each $i\in \Z$. 
%\begin{equation*}
%0
%\longrightarrow \mH_i^v(C_n \otimes_{(R_n,R_n)} (M \otimes_{R_n} Z_1)) \longrightarrow \mHH_i(M 
%\otimes_{R_{n+1}} T_n') \longrightarrow 
%\mH_i^v(C_n \otimes_{(R_n,R_n)} (M \otimes_{R_n} Z_2))
%\longrightarrow 0.
%\end{equation*}
Taking horizontal homology for this short exact sequence gives us a long exact sequence. However, since the homology of
$\mH_j^h\mH_i^v(C_n \otimes_{R_n^{\rm en}} (M \otimes_{R_n} Z_1))$ is clearly always zero, we get that
\begin{equation*}
\mtHHH_{j,i}((M \otimes \Bbbk[x_{n+1}])\otimes_{R_{n+1}} T_n') \cong
\mH^h_j(\mH_i^v(C_n \otimes_{R_n^{\rm en}} (M \otimes_{R_n} Z_2)))
\end{equation*}
for all $j\in \Z$. We need to analyze the latter homology space.

There is a morphism of bicomplexes
\begin{equation} \label{bicomplexmor'}
\begin{gathered}
\xymatrix{
& 0 \ar[r] \ar[dd]^{}    & \tilde{R}_{n+1} \ar[dd]^{x_{n+1} \otimes 1 - 1 \otimes x_{n+1}}  & &  0 \ar[rr] \ar[dd] & & 0 \ar[dd] &  \\
Z_2:= & & && && &:=Z_2'\\
& R_{n+1} \ar[r]^-{rb} \ar@/_2pc/[rrr]_{\Id}   & B_n^{-(x_n+x_{n+1})} \ar@/_2pc/[rrrr]_{br} & & R_{n+1} \ar[rr]^-{x_{n+1}-x_n} & & R_{n+1}^{-(x_n+x_{n+1})} \\
}
\end{gathered}
\end{equation}
whose kernel is a vertical complex connected by the identity map. 
Thus 
$$\mtHHH_{j,i}((M\otimes \Bbbk[x_{n+1}]) \otimes_{R_{n+1}} T_n') \cong \mH_j^h\mH_i^v((C_n \otimes_{R_n^{\rm en}} (M \otimes_{R_n} Z_2^\prime)))\cong \mtHHH_{j,i}(M)^{-2x_n}.$$
The result follows.
\end{proof}

The above proof serves as a model for the Markov II invariance of $\ptHHH$. We only present the necessary changes. 

\begin{cor}\label{cor-Markov-II-for-pHHH}
Let $\beta$ be a braid and $M$ be the associated $p$-chain complex of Soergel bimodules.  Then there are isomorphisms of chain complexes of relative $pH_q$-Hochschild homology groups:
\begin{enumerate}
    \item[(i)] $\ptHHH^{\dif_q}((M\otimes \Bbbk[x_{n+1}]) \otimes_{R_{n+1}} \pT_n) \cong \ptHHH^{\dif_q}(M)^{2x_n}$,
    \item[(ii)] $\ptHHH^{\dif_q}((M\otimes \Bbbk[x_{n+1}])\otimes_{R_{n+1}} \pT_n') \cong \ptHHH^{\dif_q}(M)^{-2x_n}$.
\end{enumerate}
\end{cor}
\begin{proof}
To begin with, one replaces the Koszul complex $C_n$ utilized in the proof of Proposition \ref{prop-Markov-II-for-HHH} with the $p$-extended Koszul complex $\pC_n$. Also one needs to replace the vertical homology taken there by vertical slash homology (see Remark \ref{prop-Markov-II-for-HHH}).

For part $(i)$, one adapts equation \eqref{eqn-bimod-iso-tensor-with-Tn} into\footnote{The monoidality of $\pC_{n}$ is more convenient here than the homotopy-equivalent $\mc{P}(C_n)$.}
\begin{align}
\pC_{n+1} \otimes_{R_{n+1}^{\rm en}} ((M \otimes \Bbbk[x_{n+1}]) \otimes_{R_{n+1}} \pT_n) &=
(\pC_n \otimes \pC_1^\prime) \otimes_{R_{n+1}^{\rm en}} ((M \otimes \Bbbk[x_{n+1}]) \otimes_{R_{n+1}} \pT_n) \nonumber \\
& \cong 
\pC_n \otimes_{R_n^{\rm en}} (M \otimes_{R_n} (\pC_1^\prime \otimes_{\Bbbk[x_{n+1}]^{\rm en}} \pT_n)),
\end{align}
where $\pC_1^\prime\cong \pC_1$ is the $p$-extended complex of bimodules
\begin{equation}
    0\lra q^2 \Bbbk[x_{n+1}]^{x_{n+1}}\otimes \Bbbk[x_{n+1}]^{x_{n+1}}[1]^t_\dif \xrightarrow{x_{n+1}\otimes 1-1\otimes x_{n+1}} \Bbbk[x_{n+1}]\otimes \Bbbk[x_{n+1}]\lra 0.
\end{equation}
Then, diagram \eqref{sesbicomplexes} becomes the following $p$-extended version in the Hochschild direction: 
\begin{equation} \label{sesbicomplexespHHH}
\begin{gathered}
\xymatrix@R=1em@C=1.4em{
& & & & & 0 \ar[dd] \\
 q^3  R_{n+1}^{x_n+3x_{n+1}} [1]_\dif^t\ar@/_4pc/[dddd]_{\phi} \ar[rrr]^-{2(x_{n+1}-x_n)} \ar[dd] & & &  q R_{n+1}^{2x_{n+1}}\ar[dd] \ar@/^2pc/[dddd]^{\Id} &  \\
& & & & := & pY_1 \ar[dddd]\\
0 \ar[rrr] & & & 0 \\ 
& & &  \\
q ({}^{x_{n+1}}B_n^{x_{n+1}})[1]_\dif^t \ar[rrr]^{br} \ar[dd]^{x_{n+1} \otimes 1 - 1 \otimes x_{n+1}} \ar@/_4pc/[dddd]_{\tilde{br}} & & & q^{} R_{n+1}^{2x_{n+1}}  \ar[dd]^{0} &  & \\
& & &  & = &
\pC_1^\prime \otimes_{\Bbbk[x_{n+1}]^{\rm en}} \pT_n \ar[dddd] \\
 q^{-3} B_n \ar[rrr]^{br} \ar@/_4pc/[dddd]_{2 \Id} & & &  q^{-3 } R_{n+1}[-1]^t_\dif \ar@/^2pc/[dddd]^{2\Id} \\
&  &  & \\
q^{} \tilde{R}_{n+1}^{x_n+x_{n+1}} [1]_\dif^t \ar[rrr] \ar[dd]^{(x_{n+1}-x_n) \otimes 1 - 1 \otimes (x_{n+1}-x_n)} & & & 0 \ar[dd] &  & \\
& & & & := & pY_2 \ar[dd] \\
q^{-3} B_n \ar[rrr]^{br} & & &  q^{-3} R_{n+1}[-1]^t_\dif \\
& & & & & 0  
}
\end{gathered} \ .
\end{equation}
Here $\phi$ is the map that sends $1$ to $(x_{n+1}-x_n) \otimes 1 + 1 \otimes (x_{n+1}-x_n)$

Now, as in the previous proof, one needs to show that $\pC_n \otimes_{R_n^{\rm en}} (M \otimes_{R_n} pY_2)$ does not contribute to $\ptHHH^{\dif_q}$. This is easier since now $pY_2$ is an acyclic $p$-complex. Furthermore, $pY_1$ is quasi-isomorphic to the $p$-complex $q R_n^{2x_n}$ sitting in $t$-degree zero. Hence we obtain the isomorphism
 \[
 \ptHHH^{\dif_q}((M\otimes \Bbbk[x_{n+1}]) \otimes_{R_{n+1}} \pT_n) \cong
q^{-n} \mH^/_\bullet(\pC_n \otimes_{R_n^{\rm en}} (M \otimes_{R_n} R_n^{2x_n})) \cong \ptHHH^{\dif_q}(M)^{2x_n}.
 \]

For the second isomorphism, again there is a short exact sequence of bicomplexes of $(R_{n+1},R_{n+1})$-bimodules
\begin{equation} \label{sesbicomplexes2pHHH}
\begin{gathered}
\xymatrix@R=1em@C=2em{
& & & & & 0 \ar[dd] \\
q^7 R_{n+1}^{2x_{n+1}}[2]^t_\dif \ar@/_2pc/[dddd]_{\Id} \ar[rrr]^{\Id} \ar[dd] & & & q^7 R_{n+1}^{2x_{n+1}}[1]^t_\dif \ar[dd] \ar@/^4pc/[dddd]^{rb} &  & \\
& & & & := & pZ_1\ar[dddd] \\
0 \ar[rrr] & & & 0 \\ 
& & & & &  \\
q^7 R_{n+1}^{2x_{n+1}} [2]^t_\dif \ar[rrr]^{rb} \ar[dd]^{0}  & & &  q^5 ( {}^{x_{n+1}}B_{n}^{-x_n})[1]^t_\dif \ar[dd]_{x_{n+1} \otimes 1 - 1 \otimes x_{n+1}} \ar@/^4pc/[dddd]^{\tilde{br}} & & \\
& & & & = & 
pC_1^\prime \otimes_{\Bbbk[x_{n+1}]^{\rm en}} pT_n' \ar[dddd] \\
 q^3  R_{n+1} [1]^t_\dif \ar[rrr]^{rb} \ar@/_2pc/[dddd]_{ \Id} & & &  q B_{n}^{-(x_n+x_{n+1})} \ar@/^4pc/[dddd]^{\Id} \\
& & & & & \\
0 \ar[rrr] \ar[dd]^{} & & & q^5 \tilde{R}_{n+1}[1]^t_\dif \ar[dd]_{x_{n+1} \otimes 1 - 1 \otimes x_{n+1}} &  &   \\
 & & & & := & pZ_2\ar[dd]\\
 q^3 R_{n+1}[1]^t_\dif \ar[rrr]^{rb} & & &  q B_{n}^{-(x_n+x_{n+1})} \\
& & & & & 0  
}
\end{gathered}
\end{equation}
Getting rid of contractible summands, we see that $pZ_2$ is homotopy equivalent to
\begin{equation}\label{eqn-twisting-factor-pHHH-2}
   q^3 R_{n+1}[1]^t_\dif \xrightarrow{x_{n+1}-x_n}  q R_{n+1}^{-(x_n+x_{n+1})},
\end{equation}
which is, in turn, quasi-isomorphic to $q R_n^{-2x_n}$. The result follows after accounting for the shift built into $\ptHHH^{\dif_q}$.
\end{proof}

\begin{rem}
A closer examination of the proof of Corollary \ref{cor-Markov-II-for-pHHH} shows the necessity of collapsing $t$ and $a$ gradings for the construction of $\ptHHH^{\dif_q}$. Indeed, a comparison between equations \eqref{sesbicomplexes} and \eqref{sesbicomplexespHHH}, \eqref{sesbicomplexes2} and \eqref{sesbicomplexes2pHHH} shows that, if there were an $a$ and $t$ bigrading as for $\mtHHH$, then the grading shifts arising from the positive
and negative Markov II moves would not match. This is caused by the fact that the homological shift $[1]^t_\dif$ and grading shift $t$ functors are different when $p>2$, and thus 
\[
(t^{[\frac{1}{2}]})^{\circ 2}=t\neq [1]^t_\dif.
\]
Consequently, there does not seem to exist an overall compensation factor such as  $a^{-\frac{n}{2}}t^{\frac{n}{2}}$ in Definition \ref{def-HHH} that would make a triply graded $\ptHHH$ invariant under both Markov II moves.
\end{rem}

\begin{thm}
Let $\beta_1$ and $\beta_2$ be two braids whose closures represent the same link $L$ of $r$ components up to framing.  Suppose the framing numbers of the closures $\widehat{\beta}_1$ of $\beta_1$ and $\widehat{\beta}_2$ of $\beta_2$ differ by $\mathtt{f}_i(\widehat{\beta}_1)-\mathtt{f}_i(\widehat{\beta}_2)=a_i$, $i=1,\dots, r$. Then 
\begin{equation*}
\mtHHH^{\dif_q}(\beta_1) \cong \mtHHH^{\dif_q}(\beta_2)^{2\sum_{i=1}^r a_i x_i}
\end{equation*}
and
\begin{equation*}
p \mtHHH^{\dif_q}(\beta_1) \cong p \mtHHH^{\dif_q}(\beta_2)^{2 \sum_{i=1}^r a_i x_i}
\end{equation*}
where the generator of the polynomial action for the $i$th component is denoted $x_i$ and $\mtHHH^\dif(\beta_2)^{2 x_i}$ means we twist the $H_q$-module structure on the $i$th component by $2x_i$.
\end{thm}

\begin{proof}
The topological invariance follows from the proof of the braid relations in Section \ref{secbraidrelations} and the proof of the Markov moves.
\end{proof}

%Thus we obtain a link invariant 
%$HHH(L)$ with values in
%$\Bbbk[\partial_q,\partial_a] / (\partial_q^{j_q}, \partial_a^{j_a})\udmod$ where $j_a \in \{2,p\}$ and $j_q \in \{1,p \}$.

%This invariant categorifies the following invariants.
%\begin{itemize}
%    \item If $j_a=j_q=p$, then $\chi(HHH(L))$ is the HOMFLYPT polynomial where $a$ and $q$ are prime roots of unity.
 %     \item If $j_a=2, j_q=p$, then $\chi(HHH(L))$ is the HOMFLYPT polynomial where $a$ is generic and $q$ is a prime root of unity.
 %     \item If $j_a=p, j_q=1$, then $\chi(HHH(L))$ is the HOMFLYPT polynomial where $q$ is generic and $a$ is a prime root of unity.
  %    \item If $j_a=2, j_q=1$, then $\chi(HHH(L))$ is the HOMFLYPT polynomial where $a$ and $q$ are generic.
%\end{itemize}

\subsection{Unlinks and twistings}\label{subsecHOMFLYunlink}
In this section, we compute $\mHHH^{\dif_q}$ and $\pHHH^{\dif_q}$ for the identity element of the braid group $\mathrm{Br}_n$, and define an unframed link invariant in $\R^3$.

For the unknot, recall from the previous section the Koszul resolution $C_1$ of $\Bbbk[x]$, as a bimodule, is given by
\begin{equation*}
\xymatrix{
q^2a\Bbbk[x]^{x} \otimes \Bbbk[x]^{x}   \ar[rr]^-{x\otimes 1-1\otimes x} &&
\Bbbk^{}[x] \otimes \Bbbk^{}[x]
}
\ .
\end{equation*}
Tensoring this complex with $\Bbbk[x]$ as a bimodule yields 
\begin{equation*}
\xymatrix{
q^2a\Bbbk[x]^{2x}  \ar[r]^{\hspace{.2in} 0} & \Bbbk[x]
}    
\ .
\end{equation*}
Thus the homology of the unknot (up to shift) is identified with the bigraded $H_q$-module
\begin{equation*}
\Bbbk[x]\oplus q^2 a\Bbbk[x]^{2x}  
\ .
\end{equation*}

More generally, via the Koszul complex $C_n=C_1^{\otimes n}$, we have that the homology of the $n$-component unlink $L_0$   is equal to
\begin{equation}
   \mtHHH^{\dif_q}(L_0)\cong  a^{-\frac{n}{2}}t^{\frac{n}{2}} \mHH_\bullet(R_n) \cong a^{-\frac{n}{2}}t^{\frac{n}{2}} \bigotimes_{i=1}^n \left( \Bbbk[x_i] \oplus q^2 a \Bbbk[x_i]^{2x_i} \right).
\end{equation}
Alternatively, up to the grading shift $ a^{-\frac{n}{2}}t^{\frac{n}{2}} $,  we may identify $\mtHHH^{\dif_q}(L_0)$ with the exterior algebra over $R_n$ generated by the differential forms $dx_i$ of bidegree $aq^2$, $i=1,\dots, n$, subject to the condition that each $dx_i$ accounts for a twisting of $H_q$-module structure by $2x_i$.

It follows that, as for the ordinary HOMFLYPT homology, given a framed link $L$ of $n$ components arising as a braid closure $\widehat{\beta}$, its untwisted HOMFLYPT $H_q$-homology $\mtHHH^{\dif_q} (\beta) $ is a module over 
$$\mtHHH_{0,0,\bullet}^{\dif_q}(L_0)\cong R_n,$$ 
and thus one may consider a twisting\footnote{The independence of choices as to where to introduce the twist on $\beta$ can be proven as in the usual triply graded homology case. See, for instance, \cite[Theorem 1.3]{KRWitt}.} of the $H_q$-module structure on $\mtHHH^{\dif_q}(\beta)$ via the functor
$ R_n^f\otimes_{R_n}(\mbox{-}) $, where $f$ is a linear polynomial in $x_1, \dots , x_n$, (see Section \ref{subset-p-DG-pol}).

\begin{defn}\label{def-twisted-HHH}
Let $L$ be an $n$-strand framed link arising from the closure of a braid $\beta$. Label the components of $L$ by $1$ through $k$, and set the \emph{(linear) framing factor} of $\beta$ to be the linear polynomial
\[
\mathtt{f}_\beta = -\sum_{i=1}^k 2 \mathtt{f}_i x_i.
\]
\begin{enumerate}
\item[(1)]The \emph{$H_q$-HOMFLYPT homology} of $\beta$ is the triply graded $H_q$-module
\[
\mHHH^{\dif_q}(\beta):= \mtHHH^{\dif_q}(\beta)^{\mathtt{f}_\beta}\cong R_k^{\mathtt{f}_\beta}\otimes_{R_k} \mtHHH^{\dif_q}(\beta).
\]
\item[(2)]Likewise, the  \emph{$H_q$-HOMFLYPT $p$-homology} is the doubly graded $H_q$-module
\[
\pHHH^{\dif_q}(\beta):= \ptHHH^{\dif_q}(\beta)^{\mathtt{f}_\beta}\cong R_k^{\mathtt{f}_\beta}\otimes_{R_k} \ptHHH^{\dif_q}(\beta).
\]
\end{enumerate}
\end{defn}
   
 \begin{cor}\label{cor-HHH-twisted-Euler-characteristic}
Given a braid $\beta$, both  $\mHHH^{\dif_q}(\beta)$ and $\pHHH^{\dif_q}(\beta)$ are link invariants that only depend on the closure of $\beta$ as a link in $\R^3$.
 \begin{enumerate}
     \item[(i)] The slash homologies of $\mHHH^{\dif_q}(\beta)$ and $\pHHH^{\dif_q}(\beta)$ are finite dimensional.
     \item[(ii)] The Euler characteristic of $\mHHH^{\dif_q}(\beta)$ is equal to the HOMFLYPT polynomial of $\widehat{\beta}$ in the formal variables $q$ and $a$, while the Euler characteristic of $\pHHH^{\dif_q}(\beta)$ is equal to the Jones polynomial of $\widehat{\beta}$ in a formal $q$-variable.
     \item[(iii)] The Euler characteristic of the slash homology of $\mHHH^{\dif_q}(\beta)$ is equal to the specialization of the HOMFLYPT polynomial of $\widehat{\beta}$ at a root of unity $q$, while the Euler characteristic of the slash homology of $\pHHH^{\dif_q}(\beta)$ is equal to the specialization of the Jones polynomial of $\widehat{\beta}$ at a root of unity $q$.
 \end{enumerate}
 \end{cor}  
\begin{proof}
For the first statement, just notice that the twisting of the $p$-DG structure by the framing factor takes care of the Markov II move.

Next, the finite dimensionality of the homology theories follows, by construction, from the fact that ${^{f_{i_1}}B_{i_1}^{g_{i_1}}}\otimes_R \cdots \otimes_R {^{f_{i_k}}B_{i_k}^{g_{i_k}}}$ is an $H_q$-module with $2^k$-step filtration whose subquotients are isomorphic to $R^f$ as left $R\# H_q$-modules, and thus Corollary \ref{cor-finite-slash-homology} applies.

The Poincar\'{e} polynomial of $\mHHH^{\dif_q}(\beta)$, which is independent of the $H_q$-module structure on $\mHHH^{\dif_q}(\beta)$, is well known to be a Laurent polynomial in
$a$ and $t$ (i.e., in $\Z[a^{\pm},t^{\pm}]$), and Laurent series in $q$ (i.e., in
$\Z[q^{-1},q]]$). Specializing $t=-1$ recovers the HOMFLYPT polynomial (see, e.g., \cite{KR2,
KRWitt}). On the other hand, in the construction of $\pHHH^{\dif_q}(\beta)$, we have
categorically specialized the $a$ and $t$ grading shifts according to the relation $a=q^{2}t$,
and then transformed $t$ into $[1]^t_\dif$. The Grothendieck ring of $t$ and $q$ bigraded
$p$-complexes up to homotopy is equal to $\mathcal{O}_p\otimes_{\Z} \Z[q,q^{-1}]$ (c.f.~Corollary \ref{cor-K0-bicomplexes}). In this
ring, $[1]^t_\dif$ descends to $-1\in \mathcal{O}_p$. In turn, the $a$ variable is then
evaluated at $-q^2$. Hence the Poincare polynomial of $\pHHH^{\dif_q}(\beta)$, taking value in
$$\Z[q^{-1},q]]\cong \Z\otimes_\Z \Z[q^{-1}, q]]\subset \mathcal{O}_p \otimes_\Z \Z[q^{-1},q]],$$
is equal to the HOMFLYPT polynomial with $a=-q^2$. This is just the Jones polynomial in a formal variable $q$, and the second part follows.

Finally, taking slash homology of the homology theories has the effect, on the level of Grothendieck groups, of passing from $\Z[q^{-1},q]$ onto $\mathbb{O}_p$ (here we need part (i) showing that both $\mHHH^{\dif_q}(\beta)$ and $\pHHH^{\dif_q}(\beta)$ are quasi-isomorphic to finite-dimensional $H_q$-modules). Therefore, taking slash homology of $\mHHH^{\dif_q}(\beta)$ and $\pHHH^{\dif_q}(\beta)$ is equivalent to categorically specializing $q$ at a primitive $p$th root of unity. This finishes the proof of the corollary.
\end{proof}

\begin{rem}
One of the open problems in the triply graded Khovanov-Rozansky theory is whether the theory is (projectively) functorial with respect to link cobordisms. A fundamental obstruction lies in the fact that the conditions of the (projective) TQFT would require one to assign, to the unknot, a Frobenius algebra that is finite dimensional (or rather, a compact Frobenius algebra object in a triangulated category). Therefore, the slash homology of $\mHHH^{\dif_q}(\beta)$ and $\pHHH^{\dif_q}(\beta)$ serve as candidates of potentially functorial link homology theories.
\end{rem}

\section{A finite-dimensional \texorpdfstring{$\mathfrak{sl}_2$}{sl(2)}-homology theory} \label{sl2section}
From Corollary \ref{cor-HHH-twisted-Euler-characteristic}, one sees that the slash homology of $\pHHH^{\dif_q}(\beta)$ categorifies the Jones polynomial at a root of unity $q$. However, the specialized Jones polynomial, as an element of $\mathbb{O}_p$, sits inside the ambient ring 
\[
\mathcal{O}_p\otimes_{\Z} \mathbb{O}_p\supset \Z\otimes_\Z \mathbb{O}_p \cong \mathbb{O}_p.
\]

Following \cite{Cautisremarks} (see also \cite{RW} and \cite{QRS}), we define a $p$-differential on $\pHH^\dif_\bullet(\beta)$ of a braid $\beta$ as a categorically specialized homology theory of links. The slash homology of this theory bypasses the ambient ring construction, and directly constructs a singly graded finite-dimensional homology theory whose Euler characteristic lives in $\mathbb{O}_p$.

\subsection{A singly graded homology}
Consider the $H_q$-Koszul complex in one variable: 
\begin{equation}\label{eqn-ordinary-Koszul-with-Cautis-d}
  C_1:  0 \lra a q^2 \Bbbk[x]^x\otimes \Bbbk[x]^{x} \stackrel{d_C}{\lra} \Bbbk[x]\otimes \Bbbk[x]\lra 0
\end{equation}
where $d_C$ is the map
$d_C(f)=(x^2\otimes 1+1\otimes x^2)f$.
We regard the differential on the arrow as an endomorphsim of the Koszul complex, of bidegree $(-1,2)$.

\begin{lem}
The commutator of the endomorphisms $d_C$ and $\dif_q\in H_q$ is null-homotopic on the Koszul complex $C_1$.
\end{lem}
\begin{proof}
The commutator map $\phi:=[d_C,\dif_q]$ is given by
\[
\xymatrix{
 & 0 \ar[rr] && \Bbbk[x]^x\otimes \Bbbk[x]^{x} \ar[rr]^{x\otimes 1-1\otimes x} \ar[d]^{\phi}&& \Bbbk[x]\otimes \Bbbk[x] \ar[r] & 0 &\\
0\ar[r] &\Bbbk[x]^x\otimes \Bbbk[x]^{x} \ar[rr]^{x\otimes 1-1\otimes x} && \Bbbk[x]\otimes \Bbbk[x] \ar[rr] && 0  & 
}
\]
where $\phi$ sends the bimodule generator $1\otimes 1 \in \Bbbk[x]^x\otimes\Bbbk[x]^{x}$ to
\begin{align*}
    \phi(1\otimes 1) & = d_C(\dif_q(1\otimes 1)) -\dif_q d_C(1\otimes 1) =d_C(x\otimes 1+1\otimes x)-\dif_q(x^2\otimes 1+1\otimes x^2)\\
     & = (x\otimes 1+1\otimes x)(x^2\otimes 1+1\otimes x^2)-2(x^3\otimes 1+1\otimes x^3) \\
     & = (x\otimes 1+1\otimes x)(x^2\otimes 1+1\otimes x^2)-2(x\otimes 1+1\otimes x)(x^2\otimes 1-x\otimes x+1\otimes x^2)\\
     & = (x\otimes 1+1\otimes x)(-x^2\otimes 1+2x\otimes x-1\otimes x^2)\\
     & =-(x\otimes 1+1\otimes x)(x\otimes 1-1\otimes x)^2.
\end{align*}
We may thus choose a null-homotopy to be
\[
\xymatrix{
 & 0 \ar[rr] && \Bbbk[x]^x\otimes \Bbbk[x]^{x}\ar[dll]_h \ar[rr]^{x\otimes 1-1\otimes x} \ar[d]^{\phi}&& \Bbbk[x]\otimes \Bbbk[x] \ar[r] & 0 &\\
0\ar[r] &\Bbbk[x]^x\otimes \Bbbk[x]^{x} \ar[rr]^{x\otimes 1-1\otimes x} && \Bbbk[x]\otimes \Bbbk[x] \ar[rr] && 0  & 
}
\]
where $h$ is given by multiplication by the element $1\otimes x^2-x^2\otimes 1$, and acts on the rest of the complex by zero. The result follows.
\end{proof}

The Koszul complex $C_n$ inherits the endomorphism $d_C$ by forming the $n$-fold tensor product from the one-variable case. It follows, that for a given $p$-DG bimodule $M$ over $R_n$, there is an induced differential, still denoted $d_C$, given via the identification
\begin{equation}
    \mHH_\bullet^{\dif_q}(M)\cong \mH_\bullet (M\otimes_{R_n^{\rm en}} C_n),
\end{equation}
where the induced differential acts on the right-hand side by $\Id_M \otimes d_C$. By construction, $d_C$ has Hochschild degree $-1$ and $q$-degree $2$.

Lemma \ref{lemma-acylicity-commutator-relation} immediately implies the following.

\begin{cor}\label{cor-dC-commutes-with-H}
The induced differential $d_C$ on $\mHH^{\dif_q}_\bullet(M)$ commutes with the $H_q$-action.\hfill$\square$
\end{cor}

\begin{rem}
The differential, first observed by Cautis \cite{Cautisremarks}, has the following more algebro-geometric meaning. Identifying $\mHH^1(R_n)$ as vector fields on $\mathrm{Spec}(R_n)=\mathbb{A}^n$, $\mHH^1(R_n)$ acts as differential operators on $\mHH_\bullet(M)$ for any $R_n$-bimodule $M$, regarded as a coherent sheaf on $\mathbb{A}^n\times \mathbb{A}^n \cong T^* (\mathbb{A}^n)$. Under this identification, $d_C$ is given by, up to scaling by a nonzero number, contraction with the vector field
\[
\zeta_C:=\sum_{i=1}^n x_i^2\frac{\dif}{\dif x_i}.
\]
On the other hand, this is also the vector field that defines the $p$-DG structure on $R_n$ by derivation. Therefore, via the Gerstenhaber module structure on $\mHH_\bullet(M)$, the two actions naturally commute with each other on $\mHH_\bullet(M)$.

In a more general context, Hochschild homology is a Gerstenhaber module over Hochschild cohomology viewed
as a Gerstenhaber algebra.  We may view $d_C$ and $\dif_q$ as the same element $\zeta$ in Hochschild
cohomology but the element $d_C$ acts on homology via cap product $\zeta \cap \bullet $ and the element
$\dif_q$ acts via a Lie algebra action $\mathcal{L}_{\zeta}(\bullet)$.  The compatibility of these actions
is given by the equation
\[
\zeta \cap \mathcal{L}_{\zeta}(x)=[\zeta,\zeta] \cap x + \mathcal{L}_{\zeta}(\zeta \cap x).
\]
Since $[\zeta,\zeta]=0$, these actions commute.
\end{rem}

Now we are ready to introduce a collapsed $p$-homology theory of a braid closure. Let $\beta\in \mathrm{Br}_n$ be an $n$-stranded braid. We have associated to $\beta$ a usual chain complex of $H_q$-equivariant Soergel bimodules $T_\beta$ as in equation \eqref{eqn-chain-complex-for-braid}, of which we take $\pHH^{\dif_q}_\bullet$ for each term: 
\begin{equation}\label{eqn-collapsed-diagram-1}
\begin{gathered}
   \xymatrix{ & \vdots & \vdots & \vdots & \\
 \dots \ar[r]^-{\dif_t} & \pHH_i^{\dif_q}(\pT_\beta^{k+1}) \ar[u]^{\dif_C} \ar[r]^-{\dif_t} & \pHH_i^{\dif_q}(\pT_\beta^{k})\ar[u]^{\dif_C}  \ar[r]^-{\dif_t} & \pHH_i^{\dif_q}(\pT_\beta^{k-1}) \ar[u]^{\dif_C} \ar[r]^-{\dif_t}  &\dots \\
 \dots \ar[r]^-{\dif_t} & \pHH_{i+1}^{\dif_q}(\pT_\beta^{k+1})  \ar[r]^{\dif_t} \ar[u]^{\dif_C} & \pHH_{i+1}^{\dif_q}(\pT_\beta^{k}) \ar[u]^{\dif_C} \ar[r]^-{\dif_t} & \pHH_{i+1}^{\dif_q}(\pT_\beta^{k-1}) \ar[u]^{\dif_C} \ar[r]^-{\dif_t} &\dots \\
              & \vdots \ar[u]^{\dif_C} & \vdots \ar[u]^-{\dif_C} & \vdots \ar[u]^{\dif_C} &
   } 
   \end{gathered}
\end{equation}
Here, $\dif_C$ is a $p$-differential arising from $d_C$ as follows. By Proposition \ref{prop-pHH-determined-by-HH}, the $p$-Hochschild homology groups in a column above are identified with the terms in 
\begin{equation}\label{eqn-pH-construction}
\begin{gathered}
\xymatrix@C=1.5em{ 
\cdots \ar[r] & \mHH_{2i+1}^{\dif_q}(pT_\beta^k) \ar@{=}[r] & \cdots \ar@{=}[r] & \mHH_{2i+1}^{\dif_q}(pT_\beta^k) \ar[r]^{d_C} & \mHH_{2i}^{\dif_q}(pT_\beta^k) \ar `r[rd] `_l `[lllld]^{d_C} `[d][llld] %`[d] [ld]^{d_C}
&
\\
&  \mHH_{2i-1}^{\dif_q}(pT_\beta^k) \ar@{=}[r] & \cdots\ar@{=}[r] & \mHH_{2i-1}^{\dif_q}(pT_\beta^k) \ar[r]^-{d_C} & \cdots  & 
}
\end{gathered},
\end{equation}
where each term in odd Hochschild degree is repeated $p-1$ times.  Here the horizontal differential is the 
$p$-Hochschild induced map of the topological differential, which we have denoted by $\dif_t$ to indicate its 
origin. On the arrows connecting even and odd Hochschild degree terms, we put the map $d_C$ while keeping the
repeated terms connected by identity maps. This defines a $p$-complex structure, denoted $\dif_C$,  in each column 
in diagram \eqref{eqn-collapsed-diagram-1}. The $p$-differential $\dif_C$ commutes with the $H_q$-action on each 
term by Corollary \ref{cor-dC-commutes-with-H}. It follows that, applying the totalization construction $\mc{T}$ of
Lemma \ref{lem-totalization-functor}, we obtain a bigraded $p$-bicomplex of $H_q$-modules, with a horizontal 
(topological) $p$-differential $\dif_t$, a vertical $p$-differential $\dif_C$ and internal $p$-differential 
$\dif_q$. Denote the total $p$-differential $\dif_T:= \dif_t+\dif_C+\dif_q$, which collapses the double grading 
into a single $q$-grading.

\begin{defn} \label{pdgjonesdef} 
Let $\beta$ be an $n$ stranded braid. The \emph{untwisted $\mathfrak{sl}_2$ 
$p$-homology} of $\beta$ is the slash homology group
\[
p\widehat{\mH}(\beta):=
q^{-n}\mH^/_{\bullet}(\pHH^{\dif_q}_\bullet(pT_\beta),\dif_T),
\]
viewed as an object in $\mc{C}(\Bbbk,\dif_q)$.
\end{defn} 

The homology group $p\widehat{\mH}(\beta)$ is only singly graded as an object in $\mc{C}(\Bbbk,\dif_q)$. By 
construction, $p\widehat{\mH}(\beta)$ is the slash homology with respect to the $\dif_T$ action on $\oplus_{i,j} 
\pHH_i^{\dif_q}(\pT_\beta^j)$ (see diagram \eqref{eqn-pH-construction}). The latter space is doubly graded by the 
topological degree and $q$-degree with values in $\Z\times \Z$ (the Hochschild $a$ degree is already forced to be 
collapsed with the $q$ degree to make the Cautis differential $\dif_C$ homogeneous).  However, as in the proof of 
Corollary \ref{cor-Markov-II-for-pHHH}, the Markov II invariance for the homology theory requires one to collapse 
the $t$-grading onto the $a$-grading, thus also onto the $q$-grading.  We will use $p\widehat{\mH}_{i}(\beta)$ to 
stand for the homogeneous subspace sitting in some $q$-degree $i$. 

Let us also emphasize an important point about the vertical grading collapsing as the following remark.

\begin{rem}
A special remark is needed here about the grading specialization. In order to $p$-extend the Koszul complex \eqref{eqn-ordinary-Koszul-with-Cautis-d} into a $p$-Koszul complex with $\dif_C$ of degree two, we are forced to make the functor specialization from $[1]^a_d=a$ into $q^2[1]^q_\dif$, so that the $p$-extended complex looks like
\begin{equation}\label{eqn-pC1-with-q-shift}
     pC_1:~  0 \lra  q^4 \Bbbk[x]^x\otimes \Bbbk[x]^{x}[1]^q_\dif \stackrel{d_C}{\lra} \Bbbk[x]\otimes \Bbbk[x]\lra 0.
\end{equation}
Forming iterated tensor products of $pC_1$ determines the correct vertical $q$-degree shifts in each column of diagram \eqref{eqn-collapsed-diagram-1} of the $p$-Hochschild homology groups. 

Notice that, on the level of Grothendieck groups, this has the effect of specializing the formal variable $a$ into $-q^2$.
\end{rem}

When $[1]^t_{\dif}=[1]^q_{\dif}$ and $a=q^2[1]_{\dif}^q$, the grading shifts in equation \eqref{eqn-elementary-braids-p-ext} translate
into
\begin{equation}\label{eqn-elementary-braid-q-ext}
pT_i :=
 q^{-3} \left(B_i  \xrightarrow{br_i} R[-1]^q_{\dif}\right)
,
\quad \quad \quad
pT_i' :=  q^{3}\left(R[1]^q_\dif \xrightarrow{rb_i} q^{-2}  B_i^{-(x_i+x_{i+1})}\right)
.
\end{equation}
This also explains the necessity of $p$-extension in the collapsed $t$ and $a$ direction in $\pHHH$ in the previous section: The homological shift in that direction needs to be $p$-extended to agree with the homological shift in the $q$-direction.

Furthermore, the bigrading in diagram \eqref{eqn-collapsed-diagram-1} is now interpreted as a single grading, with both $\dif_C$ and $\dif_t$ raising $q$-degree by two.

This approach to a categorification of the Jones polynomial, at generic values of $q$, was first developed by Cautis \cite{Cautisremarks}.
We follow the exposition of Robert and Wagner from \cite{RW} and the closely related approach of Queffelec, Rose, and Sartori \cite{QRS}.

\subsection{Framed topological invariance}
In this subsection, we establish the topological invariance of the untwisted homology theory.

\begin{thm}\label{thm-untwisted-sl2-homology}
The homology $p\widehat{\mH}(\beta)$ is a finite-dimensional framed link invariant depending only on the braid closure of $\beta$.
\end{thm}

\begin{proof}
The proof of the theorem will mostly be parallel to that of Proposition \ref{prop-Markov-II-for-HHH} and Corollary \ref{cor-Markov-II-for-pHHH}. It amounts to showing that taking slash homology of $\pHH_\bullet^{\dif}(\beta)$ with respect to $\dif_T$ satisfies the Markov II move.

We start by discussing the normal $H_q$-equivariant Hochschild homology version. Let $L$ be a link in $\R^3$ obtained as a braid closure $\widehat{\beta}$, where $\beta\in \mathrm{Br}_n$ is an $n$-stranded braid. Recall that the homology groups $\mHH^\dif_\bullet(L)$ are defined by tensoring a complex of Soergel bimodules $M$ determined by $\beta$ with the Koszul complex $C_n$ and computing its termwise vertical (Hochschild) homology.
The differential $d_C$ is defined on the Koszul complex $C_n$. To emphasize its dependence on $n$, we will write $d_C$ on $C_n$ as $d_n$ in this proof, and likewise write $\dif_n$ for the $p$-extended differential on $\pC_n$.

Since
\begin{equation*}
C_{n+1}=C_n\otimes C_1^\prime=C_n \otimes \Bbbk[x_{n+1}] \otimes \Lambda \langle dx_{n+1} \rangle \otimes \Bbbk[x_{n+1}],
\end{equation*}
the vertical differential may be inductively defined as
\begin{equation} \label{Drecursive}
d_{n+1}=d_n \otimes\mathrm{Id} +
\mathrm{Id}\otimes d_1^{\prime}.
\end{equation}
Here we have  set
$C_1^\prime=\Bbbk[x_{n+1}] \otimes \Lambda \langle dx_{n+1} \rangle \otimes \Bbbk[x_{n+1}] $ equipped with part of the Cautis differential
\[
{d}^\prime_{1}:=x_{n+1}^2 \otimes \iota_{\frac{\partial}{\partial x_{n+1}} }\otimes 1+ 1 \otimes \iota_{\frac{\partial}{\partial x_{n+1}} }\otimes x_{n+1}^2
\ .
\]
The notation $\iota$ denotes the contraction of $dx_{n+1}$ with $\frac{\partial}{\partial x_{n+1}}$ .

From the proof of Proposition \ref{prop-Markov-II-for-HHH} (see equation \eqref{eqn-split-exact-sequence}), we have a vector space decomposition
\[
 \mHH_\bullet^{\dif_q}((M\otimes \Bbbk[x_{n+1}]) \otimes_{R_{n+1}} T_n) \cong
\mH_\bullet^v(C_n \otimes_{R_n^{\rm en}} (M \otimes_{R_n} Y_1)) \oplus
\mH_\bullet^v(C_n \otimes_{R_n^{\rm en}} (M \otimes_{R_n} Y_2)).
\]
Here $Y_1$ and $Y_2$ are the terms of $C_1^\prime\otimes_{\Bbbk[x_{n+1}]^{\rm en}} T_n$ (see equation \eqref{sesbicomplexes} for the definition). 
We claim that, instead of a direct sum decomposition,  we obtain a filtration of $ \mHH_\bullet^{\dif_q}((M\otimes \Bbbk[x_{n+1}]) \otimes_{R_{n+1}} T_n) $ as a module over $\Bbbk[d_C]/(d_C^2)$:
\begin{equation}\label{eqn-dc-filtration}
0\rightarrow
\mH_\bullet^v(C_n \otimes_{R_n^{\rm en}} (M \otimes_{R_n} Y_2))
\rightarrow
 \mHH_\bullet^{\dif_q}((M\otimes \Bbbk[x_{n+1}]) \otimes_{R_{n+1}} T_n) 
\rightarrow
\mH_\bullet^v(C_n \otimes_{R_n^{\rm en}} (M \otimes_{R_n} Y_1)) 
\rightarrow 0.
\end{equation}
Indeed, since $d_C$ acts on the $Y_1$ and $Y_2$ tensor factors via $d_1^\prime$, it suffices to check that $d_1^\prime$ preserves the submodule arising from $Y_2$ and presents the part arising from $Y_1$ as a quotient. To do this, we reexamine the sequence \eqref{sesbicomplexes} under vertical (Hochschild) homology. The part $Y_2$, under vertical homotopy equivalence, contributes to the horizontal (topological) complex (see equation \eqref{bicomplexmor})
\begin{subequations}
\begin{equation}
    {Y}^\prime_2:=\left(R_{n+1}[1]^t_d \xrightarrow{d_t=\mathrm{Id}} R_{n+1}\right)
\end{equation}
sitting entirely in Hochschild degree $0$. Likewise, the part $Y_1$ contributes to the horizontal
\begin{equation}
   {Y}^\prime_1:= \left(R_{n+1}^{x_n+3x_{n+1}}[1]^t_d [1]^a_d  \xrightarrow{d_t=2(x_{n+1}-x_n)}R_{n+1}^{2x_{n+1}}[1]^a_d \right)
\end{equation}
\end{subequations}
sitting entirely in Hochschild degree $1$.
Since $d_1^{\prime}$ decreases Hochschild degree by one, $ {Y}^\prime_1 $ must be preserved under $d_1^\prime$, acting upon it trivially, and $  {Y}^\prime_2 $ is equipped with the (zero) quotient action by $d_1^\prime$. 

Similar behavior happens under $p$-extension and degree collapsing with respect to the $\dif_C$ action (c.f.~the proof and notation of Corollary \ref{cor-Markov-II-for-pHHH}). Write  $\dif_1^\prime$ as the $p$-extended differential of $d_1^\prime$. Consider the degree collapsed diagram obtained from tensoring equation \eqref{eqn-elementary-braid-q-ext} with equation \eqref{eqn-pC1-with-q-shift} while taking vertical slash homology:
\begin{equation} \label{eqncollapsedsquare}
\begin{gathered}
\xymatrix@R=1em@C=1.4em{
q({}^{x_{n+1}}B_n^{x_{n+1}})[1]_\dif^q \ar[rrr]^{br} \ar[dd]^{x_{n+1} \otimes 1 - 1 \otimes x_{n+1}} & & &  qR_{n+1}^{2x_{n+1}} \ar[dd]^{0} &  & \\
& & &  & = &
\pC_1^\prime \otimes_{\Bbbk[x_{n+1}]^{\rm en}} \pT_n \\
 q^{-3} B_n \ar[rrr]^{br} & & &  q^{-3} R_{n+1} [-1]_\dif^q
}
\end{gathered} \ .
\end{equation}

After taking (vertical) $p$-Hochschild homology, the part of the term $pY_2$ that is not killed arises from 
\begin{subequations}
\begin{equation}\label{eqn-pY2-cone}
    p{Y}^\prime_2:=q^{-3}\left(R_{n+1} \xrightarrow{\mathrm{Id}} R_{n+1}[-1]^q_\dif \right)
\end{equation}
sitting entirely horizontally inside the lower horizontal arrow of \eqref{eqncollapsedsquare}. On the other hand, the part $pY_1$ contributes to the horizontal
\begin{equation}
   p{Y}^\prime_1:= q^3R_{n+1}^{x_n+3x_{n+1}}[1]^q_\dif  \xrightarrow{2(x_{n+1}-x_n)} q R_{n+1}^{2x_{n+1}}
\end{equation}
\end{subequations}
sitting in the top horizontal line of the square \eqref{eqncollapsedsquare}. Since $\dif_1^\prime$ acts vertically down, upon taking the slash homology with respect to $\dif_T$, it now follows that the term arising from $pY_2^\prime$, on which $\dif_T$ acts just via $\dif_t+\dif_q$, contributes nothing to the total slash homology, as this term is the cone of the identity map in the homotopy category of $p$-complexes. 

Now, translating the exact sequence \eqref{eqn-dc-filtration} as in the final step in the proof of Corollary \ref{cor-Markov-II-for-pHHH}, we obtain that
\[
\mH_\bullet^/( \pHH_\bullet (( M\otimes \Bbbk[x_{n+1}])\otimes_{R_{n+1}} T_n ),\dif_T)
\cong \mH_\bullet^/(\pHH_\bullet( M \otimes_{R_n} pY_1^\prime),\dif_T)
\cong q\mH_\bullet^/(\pHH_\bullet(M),\dif_T)^{2x_n}.
\]
The $q$ factor is cancelled out in the overall shift of $p\widehat{\mH}$. This finishes the first part of Markov II move.

The other case of the Markov II move is entirely similar, which we leave the reader as an exercise.

Finally, the finite dimensionality of $p\widehat{\mH}(\beta)$ follows from Corollary \ref{cor-finite-slash-homology}. The theorem follows.
\end{proof}

\begin{rem}
Note that there are more general super differentials $d_N$ considered in \cite{Cautisremarks, RW, QRS} which gives rise to an $\mathfrak{sl}_N$ link homology for general $N$ when the quantum parameter $q$ is assumed to be generic.  It is not clear how to recover general $\mathfrak{sl}_N$ link homologies when $q$ is a root of unity because $d_N$ and $\partial_q$ only commute for $N=2$.
\end{rem}

\subsection{Categorical Jones number}
For the unknot, the $p$-Hochschild homology of $\Bbbk[x]$ (see Section \ref{subsecHOMFLYunlink}) but now with the Cautis differential included, is given by 
%\begin{equation}\label{eqn-unknot-sl2}
%\xymatrix{
%q^2a^{p-1}\Bbbk[x]^{2x} \ar@{=}[r] & \cdots \ar@{=}[r] & q^2a\Bbbk[x]^{2x} \ar[r]^-{x^2} & \Bbbk[x]
%}    
%\end{equation}
\begin{equation}\label{eqn-unknot-sl2}
\xymatrix{
q^{2(3-p)}\Bbbk[x]^{2x} \ar@{=}[r] & \cdots \ar@{=}[r] & q^2\Bbbk[x]^{2x} \ar[r]^-{x^2} & \Bbbk[x]
}    
\end{equation}
whose slash homology is equal to $\Bbbk[x]/(x^2)$ equipped with the trivial $p$-differential. 
Likewise, using the $p$-extended Koszul complex $\pC_k$ one obtains, for the $k$-component unlink $L_0$, that
\begin{equation}\label{eqn-unlink-sl2}
p\widehat{\mH}(L_0)=\bigotimes_{i=1}^k \dfrac{\Bbbk[x_i]}{(x_i^2)},
\end{equation}
with the zero $p$-differential. Let us call this $p$-DG algebra $A_k$.

However, to correct the twisting factor as we have done for $\mtHHH$ and $\ptHHH$ in Section \ref{subsecHOMFLYunlink}, it is a bit more subtle. As the following example would show.

\begin{example}
Consider the two rank-one $p$-DG modules over $\Bbbk[x]$, where $\dif(x)=x^2$, $\Bbbk[x]$ and $\Bbbk[x]^x$. It is clear that twisting the second module by $\Bbbk[x]^{-x}$ results in an isomorphism
\[
\Bbbk[x]\cong \Bbbk[x]^x\otimes_{\Bbbk[x]}\Bbbk[x]^{-x}.
\]
However, this can not be done after taking slash homology of $\Bbbk[x]$ and $\Bbbk[x]^x$, as the first module is quasi-isomorphic to the ground field, while the second module is acyclic.
\end{example}

We therefore need to introduce a $p$-differential twisting to correct the framing factor occurring in Theorem \ref{thm-untwisted-sl2-homology} slightly differently from what we have done in Definition \ref{def-twisted-HHH}.
For a braid $\beta\in \mathrm{Br}_n $ whose closure is a $k$-stranded framed link. Choose\footnote{Again, the independence of choices can be proven as in the usual triply graded homology case. See, for instance, \cite[Theorem 1.3]{KRWitt}.} for each framed component of $\widehat{\beta}$ in $\beta$ a single strand in $\beta$ that lies in that component after closure, say, the $i_r$th strand is chosen for the $r$th component. Then define the polynomial ring $\Bbbk[x_{i_1},\dots, x_{i_k}]$ as a subring of $\Bbbk[x_1,\dots, x_n]$ generated by the chosen variables. Set
\begin{equation}
    \Bbbk[x_{i_1},\dots, x_{i_k}]^{\mathtt{f}_\beta}:= \Bbbk[x_{i_1},\dots, x_{i_k}]\cdot 1_\beta, \quad \quad \dif(1_\beta):=-\sum_{r=1}^k 2\mathtt{f}_r x_{i_r} 1_\beta.
\end{equation}
Then we make the twisting of $H_q$-modules on the $\pHH_\bullet$-level, termwise on $\pHH_\bullet(pT_\beta^i)$:
\begin{equation}
    \pHH^{\mathtt{f}_\beta}_\bullet(pT_\beta):=\pHH_\bullet(pT_\beta)\otimes_{\Bbbk[x_{i_1},\dots, x_{i_k}]}{\Bbbk[x_{i_1},\dots, x_{i_k}]}^{\mathtt{f}_\beta}.
\end{equation}

\begin{defn}
Given $\beta\in \mathrm{Br}_n$ whose closure is a $k$ stranded framed link, the \emph{$\mathfrak{sl}_2$ $p$-homology} is the object
\[
\pH(\beta):= \mH^/_\bullet(\pHH^{\mathtt{f}_\beta}_\bullet(pT_\beta),\dif_T)
\]
in the homotopy category $\mc{C}(\Bbbk, \dif_q)$.
\end{defn}

The internal $H_q$-structure shifts do not interfere with the proof of Theorem \ref{thm-untwisted-sl2-homology}, and they correct the overall twisting of $H_q$-modules arising from Markov II moves in the theorem. It follows that we obtain the first part of the following.

\begin{thm}
The $\mathfrak{sl}_2$ $p$-homology $\pH(\beta)$ is a singly graded, finite-dimensional link invariant depending only on the braid closure of $\beta$ as a link in $\R^3$. Furthermore its graded Euler characteristic 
$$\chi(\pH(L)):=\sum_{i}q^i \mathrm{dim}_\Bbbk(\pH_{i}(L))$$
is equal to the Jones polynomial at a $2p$th root of unity.
\end{thm}
\begin{proof}
As in the proof of Corollary \ref{cor-HHH-twisted-Euler-characteristic}, the twisting compensates for the
linear factors appearing in Markov II moves, thus establishing the topological 
invariance of $\pH(\beta)$. 
  
For the last statement, we will use the fact that the Euler characteristic does not change
before or after taking slash homology. This is because, as with the usual chain complexes, taking slash homology only gets rid of acyclic summands whose Euler characteristics are zero. 

Let us revisit diagram \eqref{eqn-pH-construction}. Before collapsing the $t$-grading onto $q$, the diagram arises by $p$-extending the same Hochschild homology group diagram of $T_\beta$ in the $t$-direction. Let $P_\beta(v,t)$ be the Poincar\'{e} polynomial of Cautis's bigraded diagram for now, where $v$, $t$ are treated as formal variables coming from $q$ and $t$ grading shifts.  As shown by Cautis \cite{Cautisremarks}, $P_\beta(v,-1)$ is the Jones polynomial of the link $\widehat{\beta}$ in the variable $v$.

The $p$-extension in the topological direction is equivalent to categorically specializing $[1]^t_d$ to $[1]^q_\dif$. It has the effect, on the Euler 
characteristic level, of specializing $t=-1$. Thus we obtain that the Euler characteristic of $p\mH(\beta)$ is equal to $P_\beta(v=q,t=-1)$. This the Jones polynomial evaluated at a $2p$th root of unity $q$.
The result follows. 
\end{proof}

\section{Examples} \label{secexamples}
In this section we compute the various homologies constructed earlier for $(2,n)$ torus links. Note that there are no framing factors to incorporate in this family of examples.
\subsection{HOMFLYPT homology}
We follow the exposition in \cite{KRWitt} to compute variations of the HOMFLYPT homology of $(2,n)$ torus links while making the necessary $p$-DG modifications. 

We begin by reviewing the homology of unlinks.  Recall the Koszul resolutions $C_1$ and $pC_1 $ of $\Bbbk[x]$ as bimodules are written uniformly as
\begin{equation*}
\xymatrix{
q^2\Bbbk[x]^{x} \otimes \Bbbk[x]^{x} [1]_*^a  \ar[r] &
\Bbbk^{}[x] \otimes \Bbbk^{}[x]
}
\ ,
\end{equation*}
where $*\in \{d,\dif\}$.
Tensoring this complex with $\Bbbk[x]$ as a bimodule yields
\begin{equation*}
\xymatrix{
q^2\Bbbk[x]^{2x} [1]_*^a  \ar[r]^{\hspace{.2in} 0} & \Bbbk[x]
}    
\ .
\end{equation*}
Thus the homology of the unknot (up to shift) is given by:
\begin{equation*}
\Bbbk[x] \oplus q^2\Bbbk[x]^{2x} [1]_*^a  
\ .
\end{equation*}
By the monoidal structure, the homology of the $n$-component unlink $L_0$ is 
\begin{subequations}
\begin{align}
   \mHHH^{\dif_q}(L_0) &\cong (a^{-1}t)^{\frac{n}{2}}\bigotimes_{i=1}^n \left( \Bbbk[x_i] \oplus q^2 a \Bbbk[x_i]^{2x_i} \right), \label{eqn-HHH-unlink} \\
    \pHHH^{\dif_q}(L_0) &\cong \bigotimes_{i=1}^n q^{-1} \left( \Bbbk[x_i] \oplus q^4 [1]^t_{\dif} \Bbbk[x_i]^{2x_i} \right) \label{eqn-pHHH-unlink}
    \ .
\end{align}
\end{subequations}
Here, in the second equation \eqref{eqn-pHHH-unlink}, the shift $q^4[1]^t_d$ arises from specializing the grading shift functor $a=[1]^a_d$ in \eqref{eqn-HHH-unlink} to $q^2[1]^t_\dif$.  The slash homology of \eqref{eqn-pHHH-unlink} is
\begin{equation*}
\left(q^{-1} \Bbbk \oplus q^{p+1} V_{p-2}^q [1]^t_{\dif} \right)^{\otimes n}
\end{equation*}
and hence its Euler characteristic is $(q+q^{-1})^n$.

Throughout the remainder of this subsection, let $R=\Bbbk[x_1,x_2]$,
$B=B_1$, and $T=T_1$.
We begin with the computation of the two component unlink since it will play a role in the homology of the $(2,n)$ torus link $T_{2,n}$.

By tensoring the Koszul resolution of $\Bbbk[x]$ by itself we obtain a resolution of $\Bbbk[x_1,x_2]$ as bimodules homotopic to
%\begin{equation} \label{k[x,y]resolved}
%\xymatrix{
%& R^{x_1} \otimes R^{x_1} \ar[dr] & \\
%R^{x_1+x_2} \otimes R^{x_1+x_2}
%\ar[ur] \ar[dr] &
% &
%R \otimes R \\
%& R^{x_2} \otimes R^{x_2} \ar[ur]
%&
%}    
%\end{equation}

%\begin{equation} \label{k[x,y]resolved}
%\xymatrix{
%q^4R^{x_1+x_2} \otimes R^{x_1+x_2}  
%\ar[rrr]^{\left(\begin{smallmatrix}  x_2 \otimes 1 -  1 %\otimes x_2 \\ -x_1 \otimes 1 +1 \otimes x_1 %\end{smallmatrix}\right) \hspace{.2in}} 
%& & &  
%q^2 R^{x_1} \otimes R^{x_1} \oplus
%q^2 R^{x_2} \otimes R^{x_2}
%\ar[rrrr]^{\hspace{.5in} \left( \begin{smallmatrix} x_1 %\otimes 1 - 1 \otimes x_1 , &  x_2 \otimes 1 -  1 \otimes x_2 %\end{smallmatrix} \right) }
%  & & & & R \otimes R
%}  
%\ .
%\end{equation}

\begin{equation} \label{k[x,y]resolved}
\xymatrix{
q^4R^{x_1+x_2} \otimes R^{x_1+x_2}  [2]^a_*
\ar[rrr]^-{\left(\begin{smallmatrix}  x_2 \otimes 1 -  1 \otimes x_2 \\ -x_1 \otimes 1 +1 \otimes x_1 \end{smallmatrix}\right) \hspace{.2in}} 
& & &  
{\begin{pmatrix} q^2 R^{x_1} \otimes R^{x_1} \\ \oplus \\
q^2 R^{x_2} \otimes R^{x_2} 
\end{pmatrix}}
\ar[rrrr]^{\hspace{.5in} \left( \begin{smallmatrix} x_1 \otimes 1 - 1 \otimes x_1 , &  x_2 \otimes 1 -  1 \otimes x_2 \end{smallmatrix} \right) }[1]^a_*
  & & & & R \otimes R
}  
\ .
\end{equation}

Tensoring this complex with $R$ as a bimodule yields
\begin{equation*}
\xymatrix{
q^4 R^{2(x_1+x_2)} [2]^a_*
\ar[r]^{ 0 \hspace{.2in}} &
(q^2 R^{2x_1} \oplus q^2 R^{2x_2})[1]^a_* \ar[r]^<<<<{0}
& 
R^{}.
}    
\end{equation*}

Thus the homology of the two component unlink (up to shift) is given by the relative Hochschild homology (assuming $*=d$):
\begin{equation} \label{R2homology}
\mHH^{\dif_q}_i=
\begin{cases}
R & \text{ if } i=0 \\
q^2 R^{2x_1} \oplus q^2 R^{2x_2} & \text{ if } i=1 \\
q^4 R^{2(x_1+x_2)}  & \text{ if } i=2 \\
0 & \text{ otherwise. } 
\end{cases}   
\end{equation}
We write the relative Hochschild homology uniformly for $* \in \{ d,\dif \}$ as
\begin{equation*}
R \oplus q^2(R^{2x_1} \oplus R^{2x_2}) [1]^a_* 
\oplus q^4 R^{2(x_1+x_2)}[2]^a_*.
\end{equation*}

%%%This intermediate homology calculation seems unnecessary%%
%As objects in the stable category of $H_q$, the relative %Hochschild homology of $R$ (for $*=d_0$) is:
%\begin{equation*}
%\mHH^{\dif_q}_i=
%\begin{cases}
%V_0^q \otimes V_0^q & \text{ if } i=0 \\
%q^2 V_{p-2}^q \otimes V_0^q \oplus q^2 V_0 \otimes V_{p-2}^q & %\text{ if } i=-1 \\
%q^4 V_{p-2}^q \otimes V_{p-2}^q  & \text{ if } i=-2 \\
%0 & \text{ otherwise. } 
%\end{cases}   
%\end{equation*}
%We write is unformly for $*=d_0,\dif_0$ as
%\begin{equation*}
%V_0^q \otimes V_0^q \oplus (V_{p-2}^q \otimes V_0^q \oplus V_0^q %\otimes V_{p-2}^q) \{2 \}[1]^a_* 
%\oplus V_{p-2}^q \otimes V_{p-2}^q \{4 \} [2]^a_*.
%\end{equation*}

Now we return to the computation of the HOMFLYPT homology of the $(2,n)$ torus link.
The complex \eqref{k[x,y]resolved} is isomorphic by a change of basis in the middle term in the complex by the matrix
\begin{equation*}
\begin{pmatrix}
-1 & 0 \\ -1 & 1 
\end{pmatrix}
\ ,
\end{equation*}
to
\begin{equation} \label{k[x,y]resolved2}
\xymatrix{
q^4 R^{x_1+x_2} \otimes R^{x_1+x_2} [2]^a_*
\ar[r]  & (q^2 R \otimes R \oplus q^2 R \otimes R)[1]^a_* \ar[r]
 &
R^{} \otimes R^{} 
}
\ ,
\end{equation}
where now the $H_q$-structure in the middle term is twisted by the matrix
\begin{equation} \label{twistmid}
\begin{pmatrix}
x_1 \otimes 1 +1 \otimes x_1 & 0 \\
x_1 \otimes 1 + 1 \otimes x_1 - x_2 \otimes 1 - 1 \otimes x_2 & x_2 \otimes 1 + 1 \otimes x_2
\end{pmatrix}   
\ .
\end{equation}

Now we determine the relative Hochschild homology of the bimodule $B$.
Tensoring \eqref{k[x,y]resolved2} on the left as a bimodule with $B$ yields
\begin{equation} \label{Bcomplex}
\xymatrix{
q^4 ({}^{x_1+x_2} B^{x_1+x_2})[2]^a_* \ar[r]^{\hspace{.1in} \alpha}
& (q^2 B \oplus q^2 B)[1]^a_* \ar[r]^{\hspace{.3in} \beta} & B
}  
\end{equation}
where the middle term has an $H_q$-structure twisted by \eqref{twistmid}
%\begin{equation*}
%\begin{pmatrix}
%x_1 \otimes 1 + 1 \otimes x_1 & 0 \\
%x_1 \otimes 1 + 1 \otimes x_1 - x_2 \otimes 1 - 1 \otimes x_2 %& x_2 \otimes 1 + 1 \otimes x_2
%\end{pmatrix}    
%\end{equation*}
and $\alpha$ and $\beta$ are given by
\begin{equation*}
\alpha=\begin{pmatrix}
-x_2 \otimes 1 + 1 \otimes x_2 \\ 0
\end{pmatrix}  
\quad \quad
\beta=\begin{pmatrix}
0 & x_2 \otimes 1 - 1 \otimes x_2
\end{pmatrix}.
\end{equation*}
The kernel of $\alpha$ is generated as a left $R$-module by 
\begin{equation*}
\Gamma = \frac{1}{2}(x_2 \otimes 1 + 1 \otimes x_2  - x_1 \otimes 1 - 1 \otimes x_1).  
\end{equation*}
It is easy to verify that $\partial_q(\Gamma)=(x_1+x_2)\Gamma$ so that after accounting for the extra twist,
$\ker \alpha \cong R^{3(x_1+x_2)}$.
Note that in $\mathrm{cok} \alpha$, we have $x_2 \otimes 1 = 1 \otimes x_2$ and $x_1 \otimes 1 = 1 \otimes x_1$, so
$\mathrm{cok} \alpha$ is generated as an $R$-module by $1 \otimes 1$ and thus $\mathrm{cok}  \alpha \cong R$.  Similarly, 
$ \ker \beta \cong \mathrm{cok}  \beta \cong R$.

These observations combined with the $H_q$-structure given in \eqref{Bcomplex} yield the following result for the relative Hochschild homology of $B$:
\begin{equation} \label{HH(B)}
\mHH^{\dif_q}_i(B) \cong
\begin{cases}
R & \text{ if } i=0  \\
q^2 R \oplus q^4 R \text{ twisted by } \begin{pmatrix}
2x_1 & 0 \\
-2 & x_1+3x_2
\end{pmatrix} & \text{ if } i=1  
\\
q^6 R^{3(x_1+x_2)} & \text{ if } i=2 \\
0 & \text{ otherwise. } 
\end{cases}   
\end{equation}
%In the stable category of $H_q$-modules, we write this uniformly %for $*=d_0,\dif_0$ as
%\begin{equation*}
%V_0^q \otimes V_0^q
%\oplus
%(V_{p-2}^q \otimes V_0^q \{2 \} \oplus V_{p-1}^q \otimes %V_{p-3}^q \{4 \})[1]^a_*
%\oplus 
%V_{p-3}^q \otimes V_{p-3}^q \{ 6 \} [2]^a_*.
%\end{equation*}

\begin{lem} \label{complexforT^n}
The braiding complex $T^{\otimes n} $ simplifies in the following ways.
\begin{enumerate}
\item[(i)] In $\mc{C}^{\dif_q}(R,R,d_0)$, one has $T^{\otimes n} \cong (atq^4)^{-\frac{n}{2}}$
\begin{equation*}
%(atq^2)^{\frac{-n}{2}} 
\left(
\xymatrix{
 q^{2(n-1)}B^{(n-1)e_1}[n]^t_{d} \ar[r]^-{p_n} 
& q^{2(n-2)}B^{(n-2)e_1}[n-1]^t_{d} \ar[r]^{\hspace{.5in}p_{n-1}}
& \cdots \ar[r]^{p_3 \hspace{.2in}}
& q^2 B^{e_1}[2]^t_{d} \ar[r]^{\hspace{.2in} p_{2}}
& B^{}[1]^t_{d} \ar[r]^{br} 
& R}
\right)\ .
\end{equation*}

\item[(ii)] In $\mc{C}^{\dif_q}(R,R,\dif_0)$, one has $T^{\otimes n} \cong (q^{-3}[-1]^t_{\dif})^n$
\begin{equation*}
\left(
\xymatrix{
q^{2(n-1)}B^{(n-1)e_1}[n]^t_{\dif} \ar[r]^-{p_n} 
& q^{2(n-2)}B^{(n-2)e_1}[n-1]^t_{\dif} \ar[r]^{\hspace{.5in}p_{n-1}}
& \cdots \ar[r]^{p_3 \hspace{.2in}}
& q^2 B^{e_1}[2]^t_{\dif} \ar[r]^{\hspace{.2in} p_{2}}
& B^{}[1]^t_{\dif} \ar[r]^{br} 
& R
}    
\right) \ ,
\end{equation*}
\end{enumerate}
where
\begin{equation*}
p_{2i}=1 \otimes (x_2-x_1) - (x_2-x_1) \otimes 1
\quad \quad
p_{2i+1}=1 \otimes (x_2-x_1)+(x_2-x_1) \otimes 1.
\end{equation*}
\end{lem}

\begin{proof}
This is proved by induction on $n$.  One uses
homotopy equivalences
\[q^{2i}B^{ie_1} \otimes T \cong q^{2(i+1)}B^{(i+1)e_1}\] 
and then determining the images of the maps $br$, $p_{2i}$ and $p_{2i+1}$ under these equivalences.
\end{proof}

\begin{prop}
The $H_q$-HOMFLYPT homology of a $(2,n)$ torus link, as an $H_q$-module depends on the parity of $n$.
\begin{enumerate}
\item[(i)] If $n$ is odd:
    \begin{align*}
    & (atq^4)^{-\frac{n}{2}}a^{-1}t \Bigg( \left( q^2 [1]^a_d \Bbbk[x]^{2x} \oplus q^4 [2]^a_d \Bbbk[x]^{4x} \right)
    \bigoplus \\
    &\bigoplus_{i \in \{2,4,\ldots,n-1 \}} \Big(
    q^{2(i-1)}\Bbbk[x]^{2(i-1)x}
    \oplus [1]^a_d \begin{pmatrix} q^{2i}\Bbbk[x] \\ \oplus  \\ q^{2i+2}\Bbbk[x] \end{pmatrix} 
    \oplus q^{2i+4}[2]^a_d \Bbbk[x]^{2(i+1)x} \Big)[i]^t_d \Bigg)
    \end{align*}
    with the $H_q$-structure on the middle object 
    $\begin{pmatrix} q^{2i}\Bbbk[x] \\ \oplus  \\ q^{2i+2}\Bbbk[x] \end{pmatrix} $ given by
   $
        \begin{pmatrix}
        2i x & 0 \\
        -2 & (2i+2)x
        \end{pmatrix}
    $.
    \item[(ii)] If $n$ is even:
    \begin{align*}
    &(atq^4)^{-\frac{n}{2}}a^{-1}t \Bigg(  \left( q^2 [1]^a_d \Bbbk[x]^{2x} \oplus q^4 [2]^a_d \Bbbk[x]^{4x} \right)
    \bigoplus \\ 
    & \bigoplus_{i \in \{2,4,\ldots,n-2 \}} \Big(
    q^{2(i-1)}\Bbbk[x]^{2(i-1)x}
    \oplus [1]^a_d \begin{pmatrix} q^{2i}\Bbbk[x] \\ \oplus  \\ q^{2i+2}\Bbbk[x] \end{pmatrix} 
    \oplus q^{2i+4}[2]^a_d \Bbbk[x]^{2(i+1)x} \Big)[i]^t_d \bigoplus \\
     &
     \Big(
    q^{2(n-1)}\Bbbk[x_1,x_2]^{(n-1)(x_1+x_2)}
    \oplus [1]^a_d \begin{pmatrix} q^{2n}\Bbbk[x_1,x_2] \\ \oplus  \\ q^{2n+2}\Bbbk[x_1,x_2] \end{pmatrix} 
    \oplus q^{2n+4}[2]^a_d \Bbbk[x_1,x_2]^{(n+2)(x_1+x_2)} \Big)[n]^t_d
    \Bigg)
    \end{align*}
    with the $H_q$-structure on the middle object 
    $\begin{pmatrix} q^{2i}\Bbbk[x] \\ \oplus  \\ q^{2i+2}\Bbbk[x] \end{pmatrix} $ given by
   $
        \begin{pmatrix}
        2i x & 0 \\
        -2 & (2i+2) x
        \end{pmatrix}
    $
        and the $H_q$ structure on the middle object 
    $\begin{pmatrix} q^{2n}\Bbbk[x_1,x_2] \\ \oplus  \\ q^{2n+2}\Bbbk[x_1,x_2] \end{pmatrix} $ given by
   $
        \begin{pmatrix}
        (n+1)x_1+(n-1)x_2 & 0 \\
        -2 & n(x_1+x_2)+2x_2
        \end{pmatrix}
    $.
\end{enumerate}
\end{prop}

\begin{proof}
We sketch some of the details. 
It is clear that $\mHH_{\bullet}^{\dif_q}(p_{2i})=0$ and $\mHH_{\bullet}^{\dif_q}(p_{2i+1})=2(x_2-x_1)$.
Then taking the Hochschild homology of the complex in Lemma
\ref{complexforT^n} breaks up into a sum of pieces of the form
\begin{equation*}
\xymatrix{
\mHH_\bullet^{\dif_q}(q^2 B) \ar[rr]^{\mHH_\bullet^{\dif_q}(br)}   & & \mHH_\bullet^{\dif_q}(R),
}
\quad \quad 
\xymatrix{
\mHH_\bullet^{\dif_q}(q^{4i}B^{2i(x_1+x_2)}) \ar[rr]^{2(x_2-x_1) \hspace{.2in}}
& &  \mHH_\bullet^{\dif_q}(q^{2(2i-1)}B^{(2i-1)(x_1+x_2)})
}
\end{equation*}
and (if $n$ is even) the leftmost piece is
\begin{equation*}
\mHH_\bullet ^{\dif_q}(q^{2(n-1)}B^{(n-1)(x_1+x_2)}).
\end{equation*}

The result follows using the Hochschild homologies
of $R$ calculated in \eqref{R2homology} and
of $B$ calculated in \eqref{HH(B)}.
Recall that $\mHH_\bullet^{\dif_q}(B)$ is generated by elements of the form 
$\mathrm{cok}  \beta$,$ \mathrm{cok}  \alpha$, $\ker \beta$ and $ \ker \alpha$.

The morphism $\mHH_\bullet ^{\dif_q}(br)$ maps $\mathrm{cok}  \beta$ and $\mathrm{cok}  \alpha$ isomorphically onto their images. 
On the other hand, under $\mHH_\bullet^{\dif_q}(br)$, the image of $\ker \beta$ and $\ker \alpha$ identify the variables $x_1$ and $x_2$ in $\mHH_\bullet^{\dif_q}(R)$.
We call this identified variable $x$.

The map 
\[
\xymatrix{
\mHH_\bullet^{\dif_q}(p_{2i+1}) \colon &
\mHH_\bullet^{\dif_q}(q^{4i}B^{2i(x_1+x_2)}) \ar[rr]^{2(x_2-x_1)\hspace{.2in}}
& &  \mHH_\bullet^{\dif_q}(q^{2(2i-1)}B^{(2i-1)(x_1+x_2)})
}
\]
has no kernel and the image also identifies $x_1$ and $x_2$ as a common variable $x$.
\end{proof}
Similarly, one has the following analagous $p$-version of the previous result.
\begin{prop}
The bigraded $H_q$-HOMFLYPT $p$-homology of a $(2,n)$ torus knot, as an $H_q$-module depends on the parity of $n$.
\begin{enumerate}
\item[(i)] If $n$ is odd it is:
    \begin{align*}
    & q^{-3n-2}[-n]^t_{\dif} \left( q^4 [1]^t_{\dif} \Bbbk[x]^{2x} \oplus q^8 [2]^t_{\dif} \Bbbk[x]^{4x} \right)
    \bigoplus \\
    &\bigoplus_{i \in \{2,4,\ldots,n-1 \}} q^{-3n-2}\left(
    q^{2(i-1)}\Bbbk[x]^{2(i-1)x}
    \oplus q^2[1]^t_{\dif} \begin{pmatrix} q^{2i}\Bbbk[x] \\ \oplus  \\ q^{2i+2}\Bbbk[x] \end{pmatrix} 
    \oplus q^{2i+8}[2]^t_{\dif} \Bbbk[x]^{2(i+1)x} \right)[i-n]^t_{\dif}
    \end{align*}
    with the $H_q$-structure on the middle object 
    $\begin{pmatrix} q^{2i}\Bbbk[x] \\ \oplus  \\ q^{2i+2}\Bbbk[x] \end{pmatrix} $ given by
   $
        \begin{pmatrix}
        2i x & 0 \\
        -2 & (2i+2)x
        \end{pmatrix}
    $.
    \item[(ii)] If $n$ is even it is:
    \begin{align*}
    &q^{-3n-2}[-n]^t_{\dif}  \left( q^4 [1]^t_{\dif} \Bbbk[x]^{2x} \oplus q^8 [2]^t_{\dif} \Bbbk[x]^{4x} \right)
    \bigoplus \\ 
    &\bigoplus_{i \in \{2,4,\ldots,n-2 \}} q^{-3n-2}\left(
    q^{2(i-1)}\Bbbk[x]^{2(i-1)x}
    \oplus q^2[1]^t_{\dif} \begin{pmatrix} q^{2i}\Bbbk[x] \\ \oplus  \\ q^{2i+2}\Bbbk[x] \end{pmatrix} 
    \oplus q^{2i+8}[2]^t_{\dif} \Bbbk[x]^{2(i+1)x} \right)[i-n]^t_{\dif} \bigoplus \\
     &
 q^{-3n-2} \left(
    q^{2(n-1)}\Bbbk[x_1,x_2]^{(n-1)(x_1+x_2)}
    \oplus q^2[1]^t_{\dif} \begin{pmatrix} q^{2n}\Bbbk[x_1,x_2] \\ \oplus  \\ q^{2n+2}\Bbbk[x_1,x_2] \end{pmatrix} 
    \oplus q^{2n+8}[2]^t_{\dif} \Bbbk[x_1,x_2]^{(n+2)(x_1+x_2)} \right)
    \end{align*}
    with the $H_q$-structure on the middle object 
    $\begin{pmatrix} q^{2i}\Bbbk[x] \\ \oplus  \\ q^{2i+2}\Bbbk[x] \end{pmatrix} $ given by
   $
        \begin{pmatrix}
        2i x & 0 \\
        -2 & (2i+2)x
        \end{pmatrix}
    $
        and the $H_q$ structure on the middle object 
    $\begin{pmatrix} q^{2n}\Bbbk[x_1,x_2] \\ \oplus  \\ q^{2n+2}\Bbbk[x_1,x_2] \end{pmatrix} $ given by
   $
        \begin{pmatrix}
        (n+1)x_1+(n-1)x_2 & 0 \\
        -2 & n(x_1+x_2)+2x_2
        \end{pmatrix}
    $.
\end{enumerate}
\end{prop}

\begin{cor}
In the stable category of $H_q$-modules, the slash homology of the $H_q$-HOMFLYPT $p$-homology of a $(2,n)$ torus link depends on the parity of $n$.
\begin{enumerate}
\item[(i)] If $n$ is odd it is:
    \begin{align*}
    & q^{-3n-2}[-n]^t_{\dif} \left( q^{p+2} V_{p-2}^q [1]^t_{\dif} \oplus q^{p+4} V_{p-4}^q [2]^t_{\dif} \right)
    \bigoplus \\
    \bigoplus_{i \in \{2,4,\ldots,n-1 \}} & q^{-3n-2}\left(
  q^{p} V_{p-2(i-1)}^q 
    \oplus  \begin{pmatrix} q^{p+2} V_{p-2i}^q   \\ \oplus  \\ q^{p+2} V_{p-2i-2}^q  \end{pmatrix} [1]^t_{\dif}
    \oplus q^{p+6} V_{p-2(i+1)}^q  [2]^t_{\dif} \right) [i-n]^t_{\dif}
    \ .
    \end{align*}
\item[(ii)] If $n$ is even it is:
    \begin{align*}
    & q^{-3n-2}[-n]^t_{\dif} \left( q^{p+2} V_{p-2}^q  [1]^t_{\dif} \oplus q^{p+4} V_{p-4}^q [2]^t_{\dif} \right)
    \bigoplus \\ 
    &\bigoplus_{i \in \{2,4,\ldots,n-2 \}} q^{-3n-2} \left(
    q^{p} V_{p-2(i-1)}^q 
    \oplus  \begin{pmatrix} q^{p+2} V_{p-2i}^q  \\ \oplus  \\ 
    q^{p+2} V_{p-2i-2}^q \end{pmatrix} [1]^t_{\dif}
    \oplus q^{p+6} V_{p-2i-2}^q [2]^t_{\dif} \right) [i-n]^t_{\dif} \\
     &\bigoplus 
q^{-3n-2}\left(\begin{matrix}
 q^{2p} V_{p-(n-1)}^q \otimes  V_{p-(n-1)}^q  \\    
    \oplus  \\
    \left( q^{2p+2} V_{p-n-1}^q \otimes V_{p-n+1}^q   \oplus 
    q^{2p+2} V_{p-n}^q \otimes V_{p-n-2}^q  \right)
     [1]^t_{\dif} \\
    \oplus \\
    q^{2p+4}V_{p-(n+2)}^q \otimes V_{p-n-2}^q [2]^t_{\dif}
    \end{matrix}
    \right) 
    \ .
    \end{align*}
\end{enumerate}
\end{cor}

\subsection{Example of the Jones invariant when \texorpdfstring{$q$}{q} is generic}
We will compute the Jones homology of a $(2,n)$ torus link when the quantum parameter $q$ is generic so we assume $\partial_q=0$. We denote this homology by
$\mH_{\bullet,\bullet}(T_{2,n})$. 
As mentioned earlier, this more classical homology was constructed from various perspectives in \cite{Cautisremarks}, \cite{QRS}, and \cite{RW} and its formulation is built into Definition \ref{pdgjonesdef}.  Setting $\dif_q=0$ and not $p$-extending in the $a$ or $t$ directions, allows for a doubly graded theory rather than the singly graded theory of Section \ref{sl2section}.  We thus use the grading shift conventions of Section \ref{sectriple}.

Recall the Koszul resolution of $R$ in
\eqref{k[x,y]resolved}.
%\begin{equation} \label{k[x,y]resolveddifferentials}
%\xymatrix{
%q^4R^{x_1+x_2} \otimes R^{x_1+x_2} dx_1 dx_2 
%\ar[rrr]^{\left(\begin{smallmatrix}  x_2 \otimes 1 -  1 \otimes x_2 \\ %-x_1 \otimes 1 +1 \otimes x_1 \end{smallmatrix}\right) \hspace{.2in}} 
%& & &  
%q^2 R^{x_1} \otimes R^{x_1} dx_1 \oplus
%q^2 R^{x_2} \otimes R^{x_2} dx_2
%\ar[rrrr]^{\hspace{.5in} \left( \begin{smallmatrix} x_1 \otimes 1 - 1 %\otimes x_1 , &  x_2 \otimes 1 -  1 \otimes x_2 \end{smallmatrix} \right) %}
%  & & & & R \otimes R
%}  
%\end{equation}
Then $\mHH_\bullet (R)$ with the induced Cautis differential $d_C$ is given by
\begin{equation} \label{HHRwithdiffs}
\begin{gathered}
\xymatrix{
& R & \\
q^4 R  \ar[ur]^{x_1^2} & & q^4 R  \ar[ul]_{x_2^2} \\
& q^8 R  \ar[ul]^{-x_2^2} \ar[ur]_{x_1^2} &
}    
\end{gathered}
\end{equation}
and $\mHH_\bullet(B)$ with $d_C$ is given by
\begin{equation} \label{HHBwithdiffs}
\begin{gathered}
\xymatrix{
& R  & \\
q^4 R  \ar[ur]^{x_1^2 +x_2^2} & & q^6 R  \ar[ul]_{x_2^2(x_2-x_1)}   \\
& q^{10} R  \ar[ul]^{x_2^2(x_1-x_2)} \ar[ur]_{x_1^2 + x_2^2} &
}
\end{gathered}
\ .
\end{equation}

When $n$ is even, the leftmost term in $T^{\otimes n}$ maps by zero into the rest of the complex, so we need to understand the homology of $\mHH_\bullet(B)$ by itself in this case. All of the maps in \eqref{HHBwithdiffs} are injective so homology is concentrated in $R$.
Thus we need to find a basis of $R/(x_1^2+x_2^2,x_2^2(x_2-x_1))$. %In the quotient, note that $x_1^3=x_1^2 x_1 = -x_1 x_2^2 = %x_2^3$.
%Since $x_2^3=-x_1 x_2^2$, we get
%$x_2^4=-x_1 x_2^3=-x_1^4$.
In the quotient, note that $x_1^3=x_1^2 x_1 = -x_1 x_2^2 = -x_2^3$.
Since $x_2^3=x_1 x_2^2$, we get
$x_2^4=x_1 x_2^3=-x_1^4$.
But also $x_2^4=x_2^2 x_2^2=x_1^4$.
Thus $x_1^4=x_2^4=0$.  Thus the homology is spanned by
\begin{equation*}
\{1, x_1, x_2, x_1^2, x_1 x_2, x_1^2 x_2 \}.    
\end{equation*}

In the rest of the complex for $T^{\otimes n}$, there are two types of maps we need to analyze.
First we study $\mHH_\bullet (br) \colon \mHH_\bullet (B) \rightarrow \mHH_\bullet(R)$.
\begin{equation} \label{HHBRwithdiffs}
\begin{gathered}
\xymatrix{
& R  \ar@/^/[rrr]^{1 \mapsto 1} & & & R & \\
q^4 R  \ar[ur]^{x_1^2 +x_2^2} & & q^6 R  \ar[ul]^{x_2^2(x_2-x_1)} 
\ar@/^/[r]^{\begin{pmatrix}
1 & 0 \\
1 & x_2-x_1
\end{pmatrix}}
& q^4 R  \ar[ur]^{x_1^2}& & q^4 R  \ar[ul]_{x_2^2}  \\
& q^{10} R  \ar[ul]^{x_2^2(x_1-x_2)} \ar[ur]^{x_1^2 + x_2^2} \ar@/_/[rrr]^{1 \mapsto (x_2-x_1)} & & & q^8 R  \ar[ul]_{-x_2^2} \ar[ur]_{x_1^2}
}    
\end{gathered}
\end{equation}
%\begin{equation} \label{HHBRwithdiffs}
%\JS{with gammas}
%\xymatrix{
%& R \langle 1 \otimes 1 \rangle \ar@/^/[rrr]^{1 \otimes 1 \mapsto 1} & %& & R & \\
%q^2 R \langle (1 \otimes 1)dx_1 + (1 \otimes 1)dx_2 \rangle %\ar[ur]^{x_1^2 +x_2^2} & & q^2 R \langle \Gamma dx_2 \rangle %\ar[ul]^{x_2^2(x_2-x_1)} 
%\ar@/^/[r]^{\begin{pmatrix}
%1 & 0 \\
%1 & x_2-x_1
%\end{pmatrix}}
%& q^2 R dx_1 \ar[ur]^{x_1^2}& & q^2 R dx_2 \ar[ul]_{x_2^2}  \\
%& q^4 R \langle \Gamma dx_1 dx_2 \rangle \ar[ul]_{x_2^2(x_1-x_2)} %\ar[ur]^{x_1^2 + x_2^2} \ar@/_/[rrr]^{\Gamma dx_1 dx_2 \mapsto %(x_2-x_1) dx_1 dx_2} & & & q^4 R dx_1 dx_2 \ar[ul]_{-x_2^2} %\ar[ur]^{x_1^2}
%}    
%\end{equation}
While it is not very difficult to compute the total homology of this bicomplex, we use a fact from the proof of \cite[Theorem 6.2]{RW}.  According to this trick, we may calculate homology with respect to $d_C$ first and then with respect to the topological differential $d_t$.
Thus the homology of the bicomplex with respect to the total differential is spanned by
$\{t x_1^2, t x_1^2 x_2  \}$.
%$\{q^2 1 dx_2, q^2 x_1 dx_2 \}$.

Next we analyze $\mHH_\bullet (p_{2i+1}) \colon \mHH_\bullet(q^{4i} B) \longrightarrow \mHH_\bullet (q^{4i-2}B)$.
\begin{equation} \label{HHBBwithdiffs}
\begin{gathered}
\xymatrix{
& q^{4i} R  \ar@/^/[rrrr]^{2(x_2-x_1)} & & & & q^{4i-2}R & \\
q^{4i+4} R  \ar[ur]^{x_1^2 +x_2^2} & & q^{4i+6} R  \ar[ul]^{x_2^2(x_2-x_1)} 
\ar@/^/[rr]^{\bigl( \begin{smallmatrix}
2(x_2-x_1) & 0 \\
0 & 2(x_2-x_1)
\end{smallmatrix} \bigr)}
& & q^{4i+2} R  \ar[ur]_{x_1^2+x_2^2}& & q^{4i+4} R  \ar[ul]_{x_2^2(x_2-x_1)}  \\
& q^{4i+10} R  \ar[ul]^{x_2^2(x_1-x_2)} \ar[ur]^{x_1^2 + x_2^2} \ar@/_/[rrrr]^{2(x_2-x_1)} & & & & q^{4i+8} R  \ar[ul]^{x_2^2(x_1-x_2)} \ar[ur]_{x_1^2+x_2^2}
}    
\end{gathered}
\end{equation}
%\begin{equation} \label{HHBBwithdiffs}
%\JS{With gammas}
%\xymatrix{
%& R \langle 1 \otimes 1 \rangle & \\
%q^2 R \langle (1 \otimes 1)dx_1 + (1 \otimes 1)dx_2 \rangle %\ar[ur]^{x_1^2 +x_2^2} & & q^2 R \langle \Gamma dx_2 \rangle %\ar[ul]_{x_2^2(x_2-x_1)}   \\
%& q^4 R \langle \Gamma dx_1 dx_2 \rangle \ar[ul]_{x_2^2(x_1-x_2)} %\ar[ur]^{x_1^2 + x_2^2} & \\
%%%%%%%
%& R \langle 1 \otimes 1 \rangle \ar@{-->}@/^3pc/[uuu]_{2(x_2-x_1)} & \\
%q^2 R \langle (1 \otimes 1)dx_1 + (1 \otimes 1)dx_2 \rangle %\ar[ur]^{x_1^2 +x_2^2} \ar@{-->}@/^1pc/[uuu]_{2(x_2-x_1)} & & q^2 R %\langle \Gamma dx_2 \rangle \ar[ul]_{x_2^2(x_2-x_1)} %\ar@{-->}@/_3pc/[uuu]_{2(x_2-x_1)}  \\
%& q^4 R \langle \Gamma dx_1 dx_2 \rangle \ar[ul]_{x_2^2(x_1-x_2)} %\ar[ur]_{x_1^2 + x_2^2}
% \ar@{-->}@/_3pc/[uuu]^(.4){2(x_2-x_1)}
%&
%}    
%\end{equation}
Again using the proof of \cite[Theorem 6.2]{RW}, we compute the homology with respect to $d_C$ and then with respect to $d_t$ to obtain the total homology of this bicomplex is spanned by
%Thus quotienting the top complex in \eqref{HHBBwithdiffs} by %$2(x_2-x_1)$ produces
%\begin{equation} \label{HHBwithdiffsmod}
%\xymatrix{
%& R \langle 1 \otimes 1 \rangle & \\
%q^2 R \langle (1 \otimes 1)dx_1 + (1 \otimes 1)dx_2 \rangle %\ar[ur]^{x_1^2 +x_2^2} & & q^2 R \langle \Gamma dx_2 \rangle %\ar[ul]_{0}   \\
%& q^4 R \langle \Gamma dx_1 dx_2 \rangle \ar[ul]_{0} \ar[ur]^{x_1^2 + %x_2^2} &
%}    
%\end{equation}
%The complex \eqref{HHBwithdiffsmod} has homology spanned by
\begin{equation*}
\{ 
q^{4i} t^{2i+1} x_1^2, q^{4i} t^{2i+1} x_1^2 x_2, 
q^{4i-2}  t^{2i} 1, q^{4i-2}t^{2i} x_1
\}.    
\end{equation*}
%\begin{equation*}
%\{ 
%1, x_2, \Gamma dx_2, x_1 \Gamma dx_2
%\}.    
%\end{equation*}
We now assemble all of this information together to get the homology of the $(2,n)$ torus link.
Recall that the complex used for $T^{\otimes n}$ comes with a shift of
$(at)^{\frac{-n}{2}} q^{-2n}$ and the Hochschild homology functor comes with a shift of $a^{-1} t$.
Thus there is an overall shift of $a^{\frac{-n-2}{2}} t^{\frac{-n+2}{2}} q^{-2n}$.  Specializing $a=q^{2}t$ yields an overall shift of $q^{-3n-2} t^{-n} $.  Thus we get the following homology in terms of Poincar\'{e} series. 
\begin{itemize}
    \item If $n$ is odd, then the bigraded Poincar\'{e} series of $  \mH^{}_{\bullet,\bullet}(T_{2,n})$ is equal to 
   
    \begin{equation*}
    q^{-3n-2} t^{-n} \cdot 
    \Big(
 (1+q^2)q^4t(1+q^4t^2+q^8t^4+\cdots+q^{2n-2} t^{n-1})
    +(1+q^2)q^2 t^2(1+q^4 t^2 + q^8 t^4+\cdots+q^{2n-6} t^{n-3})\Big)
    .
    \end{equation*}
    \item If $n$ is even, then the bigraded Poincar\'{e} series of $  \mH^{}_{\bullet,\bullet}(T_{2,n})$ is equal to  
    \begin{align*}
  &  q^{-3n-2} t^{-n} \cdot \Big( (1+q^2)q^4t(1+q^4t^2+\cdots+q^{2n-4} t^{n-2})
    +(1+q^2)q^2 t^2(1+q^4 t^2 +\cdots+q^{2n-8} t^{n-4}) \\
    &+t^nq^{2n-2} (1+2q^2+2q^4+q^6) \Big) .
    \end{align*}
\end{itemize}
\begin{rem}
This was also computed (for $n=2,3$) in \cite{QRS} and \cite{RW}.
It is interesting to note that this is different from the Khovanov homology of the Hopf link and trefoil.
\end{rem}

\subsection{Example of the Jones invariant when \texorpdfstring{$q$}{q}-prime root of unity}
To compute the $p$-DG Jones invariant, we will utilize the following auxiliary tool.

\begin{prop} \label{filtertrickprop}
Let 
\[ 
M_\bullet = \left(
\cdots \xrightarrow{\dif_t} M_{i+1} \xrightarrow{ \dif_t } M_i \xrightarrow{ \dif_t} M_{i-1} \xrightarrow{\dif_t} \cdots
\right)
\] be a contractible $p$-complex of $H_q=\Bbbk[\dif_q]/(\dif_q^p)$-modules. Then the totalized complex $(\mc{T}(M_\bullet),\dif_T=\dif_t+\dif_q)$ is acyclic.
\end{prop}
\begin{proof}
  Since $M_\bullet$ is contractible, there is an $H_q$-linear map $\sigma:M_\bullet \lra M_{\bullet+1}$  such that $[\sigma, \dif_t]=\mathrm{Id}_M$ by Lemma \ref{lemma-acylicity-commutator-relation}. Thus we have 
  \[
  [\sigma, \dif_T]=[\sigma,\dif_t+\dif_q]=[\sigma,\dif_t]+[\sigma,\dif_q]=\mathrm{Id}_M.
  \]
The result follows again from Lemma \ref{lemma-acylicity-commutator-relation}.
\end{proof}

We will be applying Proposition \ref{filtertrickprop} in the following situation. Suppose $N_\bullet$ is a $p$-complex of $H_q$-modules whose boundary maps preserve the $H_q$-module structure. Further, let $M_\bullet$ be a sub $p$-complex that is closed under the $H_q$-action, and there is a map $\sigma$ on $M_\bullet$ as in Proposition \ref{filtertrickprop} that preserves the $H_q$-module structure. Then, when totalizing the $p$-complexes, we have $\mc{T}(M_\bullet) \subset \mc{T}(N_\bullet)$ and the natural projection map
\[
\mc{T}(N_\bullet) \lra \mc{T}(N_\bullet)/\mc{T}(M_\bullet)
\]
is a quasi-isomorphism. Similarly, if $M_\bullet$ is instead a quotient complex of $N_\bullet$ that satisfies the condition of Proposition \ref{filtertrickprop}, and $K_\bullet$ is the kernel of the natural projection map 
$$ 0 \lra K_\bullet \lra N_\bullet \lra M_\bullet \lra 0, $$
then the inclusion map of totalized complexes $\mc{T}(K_\bullet) \lra \mc{T}(N_\bullet)$ is a quasi-isomorphism.

We modify the the calculation of the homology in the previous section of the $(2,n)$ torus link to account for the differential $\dif_q$.
Recall that in this singly graded theory that $a=tq^{2}$ and $t=[1]^q_\dif$. Consequently $[1]^a_{\dif}=q^2[1]^t_{\dif}=q^2[1]^q_{\dif} $.

First we study $\pHH_\bullet(br) \colon \pHH_\bullet(B)[1]^q_\dif \longrightarrow \pHH_\bullet(R)$, 
%\QY{Try to add internal differential and twists in the left square} \JS{I did it but it spills out a little.  I commented out the attempt.  I suggest we keep it the way it is because we also reference that connection matrix a few times.}
%\begin{equation} \label{pHHBRwithdiffs}
%\xymatrix{
%& R[1]^t_\dif  \ar@/^/[rrr]^{1 \mapsto 1} & & & R & \\
%q^4 R^{2x_1}[1]^t_\dif [1]^a_\dif  \ar[ur]^{x_1^2 +x_2^2} %\ar[rr]^{2} & & q^6 R^{x_1+3x_2}[1]^t_\dif [1]^a_\dif %\ar[ul]^{x_2^2(x_2-x_1)} 
%\ar@/^/[r]^{\begin{pmatrix}
%1 & 0 \\
%1 & x_2-x_1
%\end{pmatrix}}
%& q^4 R^{2x_1}[1]^a_\dif  \ar[ur]^{x_1^2}& & q^4 %R^{2x_2}[1]^a_\dif  \ar[ul]_{x_2^2}  \\
%& q^{10} R^{3e_1}[1]^t_\dif [2]^a_\dif  %\ar[ul]^{x_2^2(x_1-x_2)} \ar[ur]^{x_1^2 + x_2^2} %\ar@/_/[rrr]^{1 \mapsto (x_2-x_1)} & & & q^8 %R^{2e_1}[2]^a_\dif  \ar[ul]_{-x_2^2} \ar[ur]_{x_1^2}
%}    
%\end{equation}
\begin{equation} \label{pHHBRwithdiffs}
\xymatrix{
& R[1]^q_\dif  \ar@/^/[rrr]^{1 \mapsto 1} & & & R & \\
q^4 R[1]^q_\dif [1]^q_\dif  \ar[ur]^{x_1^2 +x_2^2} & & q^6 R[1]^q_\dif [1]^q_\dif \ar[ul]^{x_2^2(x_2-x_1)} 
\ar@/^/[r]^{\begin{pmatrix}
1 & 0 \\
1 & x_2-x_1
\end{pmatrix}}
& q^4 R^{2x_1}[1]^q_\dif  \ar[ur]^{x_1^2}& & q^4 R^{2x_2}[1]^q_\dif  \ar[ul]_{x_2^2}  \\
& q^{10} R^{3e_1}[1]^q_\dif [2]^q_\dif  \ar[ul]^{x_2^2(x_1-x_2)} \ar[ur]^{x_1^2 + x_2^2} \ar@/_/[rrr]^{1 \mapsto (x_2-x_1)} & & & q^8 R^{2e_1}[2]^q_\dif  \ar[ul]_{-x_2^2} \ar[ur]_{x_1^2}
}    
\end{equation}
where the object $q^4 R[1]^q_\dif [1]^q_\dif \oplus q^6 R[1]^q_\dif [1]^q_\dif  $ in the left square is twisted by the matrix 
\begin{equation}
\begin{pmatrix} \label{connmatrix}
2x_1 & 0 \\
2 & x_1+3x_2
\end{pmatrix}
\ .
\end{equation}
Filtering the total complex \eqref{pHHBRwithdiffs} and applying Proposition \ref{filtertrickprop}, we obtain that the total $p$-complex is quasi-isomorphic to
\begin{equation}\label{eqn-Jones-reduced-B-to-R}
\xymatrix{
\Bbbk \langle 1,x_1, x_2, x_1 x_2, x_1^2, x_1^2 x_2 \rangle [1]^q_\dif
\ar[r]^{\hspace{.3in} 1}
&
\Bbbk \langle 1, x_1, x_2, x_1 x_2 \rangle
}    \ .
\end{equation}
Let us illustrate how this is obtained.
For instance,  the $p$-complex 
\[
q^8 R^{2e_1}[2]^q_\dif \xrightarrow{\phi=(-x_2^2,x_1^2)} \mathrm{Im}(\phi) \subset \left(  q^4 R^{2x_1}[1]^q_\dif\oplus q^4R^{2x_2}[1]^q_\dif\right)
\]
is a quotient of the total $p$-complex of the rightmost square. The map $\phi$ is an $H_q$-intertwining isomorphism onto its image because of the $H_q$-module twists imposed on the modules. Note that this $p$-complex (ignoring $q$-grading shifts)
\[
R^{2e_1}\xrightarrow[\cong]{\phi=(-x_2^2,x_1^2)} \mathrm{Im}(\phi) =  \mathrm{Im}(\phi) = \cdots =  \mathrm{Im}(\phi),
\]
where  $ \mathrm{Im}(\phi)$ is repeated $p-1$ times, is a contractible $p$-complex of $H_q$-modules.
Proposition \ref{filtertrickprop} then applies to this quotient complex, and shows that it contributes nothing to the total slash homology.

The $p$-complex in \eqref{eqn-Jones-reduced-B-to-R}, in turn, is quasi-isomorphic to 
\begin{equation*}
\Bbbk \langle x_1^2, x_1^2 x_2 \rangle [1]^q_\dif.
\end{equation*}
This is quasi-isomorphic to $q^5 V_1 [1]^q_\dif$. 

Once again when $n$ is even, the leftmost term in $T^{\otimes n}$ maps by zero into the rest of the complex so we have to understand the total homology of $\pHH_{\bullet}(q^{2(n-1)}B^{(n-1)e_1}[n]^q_\dif)$.
Filtering
\begin{equation} \label{HHBalonewithdiffs}
\begin{gathered}
\xymatrix{
& q^{2(n-1)}R^{(n-1)e_1}[n]^q_\dif  & \\
q^{2(n+1)} R^{(n-1)e_1}[n+1]^q_\dif  \ar[ur]^{x_1^2 +x_2^2} & & q^{2(n+2)} R^{(n-1)e_1}[n+1]^q_\dif  \ar[ul]_{x_2^2(x_2-x_1)}   \\
& q^{2(n+4)} R^{(n+2)e_1}[n+2]^q_\dif  \ar[ul]^{x_2^2(x_1-x_2)} \ar[ur]_{x_1^2 + x_2^2} &
}    
\end{gathered}
\end{equation}
where the middle terms
$q^{2(n+1)} R^{(n-1)e_1}[n]^q_\dif [1]^q_\dif \oplus q^{2(n+2)} R^{(n-1)e_1}[n]^q_\dif [1]^q_\dif$ are further twisted by the matrix \eqref{connmatrix},
%\begin{equation*}
%\begin{pmatrix} \label{connmatrix}
%2x_1 & 0 \\
%2 & x_1+3x_2
%\end{pmatrix}
%\end{equation*}
yields that \eqref{HHBalonewithdiffs} is quasi-isomorphic to
\[
q^{2(n-1)}\Bbbk \langle 1,x_1, x_2, x_1 x_2, x_1^2, x_1^2 x_2 \rangle [n]^q_\dif
\]
with a differential inherited from the polynomial algebra and twisted by $(n-1)e_1$.  Explicitly, the differential acts on the basis by
\begin{equation}
\begin{gathered}
\xymatrix{
& 1 \ar[dl]_{n-1} \ar[dr]^{n-1} & \\
x_1 \ar[d]_{n} \ar[drr]^<<<{n-1} & & x_2 \ar[d]^{n-1} \ar[dll]^>>>{-n} \\
x_1^2 \ar[dr]^{2n} & & x_1 x_2 \\
& x_1^2 x_2 &
}  
\end{gathered}
\ .
\end{equation}
This is isomorphic to the direct sum of $p$-complexes
\[
\xymatrix{
1 \ar[d]_{n-1}& & x_1-x_2 \ar[d]^{2n} \\
x_1+x_2 \ar[d]_{2(n-1)}& \bigoplus & x_1^2 \ar[d]^{2n} \\
x_1 x_2 & & x_1^2 x_2 \\
}
\ .
\]
Thus the total homology of $\pHH_{\bullet}(q^{2(n-1)}B^{(n-1)e_1}[n]^q_\dif)$ is isomorphic to  the following $p$-complex dependent of the characteristic of the ground field
\begin{equation} \label{defYn}
Y_{\frac{n}{2}}:=
\begin{cases}
 q^{2n}V_2[n]^q_\dif \oplus q^{2n+2} V_2 [n]^q_\dif &
\text{ if } p \nmid n-1, n \\
q^{2(n-1)}(V_0 \oplus q^2 V_0 \oplus q^4 V_0) [n]^q_\dif 
\oplus q^{2n+2} V_2 [n]^q_\dif  & 
\text{ if } p \mid n-1, p \nmid n \\
q^{2n}V_2 [n]^q_\dif 
\oplus q^{2n} (V_0 \oplus q^2 V_0 \oplus q^4 V_0) [n]^q_\dif  & 
\text{ if } p \nmid n-1, p \mid n \\
%q^{2(n-1)} (V_0 \oplus q^2 V_0 \oplus q^4 V_0) [n]^t_\dif 
%\oplus q^{2n} (V_0 \oplus q^2 V_0 \oplus q^4 V_0) [n]^t_\dif % & 
%\text{ if } p \mid n-1, n \\
\end{cases}   \ .
\end{equation}

Finally we analyze $\pHH_\bullet(p_{2i+1}) \colon \pHH_\bullet (q^{4i} B^{2ie_1}[2i+1]^q_\dif) {\lra} \mHH_\bullet(q^{4i-2}B^{(2i-1)e_1}[2i]^q_\dif)$ where
\begin{equation} \label{pHHBBwithdiffs1}
 \pHH_\bullet (q^{4i} B^{2ie_1}[2i+1]^q_\dif)~=~
 \begin{gathered}
\xymatrix@C=0.75em{
& q^{4i} R^{2ie_1}[2i+1]^q_\dif   &   \\
q^{4i+4} R[2i+2]^q_\dif   \ar[ur]^{x_1^2 +x_2^2} & & q^{4i+6} R[2i+2]^q_\dif \ar[ul]_{x_2^2(x_2-x_1)}
 \\
& q^{4i+10} R^{(2i+3)e_1}[2i+3]^q_\dif  \ar[ul]^{x_2^2(x_1-x_2)} \ar[ur]_{x_1^2 + x_2^2}  & 
}    
\end{gathered} \ ,
\end{equation} 
and
\begin{equation} \label{pHHBBwithdiffs2}
\pHH_\bullet (q^{4i-2} B^{(2i-1)e_1}[2i]^q_\dif)~=~
\begin{gathered}
\xymatrix@C=1.5em{
& q^{4i-2}R^{(2i-1)e_1}[2i]^q_\dif & \\
q^{4i+2} R^{}[2i+1]^q_\dif  \ar[ur]^{x_1^2+x_2^2}& & q^{4i+4} R^{}[2i+1]^q_\dif \ar[ul]_{x_2^2(x_2-x_1)}  \\
& q^{4i+8} R^{(2i+2)e_1}[2i+2]^q_\dif  \ar[ul]^{x_2^2(x_2-x_1)} \ar[ur]_{x_1^2+x_2^2} &
}      
\end{gathered}
\end{equation}
%\begin{equation} \label{pHHBBwithdiffs}
%\xymatrix{
%& q^{4i} R^{2ie_1}[2i+1]^t_\dif  %\ar@/^/[rrrr]^{2(x_2-x_1)} & & & & %q^{4i-2}R^{(2i-1)e_1}[2i]^t_\dif & \\
%q^{4i+4} R[2i+1]^t_\dif [1]^a_\dif  \ar[ur]^{x_1^2 %+x_2^2} & & q^{4i+6} R[2i+1]^t_\dif [1]^a_\dif %\ar[ul]^{x_2^2(x_2-x_1)} 
%\ar@/^/[rr]^{\bigl( \begin{smallmatrix}
%2(x_2-x_1) & 0 \\
%0 & 2(x_2-x_1)
%\end{smallmatrix} \bigr)}
%&  & q^{4i+2} R^{}[2i]^t_\dif [1]^a_\dif  %\ar[ur]^{x_1^2+x_2^2}& & q^{4i+4} %R^{}[2i]^t_\dif[1]^a_\dif  \ar[ul]_{x_2^2(x_2-x_1)}  \\
%& q^{4i+10} R^{(2i+3)e_1}[2i+1]^t_\dif [2]^a_\dif  %\ar[ul]^{x_2^2(x_1-x_2)} \ar[ur]^{x_1^2 + x_2^2} %\ar@/_/[rrrr]^{2(x_2-x_1)} & &  & & q^{4i+8} %R^{(2i+2)e_1}[2i]^t_\dif[2]^a_\dif  %\ar[ul]_{x_2^2(x_2-x_1)} \ar[ur]_{x_1^2+x_2^2}
%}    
%\end{equation}
where the differentials for both objects in the middle horizontal rows of \eqref{pHHBBwithdiffs1} and \eqref{pHHBBwithdiffs2} are twisted by \eqref{connmatrix} and
$\pHH_\bullet(p_{2i+1})=2(x_2-x_1)$ (diagonal multiplication by $2(x_2-x_1)$).
Filtering this total complex yields the total complex
\begin{equation}
\xymatrix{
q^{4i} \Bbbk \langle 1, x_1, x_2, x_1^2, x_1 x_2, x_1^2 x_2 \rangle [2i+1]^q_\dif
\ar[rr]^{2(x_2-x_1)}
& &
q^{4i-2} \Bbbk \langle 1, x_1, x_2, x_1^2, x_1 x_2, x_1^2 x_2 \rangle [2i]^q_\dif
}    
\ .
\end{equation}
This is quasi-isomorphic to 
\begin{equation}
q^{4i} \Bbbk \langle x_1^2, x_1^2 x_2 \rangle [2i+1]^q_\dif
\bigoplus 
q^{4i-2} \Bbbk \langle 1, x_1 \rangle [2i]^q_\dif
\ ,
\end{equation}
where the differential on the basis elements is given by
\[
\begin{gathered}
\xymatrix{
x_1^2 \ar[dd]_{4i+2}& & 1 \ar[dd]^{4i-2} \\
& \bigoplus & \\
x_1^2x_2 &  & x_1  \\
}
\end{gathered}
\ .
\]
Thus the total homology is isomorphic to the $p$-complex
\begin{equation} \label{defXi}
X_i:=
\begin{cases}
 q^{4i+5}V_1[2i+1]^q_\dif \oplus q^{4i-1} V_1 [2i]^q_\dif &
\text{ if } p \nmid 2i+1, 2i-1 \\
q^{4i+4}(V_0 \oplus q^2 V_0)[2i+1]^q_\dif \oplus q^{4i-1} V_1 [2i]^q_\dif   & 
\text{ if } p \mid 2i+1, p \nmid 2i-1 \\
q^{4i+5}V_1 [2i+1]^q_\dif 
\oplus q^{4i-2} (V_0 \oplus q^2 V_0) [2i]^q_\dif  & 
\text{ if } p \nmid 2i+1, p \mid 2i-1 \\
%q^{2(n-1)} (V_0 \oplus q^2 V_0) [2i+1]^t_\dif 
%\oplus q^{2n} (V_0 \oplus q^2 V_0) [2i]^t_\dif  & 
%\text{ if } p \mid 2i+1, 2i-1 \\
\end{cases}   \ .
\end{equation}

All of these computations together with an overall shift of 
$q^{-3n-2} [-n]^q_{\dif} $ yields the slash homology of the $(2,n)$ torus link.

\begin{equation} \label{pjonestorus}
\pH(T_{2,n}) \cong
\begin{cases}
q^{-3n-2} [-n]^q_{\dif} \left(
q^5 V_1 [1]^q_{\dif} \oplus \bigoplus_{i=1}^{\frac{n-1}{2}} X_i
\right) & \text{ if } 2 \nmid n \\
q^{-3n-2} [-n]^q_{\dif} \left(
q^5 V_1 [1]^q_{\dif} \oplus \bigoplus_{i=1}^{\frac{n-2}{2}} X_i
\oplus Y_{\frac{n}{2}}  \right) & \text{ if } 2 \mid n
\end{cases}
\end{equation}
where $X_i$ is the $p$-complex in \eqref{defXi} and
$Y_{\frac{n}{2}}$ is the $p$-complex in \eqref{defYn}.
It is interesting to note that the homology depends upon the prime $p$.
\addcontentsline{toc}{section}{References}

% ====================================================================
% REFERENCES

\bibliographystyle{alpha}
\bibliography{qy-bib}

%
% ====================================================================

\noindent Y.~Q.: { \sl \small Department of Mathematics, University of Virginia, Charlottesville, VA 22904, USA} \newline \noindent {\tt \small email: yq2dw@virginia.edu}

\vspace{0.1in}

\noindent J.~S.:  {\sl \small Department of Mathematics, CUNY Medgar Evers, Brooklyn, NY, 11225, USA}\newline \noindent {\tt \small email: jsussan@mec.cuny.edu \newline 
\sl \small Mathematics Program, The Graduate Center, CUNY, New York, NY, 10016, USA}\newline \noindent {\tt \small email: jsussan@gc.cuny.edu}

% ==============================================================================
%
\end{document}